\documentclass{article}

\usepackage[utf8]{inputenc} 
\usepackage[T1]{fontenc}    
\usepackage{url}            
\usepackage{booktabs}       
\usepackage{amsfonts}       
\usepackage{microtype}      

\usepackage{nicefrac} 
\usepackage{subcaption}
\usepackage{amsmath}
\usepackage{amsthm}
\usepackage{amssymb}
\usepackage{thmtools,thm-restate}
\usepackage{thm-restate}
\usepackage{paralist}
\usepackage{comment}
\usepackage{dsfont}

\usepackage{bbm}
\usepackage{float}
\usepackage{balance}
\usepackage{stmaryrd}

\usepackage{multicol}
\usepackage[export]{adjustbox}
\usepackage{multirow}
\usepackage{tabularx}
\usepackage{mathtools}
\usepackage{mdframed}

\usepackage{enumitem}
\setitemize{noitemsep,topsep=0pt,parsep=0pt,partopsep=0pt}

\usepackage[ruled, linesnumbered]{algorithm2e} %
\usepackage{xcolor} %
\usepackage{tikz}
\usetikzlibrary{fit,calc}

\colorlet{pink}{red!40}
\colorlet{lightblue}{blue!30}
\colorlet{lightgreen}{green!30}

\DontPrintSemicolon

\usepackage{jmlr2e}
\usepackage[margin=1in]{geometry}

\usepackage{natbib}
\setcitestyle{authoryear,round,citesep={;},aysep={,},yysep={;}}
\renewcommand{\cite}[1]{\citep{#1}}

\usepackage[vvarbb]{newtxmath}

\def\Afwt{{Q}}
\def\Rfwt{{R}}
\def\fwt{{S}}

\def\fw{{\textnormal{\texttt{FW}}}\xspace}

\def\lmo{{\textnormal{\texttt{LMO}}}\xspace}

\newcommand{\subopt}{{\mathsf{subopt}}}
\newcommand{\gap}{{\mathsf{gap}}}
\newcommand{\primaldual}{{\mathsf{primaldual}}}

\def\aa{{\mathbf{a}}}
\def\bb{{\mathbf{b}}}
\def\cc{{\mathbf{c}}}

\def\uu{{\mathbf{u}}}
\def\vv{{\mathbf{v}}}

\def\xx{{\mathbf{x}}}
\def\yy{{\mathbf{y}}}
\def\zz{{\mathbf{z}}}

\newcommand{\N}{\mathbb{N}}

\newcommand{\R}{\mathbb{R}}

\newcommand\cC{{\ensuremath{\mathcal{C}}}\xspace}

\newcommand\cI{{\ensuremath{\mathcal{I}}}\xspace}

\newcommand\cO{{\ensuremath{\mathcal{O}}}\xspace}

\newcommand\cU{{\ensuremath{\mathcal{U}}}\xspace}
\newcommand\cV{{\ensuremath{\mathcal{V}}}\xspace}

\DeclareMathOperator{\relativeinterior}{rel.int}
\DeclareMathOperator{\relativeboundary}{rel.bnd}

\DeclareMathOperator{\nuc}{nuc}
\DeclareMathOperator{\trace}{tr}
\DeclareMathOperator{\vertices}{vert}

\DeclareMathOperator{\faces}{faces}

\DeclareMathOperator{\argmax}{argmax}
\DeclareMathOperator{\argmin}{argmin}

\def\fw{{\textnormal{\texttt{FW}}}\xspace}

\def\lmo{{\textnormal{\texttt{LMO}}}\xspace}

\newcommand{\rb}[1]{( #1 )}
\newcommand{\srb}[1]{\left( #1 \right)}

\newcommand{\abs}[1]{\lvert #1 \rvert}

\theoremstyle{plain} \numberwithin{equation}{section}
\newtheorem{theorem}{Theorem}[section]
\numberwithin{theorem}{section}
\newtheorem{lemma}[theorem]{Lemma}
\newtheorem{corollary}[theorem]{Corollary}
\newtheorem{proposition}[theorem]{Proposition}

\newtheorem{fact}[theorem]{Fact}

\theoremstyle{definition}
\newtheorem{definition}[theorem]{Definition}

\newtheorem{example}[theorem]{Example}

\theoremstyle{plain}

\newcommand{\hrulealg}[0]{\vspace{1mm} \hrule \vspace{1mm}}

\newcommand{\ip}[2]{\left\langle #1 , #2 \right\rangle}    

\newcommand{\change}[1]{{\color{blue}#1}}
\usepackage{graphicx,wrapfig,lipsum}
\usepackage[
singlelinecheck=false 
]{caption}
\usepackage{makecell}
\usepackage{todonotes}

\begin{document}

\title{Accelerated Affine-Invariant Convergence Rates of the Frank-Wolfe Algorithm with Open-Loop Step-Sizes}

 \author{\name Elias Wirth \email          \href{mailto:wirth@math.tu-berlin.de}{wirth.elias.samuel@gmail.com}\\
       \addr Institute of Mathematics\\
        Berlin Institute of Technology \\
        Strasse des 17. Juni 135,  Berlin,  Germany
        \AND
       \name Javier Pe\~na \email \href{mailto:jfp@andrew.cmu.edu}{jfp@andrew.cmu.edu} \\
       \addr Tepper School of Business\\
		Carnegie Mellon University\\
       Pittsburgh, PA, United States
       \AND
      \name Sebastian Pokutta \email \href{mailto:pokutta@zib.de}{pokutta@zib.de} \\
      \addr Institute of Mathematics \& AI in Society, Science, and Technology\\
      Berlin Institute of Technology \& Zuse Institute Berlin\\
      Strasse des 17. Juni 135, Berlin, Germany}

\maketitle

\begin{abstract}
Recent papers have shown that the Frank-Wolfe algorithm (\fw) with open-loop step-sizes exhibits rates of convergence faster than the iconic $\cO(t^{-1})$ rate.  In particular, when the minimizer of a strongly convex  function over a polytope lies on the boundary of the polytope, the \fw{} algorithm with open-loop step-sizes $\eta_t = \frac{\ell}{t+\ell}$ for $\ell \in \N_{\geq 2}$ has accelerated convergence $\cO(t^{-2})$ in contrast to the rate 
$\Omega(t^{-1-\epsilon})$ attainable with more complex line-search or short-step step-sizes.
Given the relevance of this scenario in data science problems, research has grown to explore the settings enabling acceleration in open-loop \fw{}. However, despite \fw{}'s well-known affine invariance, existing acceleration results for open-loop \fw{}  are affine-dependent.  This paper remedies this gap in the literature, by merging two recent research trajectories: affine invariance~\cite{pena2023affine} and open-loop step-sizes~\cite{wirth2023acceleration}. In particular, we extend all known non-affine-invariant convergence rates for \fw{} with open-loop step-sizes to affine-invariant results.
\end{abstract}

\section{Introduction}

We consider constrained convex optimization problems of the form
\begin{equation}\label{eq:opt}\tag{OPT}
\min_{\xx\in\cC}f(\xx),
\end{equation}
where $\cC\subseteq\R^n$ is a compact convex set and $f\colon \cC \to \R$ is a smooth convex function.
When projection onto $\cC$ is unavailable or computationally expensive, the problem \eqref{eq:opt} can be tackled via the \emph{Frank-Wolfe algorithm} (\fw{}) \cite{frank1956algorithm}, alternatively known as the \emph{conditional gradients algorithm} \cite{levitin1966constrained}. 
The \fw{} algorithm requires no projections, it relies solely on first-order access to the function $f$ and a \emph{linear minimization oracle} (\lmo{}) for $\cC$. The \lmo{} returns a point in $\argmin_{\xx\in\cC}\langle\cc,\xx\rangle$ given $\cc\in\R^n$. 
During each iteration, the \lmo{} is invoked by \fw{} to obtain a new vertex $\vv_t\in\cV$, where $\cV:=\vertices(\cC)$ denotes the set of vertices of $\cC$. Subsequently, \fw{} updates the current iterate $\xx_t$, leading to $\xx_{t+1}= \xx_t + \eta_t (\vv_t - \xx_t)$, where $\eta_t\in[0,1]$ is a specific step-size rule. Possible step-size rules include line-search
$\eta_t \in \argmin_{\eta \in [0, 1]} f(\xx_t + \eta (\vv_t - \xx_t))$,
short-step $\eta_t = \min\left\{1,\frac{\langle \nabla f(\xx_t), \xx_t -\vv_t \rangle}{L\|\xx_t -\vv_t\|_2^2}\right\}$, adaptive \cite{pedregosa2018step}, and open-loop $\eta_t = \frac{\ell}{t+\ell}$ for $\ell\in\N_{\geq 1}$ \cite{dunn1978conditional, wirth2023acceleration}.

\fw{} possesses a number of attractive properties: it is easy to implement, requires only first-order information, and avoids projections. Furthermore, the iterates of \fw{} and its variants \cite{braun2019blended,  garber2016linear,guelat1986some,holloway1974extension,lacoste2015global,  tsuji2022pairwise} store the iterate as a convex combination of vertices of the feasible region, and the sparsity of this convex combination increases by at most one after each iteration. The attractive properties of \fw{} and its variants have been exploited in a variety of settings, such as approximate vanishing ideal computations \cite{ wirth2023approximate,wirth2022conditional}, deep neural network training \cite{macdonald2022interpretable,ravi2018constrained}, kernel herding \cite{bach2012equivalence}, matrix recovery \cite{garber2018fast,mu2016scalable}, nonlinear dynamics identification \cite{carderera2021cindy}, optimal transport \cite{courty2016optimal, paty2019subspace}, tensor completion \cite{bugg2022nonnegative,guo2017efficient}, and video co-localization \cite{joulin2014efficient}.  See \cite{braun2022conditional} for a summary on the latest research.

A feature that sets \fw{} apart from other first-order algorithms is its {\em affine covariance} when combined with line-search or open-loop step-sizes rules.  Affine covariance means that an affine reparametrization of the problem induces the same affine reparametrization on the iterates generated by the algorithm, as formally articulated by Kerdreux et al. in~\cite{kerdreux2021affine}:

\begin{definition}[Affine covariance]
Given an invertible matrix $B \in\R^{n\times n}$ and vector $\bb\in\R^n$, consider the affine transformation $\yy = B \xx + \bb$ of \eqref{eq:opt}:
\begin{equation}\label{eq:aff_opt}\tag{AFF-OPT}
\min_{\yy\in B\cC+\bb}f(B^{-1}(\yy -\bb)).
\end{equation}
An algorithm is considered \emph{affine covariant} if its iterates $(\xx_t)_{t\in\N}$ for \eqref{eq:opt} and $(\yy_t)_{t\in\N}$ for \eqref{eq:aff_opt} satisfy
$
\yy_t = B\xx_t + \bb
$
for all $t\in\N$.
\end{definition}

Although this property is sometimes referred to as {\em affine invariance,} that terminology is not accurate because the iterates generated by the algorithm actually change when the problem is reformulated via an affine transformation.  Instead, when an algorithm is affine covariance, some key features such as its convergence rate, must be affine-invariant.  Therefore, for line-search or open-loop step-sizes rules it is natural to expect the convergence rates derived for \fw{} to be affine-invariant as well.

\begin{algorithm}[t!]
\SetKwInput{Input}{Input} \SetKwInput{Output}{Output}
\SetKwComment{Comment}{$\triangleright$\ }{}
\caption{Frank-Wolfe algorithm (\fw{})}\label{alg:fw}
  \Input{$\xx_0\in \cC$.}
  \Output{$(\xx_t)_{t\in\N} \in \cC$.}
  \hrulealg
  \For{$t= 0, 1, \ldots $}{
        {$\vv_t \in \argmin_{\vv\in\cC} \left\langle\nabla f\left(\xx_t\right), \vv - \xx_t\right\rangle$\label{line:fw_LMO}}\\
        {$\xx_{t+1} \gets \xx_t + \eta_t \left( \vv_t - \xx_t\right)$ for some $\eta_t\in[0,1]$}\label{line:fw_update}}
\end{algorithm}

  





Several papers have explored affine-invariant convergence rates for various versions of the \fw{} algorithm \cite{gutman2021condition, kerdreux2021affine,lacoste2013affine, lacoste2015global,  pena2023affine}, predominantly when employing line-search or backtracking line-search techniques.  By contrast, only a limited number of results exist regarding \fw{} with open-loop step-sizes \cite{clarkson2010coresets, jaggi2013revisiting}, and they only cover rates of order $\cO(t^{-1})$.
This limitation of results is surprising given that open-loop step-sizes offer practical advantages over more involved step-size rules. First, they are easy-to-implement, taking the form $\eta_t = \frac{\ell}{t+\ell}$ for some $\ell\in\N_{\geq 1}$. Second, open-loop step-sizes adapt to challenging-to-compute parameters associated with growth properties. This stands in stark contrast to methods like \fw{} with short-step or adaptive step-sizes, where achieving accelerated rates requires knowledge of smoothness parameters \cite{garber2015faster, ghadimi2019conditional, kerdreux2021projection, nesterov2018complexity}. Most significantly, \cite{wirth2023acceleration} showed that open-loop \fw{} variants achieve acceleration in the challenging  \emph{Wolfe's lower bound setting}~\cite{wolfe1970convergence}:  This setting, common in practical applications, involves minimizing a smooth and strongly convex objective function over a polytope where the optimizer lies on the boundary of the polytope but is not a vertex.  When this is the case, \cite{wolfe1970convergence} showed that the rate of convergence of \fw{} with line-search or short-step converges cannot be better than $\Omega(t^{-1-\epsilon})$ for any $\epsilon > 0$. While active set variants of \fw{} have been introduced to achieve linear convergence rates in Wolfe's lower bound setting \cite{ garber2016linear, guelat1986some, lacoste2015global, tsuji2022pairwise}, they tend to be memory-intensive when iterates require substantial storage space.
For the same Wolfe's lower bound setting, recent work has established accelerated convergence rates for \fw{} with open-loop step-sizes under mild assumptions. The \emph{momentum-guided Frank-Wolfe algorithm} (MFW) \cite{li2021momentum}, employing open-loop step-size $\eta_t = \frac{2}{t+2}$ and momentum, achieves a rate of $\cO(t^{-2})$ within Wolfe's lower bound setting. For vanilla \fw{}, \cite{bach2021effectiveness} established an asymptotic convergence rate of $\cO(t^{-2})$ with open-loop step-sizes within  Wolfe's lower bound setting, although the result relies on assumptions on the third derivative of $f$. Most recently, 
\cite{wirth2023acceleration} derived a non-asymptotic version of the \cite{bach2021effectiveness} result with significantly weaker assumptions.
Furthermore, \cite{wirth2023acceleration} characterized several other settings for which \fw{} with open-loop step-sizes admits accelerated convergence rates. Most notably, they proved that \fw{} with open-loop step-size $\eta_t = \frac{\ell}{t+\ell}$ converges at a rate of order $\cO(t^{-\ell/2})$ when the feasible region is strongly convex and the gradients of $f$ are bounded away from zero. While numerical experiments in \cite{wirth2023acceleration} suggest a potential improvement to $\cO(t^{-\ell})$, their analysis is unable to capture this stronger rate.

\subsection{Contributions}\label{sec.contributions}
This paper combines two recent research paths, namely affine invariance and open-loop step-sizes as featured in \cite{pena2023affine} and \cite{wirth2023acceleration} respectively. 
This yields
the following developments for \fw{} with open-loop step-sizes of the form $\eta_t = \frac{\ell}{t+\ell}$, $\ell\in\N_{\geq 1}$:

To establish affine-invariant rates, we adapt and extend the growth properties initially introduced in~\cite{pena2023affine}. This adaptation includes the introduction of a new {\em gaps growth property}, which complements the previously defined {\em strong growth} and {\em weak growth} properties. The gaps growth property enables us to establish accelerated rates of convergence for \fw{} under the challenging Wolfe's lower bound setting~\cite{wolfe1970convergence}.
The growth properties are related to the {\em finite curvature} condition used by Jaggi~\cite{jaggi2013revisiting}. These properties, expressed in an affine-invariant manner, capture the essential characteristics of a problem instance of~\eqref{eq:opt} necessary for establishing accelerated convergence rates. This approach inherently avoids spurious dependencies on non-essential and typically affine-dependent objects. Consequently, we derive convergence results for \fw{} with open-loop steps that surpass the previous state-of-the-art results in~\cite{wirth2023acceleration}.

 We integrate and extend the developments in \cite{pena2023affine} and \cite{wirth2023acceleration}, providing affine-invariant versions of the accelerated rates initially derived in an affine-dependent manner in~\cite{wirth2023acceleration}.
Our strongest results are presented in Section~\ref{sec:strong_growth} under the strong growth property.  Theorem~\ref{thm:1_strong} and Theorem~\ref{thm:rate_strong} improve upon the previously best-known convergence rates from~\cite{wirth2023acceleration}. 
Specifically, we establish new accelerated convergence rates beyond $\cO(t^{-2})$ when the feasible region is uniformly convex, and the norm of the gradient of the objective function is bounded away from zero. Additionally, when the feasible region is strongly convex, we derive rates of order $\cO(t^{-\ell})$, surpassing the rates of order $\cO(t^{-\ell/2})$ derived in~\cite{wirth2023acceleration}.
The results in Section~\ref{sec:strong_growth}
yield convergence rates for both the suboptimality gap and the primal-dual gap, unlike the analysis in \cite{wirth2023acceleration}, which is limited to the suboptimality gap. Section~\ref{sec:weak_growth} and Section~\ref{sec:gaps_growth} present accelerated rates of order $\cO(t^{-2})$, namely Theorem~\ref{thm:rate-weak}, Theorem~\ref{thm:rate-gaps}, and 
Theorem~\ref{thm:rate-gaps-relaxed}, when the milder weak growth and gaps growth properties hold. 
These results improve over the corresponding results in \cite{wirth2023acceleration} in that our work is affine-invariant and holds for any $\ell\in\N_{\geq 1}$, whereas the rates in \cite{wirth2023acceleration} are affine-dependent and limited to $\ell = 4$.
Section~\ref{sec:gaps_growth} also includes Corollary~\ref{cor:rel_gap}, which establishes the first affine-invariant convergence rate guarantee of order $\cO(t^{-2})$ in the challenging Wolfe's lower bound setting~\cite{wolfe1970convergence}. This result is particularly significant as it highlights the superior performance of the simple open-loop \fw{} approach over the more intricate line-search or short-step alternatives for the polytope setting.

\subsection{Paper outline}

The remaining sections of the paper are organized as follows.  Section~\ref{sec:preliminaries} introduces notation and terminology used throughout, including the growth properties that provide the backbone for our main developments.

Section~\ref{sec:strong_growth},  Section~\ref{sec:weak_growth}, 
and Section~\ref{sec:gaps_growth} present our main results.  Each of these three sections develops convergence rates when $(\cC,f)$ satisfies respectively the strong growth, weak growth, and gaps growth properties in addition to the strong $(M,0)$ growth property. Each of these sections includes a subsection detailing sufficient conditions for the 
relevant growth property to hold and a subsection with numerical experiments on stylized examples.  For ease of exposition, in our stylized examples we take $\cC$ to be a unit $\ell_p$-norm ball for $1\le p\le\infty$.  However, the sufficient conditions that guarantee the growth properties in each of these sections apply to a much broader classes of domains.

Section~\ref{sec:numerical_experiments} reports computational experiments involving problem instances from data science applications.  Therein, \fw{} with open-loop steps attains accelerated convergence rates even in some cases that are not yet covered by our theoretical developments.  We hope the developments in this paper inspire future research to investigate new classes of problems that satisfy growth properties.  

We conclude the paper with a discussion in Section~\ref{sec.discussion}.

The reader interested in get a quick grasp of our main developments can easily do so by skipping the proofs of Theorem~\ref{thm:rate_strong}, Theorem~\ref{thm:rate-weak}, and Theorem~\ref{thm:rate-gaps-relaxed}. These are the most technically involved components of our paper.

\section{Preliminaries}\label{sec:preliminaries}
We rely on the following notation and growth properties as introduced in \cite{pena2023affine}.
Throughout, let $\cC\subseteq \R^n$ be a compact convex set and let $f\colon \cC \to \R$ be convex and differentiable in an open set containing $\cC$.
For a set $\cC\subseteq\R^n$, let $\relativeinterior(\cC)$ and $\relativeboundary(\cC)$ denote the relative interior and boundary of $\cC$, respectively. Given a norm $\|\cdot\|$, let $\|\cdot\|_*$ denote its dual norm.
\subsection{Measures of optimality}\label{sec.measures_of_optimality}
We recall several measures of optimality for the optimization problem \eqref{eq:opt}.
Recall that the \emph{Bregman divergence} $D_f:\cC\times\cC\rightarrow \R$  of $f$ is defined as follows
\begin{align*}
D_f(\yy,\xx) = f(\yy) - f(\xx) - \left\langle\nabla f\left(\xx\right), \yy-\xx \right\rangle.
\end{align*}
We will also rely on the {\em primal suboptimality gap} function $\subopt\colon\cC\rightarrow \R$ and {\em Wolfe gap} function 
$\gap:\cC\rightarrow \R$
defined as follows.  For $\xx \in \cC$ and $\vv \in \argmin_{\yy\in \cC}\left\langle\nabla f\left(\xx\right), \yy-\xx \right\rangle$, let
\begin{align*}
\subopt(\xx):= f(\xx) - \min_{\yy \in \cC} f(\yy) \qquad \text{and} \qquad \gap(\xx):= \left\langle\nabla f\left(\xx\right), \xx - \vv\right\rangle.
\end{align*}
Note that $\gap(\xx)$ is a particularly attractive measure of optimality for \fw{} as it has to be computed for each iteration in Line~\ref{line:fw_LMO} of Algorithm~\ref{alg:fw}.
A straightforward calculation shows that the iterates generated by Algorithm~\ref{alg:fw} satisfy the following equality that will be used repeatedly throughout the paper:
\begin{align}\label{eq:fw-step}
\subopt(\xx_{t+1}) = \subopt(\xx_t) -\eta_t \gap(\xx_t) + D_f(\xx_{t}+\eta_t(\vv_t-\xx_t),\xx_t).  
\end{align}
We give convergence results both for the suboptimality gap $\subopt(\xx_t)$ of the $t$-th iterate $\xx_t$ as well as for the best duality gap defined next. 
Consider the \emph{Fenchel dual} of~\eqref{eq:opt} namely
\begin{align*}
\max_{\uu\in \R^n} \left\{-f^*(\uu) - \sigma_{\cC}(-\uu)\right\},
\end{align*}
where $f^*\colon\R^n \to [-\infty, \infty]$ denotes the convex conjugate of $f$, that is, for $\yy\in\R^n$, 
\begin{align*}
    f^*(\yy) := \sup_{\xx\in\cC} \{\langle \yy, \xx\rangle - f(\xx)\}
\end{align*}
and $\sigma_{\cC}:\R^n\rightarrow \R$ denotes the support function of $\cC$, that is, for $\yy\in\R^n$,
\begin{align*}
    \sigma_{\cC} (\yy) = \max_{\xx\in \cC}\ip{\yy}{\xx}.
\end{align*}
For $t\in\N$, define
\begin{align*}
\primaldual(\xx_t) = \min_{k=0,\dots,t}\left\{f(\xx_t) + f^*(\nabla f(\xx_k)) + \sigma_{\cC}(-\nabla f(\xx_k))\right\}.
\end{align*}
A simple calculation shows that
\begin{align*}
    \primaldual(\xx_t) = \min_{k=0,\dots,t}\left\{f(\xx_t) - f(\xx_k) + \gap(\xx_k)\right\}.
\end{align*}
The quantity $\primaldual(\xx_t)$ is the duality gap between the primal iterate $\xx_t$ and the best of the dual iterates $\uu_k:=\nabla f(\xx_k)$ for $k=0,1,\dots,t$.  In particular, Fenchel duality and a simple calculation show that
\begin{equation}\label{eq:various.gaps}
\subopt(\xx_t) \leq \primaldual(\xx_t) \leq f(\xx_t) + f^*(\nabla f(\xx_t)) + \sigma_{\cC}(-\nabla f(\xx_t))=\gap(\xx_t).
\end{equation} 
Furthermore,~\eqref{eq:fw-step} implies that
\begin{align}\label{eq:gap}
\primaldual(\xx_{t+1}) \leq \primaldual(\xx_t)-\eta_t \gap(\xx_t) + D_f(\xx_{t}+\eta_t(\vv_t-\xx_t),\xx_t).  
\end{align}
To ease notation, sometimes we will write 
$\gap_t$, $\primaldual_t$, and $\subopt_t$
as shorthands for $\gap(\xx_t)$, $\primaldual(\xx_t)$, and $\subopt(\xx_t)$, respectively, when $\xx_t$ is the $t$-th iterate generated by Algorithm~\ref{alg:fw}. 

\subsection{Growth properties}\label{sec.growth.properties}
The following growth properties concerning the functions $D_f, \gap,$ and $\subopt$ are adapted from~\cite{pena2023affine} and will play a central role in our developments.

\begin{definition}[Growth properties]\label{def:growth_properties}
    Let $\cC\subseteq \R^n$ be a compact convex set, let $f\colon\cC \to \R$ be convex and differentiable in an open set containing $\cC$.
    For $M >0$, $m>0$, and $r\in[0,1]$, we say that $(\cC,f)$ satisfies the 
    \begin{enumerate}
        \item \emph{strong $(M,r)$-growth property} if for all $\xx \in \cC$ and $\vv \in \argmin_{\yy\in \cC}\left\langle\nabla f\left(\xx\right), \yy-\xx \right\rangle$, it holds that
    \begin{align}\label{eq:strong_gp}
        D_f(\xx+\eta(\vv-\xx),\xx) \leq \frac{M \eta^2}{2} \gap(\xx)^r \qquad \text{for all} \ \eta \in [0, 1].
    \end{align}
\item {\em weak $(M,r)$-growth property} if for all $\xx \in \cC$ and $\vv \in \argmin_{\yy\in \cC}\left\langle\nabla f\left(\xx\right), \yy-\xx \right\rangle$, it holds that
    \begin{align}\label{eq:weak_gp}
        D_f(\xx+\eta(\vv-\xx),\xx)\cdot \subopt(\xx)^{1-r} \leq \frac{M \eta^2}{2} \gap(\xx) \qquad \text{ for all} \ \eta \in [0, 1].
    \end{align}
    \item \emph{gaps $(m,r)$-growth property} if for all $\xx \in \cC$ it holds that
    \begin{align}\label{eq:gap_gp}
        m \cdot \subopt(\xx)^{1-r} \leq \gap(\xx).
    \end{align}
    \end{enumerate}
\end{definition}
Note that the gaps $(m, r)$-growth property follows from convexity for $m = 1$ and $r = 0$.
Since $\gap(\xx) \geq \subopt(\xx)$ for $\xx\in \cC$, it is evident that the 
strong $(M, r)$-growth property \eqref{eq:strong_gp} implies the weak $(M, r)$-growth property \eqref{eq:weak_gp}. 
In the special case when $r = 0$, the strong $(M,0)$-growth property is a slight relaxation of the curvature condition in~\cite{jaggi2013revisiting}. Indeed, for $r = 0$, Inequality~\eqref{eq:strong_gp} reads 
\begin{align}\label{eq.jaggi}
    D_f(\xx+\eta(\vv-\xx),\xx) \leq \frac{M \eta^2}{2} \qquad \text{for all} \ \eta \in [0, 1].
\end{align}
By combining~\eqref{eq:gap_gp} and~\eqref{eq.jaggi} it follows that the weak $(M/m,r)$-growth property holds if both the strong $(M,0)$-growth and the gaps $(m,r)$-growth properties hold.

The growth properties stated in Definition~\ref{def:growth_properties} are properties of the problem~\eqref{eq:opt} only.  In particular, they do not depend on any algorithmic approach to it.  Our developments concerning the behavior of \fw{} when $\cC$ is a polytope rely on the following {\em relaxed gaps growth property} stated in terms of the problem~\eqref{eq:opt} together with two sequences $(\xx_t)_{t\in\N}$ and $(\delta_t)_{t\in\N}$, where $\xx_t \in\cC$ and $\delta_t\in\R_{>0}$ for all $t\in\N$. The additional dependence on the two sequences enables the relaxed gaps growth property to depend on both the problem~\eqref{eq:opt} and on the algorithmic approach to it.  We rely on this feature for our developments in Section~\ref{sec.suff_relaxed_gaps_m_r}.

\begin{definition}[Relaxed gaps growth property]\label{def:relaxed_gaps_growth}
    Let $\cC\subseteq \R^n$ be a compact convex set, let $f\colon\cC \to \R$ be convex and differentiable in an open set containing $\cC$, and let $\xx_t\in \cC$ and $\delta_t > 0$ for $t\in\N$.  For $M >0$, $m>0$, and $r\in[0,1]$ we say that $(\cC,f,(\xx_t)_{t\in\N},(\delta_t)_{t\in\N})$ satisfies the {\em relaxed gaps growth property} if there exists a threshold $\fwt\in \N$ such that for all $t \in \N_{\geq \fwt}$ either
     $\subopt(\xx_t) < \delta_t$
    or
    \begin{equation}\label{eq:relaxed_growth}
    m \cdot \subopt(\xx_t)^{1-r} \le \gap(\xx_t).
    \end{equation}
\end{definition}

It is evident that the gaps growth property in Definition~\ref{def:growth_properties} implies that relaxed gaps growth property in Definition~\ref{def:relaxed_gaps_growth} for any two  sequences $(\xx_t)_{t\in\N}$ and $(\delta_t)_{t\in\N}$, where $\xx_t \in\cC$ and $\delta_t\in\R_{>0}$ for all $t\in\N$, since~\eqref{eq:gap_gp} implies~\eqref{eq:relaxed_growth} for any sequence $(\xx_t)_{t\in\N}$, where $\xx_t \in\cC$ for all $t\in\N$.
%
%
\subsection{Curvature of the feasible region and objective}
We introduce several notions of curvature of the feasible region and objective function. Combinations of these properties and assumptions on the location of the optimizer of \eqref{eq:opt} imply various growth properties. Although the concepts below can be stated in terms of any norm in $\R^n$, we will restrict ourselves to the following affine-invariant $\cC$-norm intrinsic to the set $\cC$. This will guarantee affine invariance throughout our exposition.

\begin{definition}[$\cC$-norm]\label{def:C_norm}
Let $\cC \subseteq \R^n$ be a compact convex set. Define the \emph{$\cC$-norm} $\|\cdot\|_{\cC}$ on the linear subspace  
$\text{span}(\cC - \cC)$ as follows. For $\xx \in \text{span}(\cC - \cC)$ let
\begin{align*}
    \|\xx\|_\cC = \inf\left\{\lambda>0: \xx = \frac{\lambda}{2}(\uu-\vv) \text{ for some } \uu,\vv\in \cC\right\}.
\end{align*}
In other words, $\|\cdot\|_\cC$ is the gauge function of the set $\frac{1}{2}(\cC-\cC)$.
\end{definition}

To simplify notation, we will assume without loss of generality that $\text{span}(\cC - \cC) = \R^n$.

\begin{definition}[Uniformly convex set]\label{def:unif_cvx_C}
Let $\cC \subseteq \R^n$ be a compact convex set, $\alpha >0$, and $p>0$. We say that $\cC$ is \emph{$(\alpha, p)$-uniformly convex} with respect to $\|\cdot\|$ if for all $\xx, \yy \in \cC$, $\gamma \in [0,1]$, and $\zz \in \R^n$ such that $\|\zz\|=1$, it holds that
$$
    \gamma \xx + ( 1- \gamma) \yy + \gamma (1 - \gamma) \alpha \|\xx-\yy\|^p \zz \in \cC.
$$
We refer to $(\alpha, 2)$-uniformly convex sets as \emph{$\alpha$-strongly convex sets}. 
\end{definition}
Famous examples of uniformly convex sets are the $\ell_p$-balls: For $p\in[1, \infty]$, the $\ell_p$-norm is defined as
\begin{align*}
    \|\cdot\|_p \colon \R^n & \to \R_{\geq 0} \\
    \xx& \mapsto \begin{cases}
        (\sum_{i = 1} |x_i|^p)^\frac{1}{p}, & \text{if} \ p\in [1, \infty[\\
        \max\{|x_i|\mid i\in\{1, \ldots, n\}\}, & \text{if} \ p = \infty.
    \end{cases} 
\end{align*}
%
For $p\geq 1$ the $\ell_p$-ball in $\R^n$ is the set $\{\xx\in \R^n \mid \|\xx\|_p \le 1\}.$
Throughout, we often rely on the following relationship between different $\ell_p$-norms.
\begin{fact}\label{fact:lp_norm}
 For $p > r \geq 1$ and $\xx\in\R^n$, it holds that $\|\xx\|_p \leq \|\xx\|_r \leq n^{\frac{1}{r}-\frac{1}{p}}\|\xx\|_p $.
\end{fact}
\begin{proposition}[\cite{hanner1956uniform}]\label{prop:unif_cvx}
    For $p\in]1, \infty[$, let $\cC\subseteq \R^n$ be the $\ell_p$-ball.
    Then, the following holds:
    \begin{enumerate}
        \item When $p\in]1, 2]$, then $\cC$ is $(p-1)/4$-strongly convex with respect to $\|\cdot\|_p$.
        \item When $p\in [2,\infty[$, then $\cC$ is $(1/(p2^{p-1}), p)$-uniformly convex with respect to $\|\cdot\|_p$.
    \end{enumerate}
\end{proposition}

When $\cC$ is the $\ell_p$-ball for $p\geq 1$, we have $\cC = \frac{1}{2}(\cC-\cC)$ and thus $\|\cdot\|_\cC$ is the gauge function of the $\ell_p$-ball, which is precisely $\|\cdot\|_p$~\cite{hiriU04}. 
The $\ell_p$-balls will serve as running examples throughout this paper. For the remainder of this section, we focus on the curvature of the objective.
\begin{definition}[Smoothness]\label{def:smooth_f}
Let $\cC \subseteq \R^n$ be a convex set, let $f\colon \cU \to \R$ be differentiable, where $\cU$ is an open set containing $\cC$, and let $L > 0$. We say that $f$ is \emph{$L$-smooth} over $\cC$ with respect to $\|\cdot\|$ if for all $\xx, \yy\in \cC$, it holds that
$$
    f(\yy) \leq f(\xx) + \langle \nabla f(\xx), \yy - \xx\rangle + \frac{L}{2}\|\xx-\yy\|^2.
$$
\end{definition}
A simple calculation shows that $f$ is \emph{$L$-smooth} over $\cC$ with respect to $\|\cdot\|$ if $\nabla f$ is $L$-Lipschitz continuous over $\cC$ with respect to $\|\cdot\|$, that is, if for all $\xx,\yy\in \cC$, it holds that
\begin{align*}
\|\nabla f(\xx)-\nabla f(\yy)\|_* \le L\|\xx-\yy\|.
\end{align*}
\begin{proposition}\label{prop.strong.M.0}
Suppose that $\cC \subseteq \R^n$ is a compact convex set and $f$ is \emph{$L$-smooth} over $\cC$ with respect to $\|\cdot\|_{\cC}$.  Then, $(\cC, f)$ satisfies the strong $(M,0)$-growth property where $M = 4L.$
\end{proposition}
\begin{proof}
This readily follows from the definition of strong $(M,0)$-growth, $\|\cdot\|_{\cC}$, $L$-smoothness and the fact that $\|\xx-\yy\|_{\cC}\le 2$ for all $\xx,\yy\in\cC.$
\end{proof}

\begin{definition}[Hölderian error bound]\label{def:heb}
Let $\cC \subseteq \R^n$ be a compact convex set, let $f\colon \cC \to \R$ be convex,
let $\mu> 0$, and let $\theta \in [0, 1/2]$. We say that $f$ satisfies a \emph{$(\mu, \theta)$-Hölderian error bound} with respect to $\|\cdot\|$ if for all $\xx\in \cC$ and $\xx^*\in\argmin_{\xx\in\cC}f(\xx)$, it holds that
\begin{align*}
    \mu(f(\xx)-f(\xx^*))^\theta \geq \min_{\yy\in\argmin_{\zz\in\cC}f(\zz)} \|\xx-\yy\|.
\end{align*}
\end{definition}
Our running examples throughout the paper will rely on the following facts concerning the $\ell_p$-balls and the $L$-smooth functions satisfying a $(\mu, \theta)$-Hölderian error bound with respect to $\|\cdot\|_2$.

\begin{lemma}\label{lemma:lp_balls}
Suppose that $\cC\subseteq \R^n$ is the $\ell_p$-ball for some $p\geq 1$ and for $L>0$, $\mu > 0$ and $\theta \in[0,1/2]$, let $f\colon \cC \to \R$ be a convex and $L$-smooth function satisfying a $(\mu, \theta)$-Hölderian error bound with respect to $\|\cdot\|_2$. Then, the following properties hold:
\begin{enumerate}
\item
 When $p\in[1, 2]$, the function $f$ is $L$-smooth and satisfies a $(\mu n^{\frac{1}{p}-\frac{1}{2}}, \theta)$-Hölderian error bound with respect to $\|\cdot\|_{\cC}$. In particular, $(\cC,f)$ satisfies the strong $(4L,0)$-growth property.
\item When $p\geq 2$, the function $f$ is 
$Ln^{\frac{1}{2}-\frac{1}{p}}$-smooth and satisfies a $(\mu,\theta)$-Hölderian error bound with respect to $\|\cdot\|_{\cC}$.
In particular, $(\cC,f)$ satisfies the strong $(4Ln^{\frac{1}{2}-\frac{1}{p}},0)$-growth property.
\end{enumerate}
\end{lemma}
\begin{proof}
The above facts are immediate consequences of Fact~\ref{fact:lp_norm} and Proposition~\ref{prop.strong.M.0}.
\end{proof}

\section{Strong $(M, r)$-growth}\label{sec:strong_growth}

This section presents affine-invariant convergence results for the Frank-Wolfe algorithm (\fw) with open-loop step-size of the form $\eta_t = \frac{\ell}{t+\ell}$ for $\ell \in \N_{\geq 2}$ when $(\cC, f)$ satisfies the strong $(M, r)$-growth property for some $r\in [0,1]$. As we detail in Section~\ref{sec.suff_strong_M_r} below, the developments in this section  answer 
several open questions raised in \cite{wirth2023acceleration} and also improve on some of the main convergence results therein.

\subsection{Convergence rates}\label{sec:strong_growth.results}

For ease of exposition, we consider three cases separately: $r=0$, $r=1$, and $r\in]0,1[$ as detailed in Theorem~\ref{thm:slow_conv}, Theorem~\ref{thm:1_strong}, and Theorem~\ref{thm:rate_strong}, respectively.  
The latter two theorems improve substantially the previous state-of-the-art rates derived in \cite{wirth2023acceleration}.

Theorem~\ref{thm:slow_conv} considers the one end of the spectrum when the strong $(M,0)$-growth property holds and shows that \fw{} with open-loop step-size as above attains convergence rate $\cO(t^{-1})$.  This is the same rate (modulo a small constant) that is attainable with step-sizes selected via exact line-search, approximate line-search, or short-step under the slightly stronger but nearly equivalent finite-curvature condition. 

\begin{theorem}[Strong $(M,0)$-growth]\label{thm:slow_conv}
    Let $\cC\subseteq \R^n$ be a compact convex set, let $f\colon\cC\to\R$ be convex and differentiable in an open set containing $\cC$, let $\eta_t = \frac{\ell}{t+\ell}$ for some $\ell\in\N_{\geq 2}$ and all $t\in\N$, and suppose that $(\cC, f)$ satisfies the strong $(M,0)$-growth property for some $M > 0$. Then, for the iterates of Algorithm~\ref{alg:fw} and $t\in\N_{\geq 1}$, it holds that 
\begin{equation}\label{eq.sublinear.rate}
\primaldual_t\leq \eta_t M = \cO(t^{-1}).
\end{equation}
\end{theorem}
\begin{proof}
We generalize the proof of Proposition~3.1 in \cite{wirth2023acceleration}.
Inequality~\eqref{eq:various.gaps}, Inequality~\eqref{eq:gap}, and the $(M,0)$-strong growth property imply that for $t\in \N$
\begin{equation}\label{eq:jaggi_early}
    \primaldual_{t+1}  \leq \primaldual_{t} - \eta_{t}\gap_{t} + \frac{\eta_{t}^2M}{2} \leq \primaldual_{t}(1 - \eta_{t}) + \frac{\eta_{t}^2M}{2}.
\end{equation}
We now proceed by induction. For $t=0$, we have $\eta_0=1, \; \eta_1 = \ell/(1+\ell) \geq 1/2,$ and thus~\eqref{eq:jaggi_early} implies that
\begin{align}\label{eq.pd1}
    \primaldual_1\le \frac{\eta_0^2M}{2} = \frac{M}{2} \leq \eta_1 M.
\end{align}
Hence~\eqref{eq.sublinear.rate} holds for $t=1$. 
Next, suppose~\eqref{eq.sublinear.rate} holds for $t\in \N_{\geq 1}$. Thus, \eqref{eq:jaggi_early} and $\eta_t = \frac{\ell}{t+\ell}$ imply that
\[
\primaldual_{t+1} \leq \frac{M\ell}{t+\ell}\cdot \frac{t}{t+\ell} + \frac{M\ell^2}{2(t+\ell)^2} =
\frac{M\ell}{t+\ell}\cdot\frac{t+\ell/2}{t+\ell} \leq \frac{M\ell}{t+\ell+1} = \eta_{t+1}M,
\]
where the last inequality follows because $(t+\ell)^2 \geq  (t+ \ell/2)(t+\ell +1)$ for $\ell \geq 2$. 
Thus, \eqref{eq.sublinear.rate} holds for $t+1$.
\end{proof}
Theorem~\ref{thm:1_strong} considers the other end of the spectrum when the strong $(M,1)$-growth property holds. In this case, the Frank-Wolfe algorithm with open-loop step-size $\frac{\ell}{t+\ell}$ for $\ell \in \N_{\geq 2}$ attains accelerated rate of convergence $\cO(t^{-\ell})$ once a threshold $\fwt\in\N$ is surpassed. There is a simple trade-off between $\ell$ and $\fwt$: the larger $\ell$, the faster the rate of convergence but also the higher the threshold $\fwt$.
\begin{theorem}[Strong $(M,1)$-growth]\label{thm:1_strong}
Let $\cC\subseteq \R^n$ be a compact convex set, let $f\colon\cC\to\R$ be convex and differentiable in an open set containing $\cC$, let $\eta_t = \frac{\ell}{t+\ell}$ for some $\ell\in\N_{\geq 1}$ and all $t\in\N$, and suppose that $(\cC, f)$ satisfies the strong $(M,1)$-growth property for some $M >0$.
Let $\fwt = \min\{i\in\N_{\geq 1} \mid M\eta_i/2 \leq 1 \} = \max\left\{1, \left\lceil\frac{\ell M }{2} - \ell\right\rceil\right\}$.  Then, for the iterates of Algorithm~\ref{alg:fw} and $t\in\N_{\geq \fwt}$, it holds that
\begin{align*}
        \primaldual_{t} \leq  \primaldual_{\fwt}  \srb{\frac{\eta_{t-1}}{\eta_{\fwt-1}}}^\ell \exp\left({\frac{M\ell^2}{2S}}\right) = \cO(t^{-\ell}).
\end{align*}
Let $c = \exp(\frac{1}{\ell})$ and $\epsilon \in ]0, 1[$. Then, for the iterates of Algorithm~\ref{alg:fw} and $t\in\N$, $t\geq \left\lceil c\cdot\frac{\ell M}{2(1-\epsilon)}-\ell + 1 \right\rceil$, it holds that
\begin{align*}
    \min_{i \in \{0, 1, \ldots, t\}} \gap_i 
    & \leq  \primaldual_{\fwt}  \srb{\frac{\eta_{t-1}}{\eta_{\fwt-1}}}^\ell \frac{\exp\left(3+{\frac{M\ell^2}{2S}}\right)}{\epsilon} = \cO(t^{-\ell}).
\end{align*}
\end{theorem}
The proof of Theorem~\ref{thm:1_strong} relies on the following technical lemma.

\begin{lemma}\label{lemma:telescope.simple}
    Let $\eta_t = \frac{\ell}{t+\ell}$ for some $\ell\in\N_{\geq 1}$ and all $t\in\N$. Then, for $\fwt, t\in \N_{\geq 1}$ such that $\fwt \leq t$, it holds that 
    \begin{align*}
        \prod_{i=\fwt}^{t} \left(1-\eta_i\right)\leq \left(\frac{\eta_t}{\eta_{\fwt-1}}\right)^\ell.
    \end{align*}
\end{lemma}
\begin{proof}
Since $\eta_i = \ell/(i+\ell)$ for $i\in \N$ and $\ell \in \N_{\geq 1}$ it follows that
\[ \prod_{i=\fwt}^{t} \left(1- \eta_i\right)
        = \prod_{i=\fwt}^{t} \frac{i}{i+\ell} 
        = \frac{\fwt  \cdots \left(\fwt-1 + \ell\right)}{(t+1)  \cdots (t+\ell)}
        \leq \left(\frac{\fwt -1+ \ell}{t+\ell}\right)^\ell  
        = \left(\frac{\eta_t}{\eta_{\fwt-1}}\right)^{\ell}.\]
\end{proof}

With Lemma~\ref{lemma:telescope.simple}, the proof of Theorem~\ref{thm:1_strong} follows from straightforward calculations.
\begin{proof}[Proof of Theorem~\ref{thm:1_strong}]
    The statement is trivially satisfied for $t = \fwt$. Suppose that $t\in\N_{\geq\fwt}$. Inequality~\eqref{eq:gap} and the strong $(M,1)$-growth property imply that 
\begin{align}\label{eq:1_strong_eq_for_gap}
    \primaldual_{t+1} &\le \primaldual_{t} - \eta_{t} \gap_{t} + D_f(\xx_{t}+\eta_{t}(\vv_{t} - \xx_{t}),\xx_{t}) \notag \\ &\leq   \primaldual_{t} + \eta_{t} \gap_{t}  \left(\frac{\eta_{t}M}{2}  - 1\right).
\end{align}
Inequality~\eqref{eq:various.gaps} and $\frac{\eta_{t} M}{2} \leq 1$ for $t\in \N_{\geq S}$, it follows that for $t\in\N_{\geq \fwt}$, it holds that
\begin{align*}
    \primaldual_{t+1} & \leq \primaldual_{t} + \eta_{t}\primaldual_{t}  \srb{\frac{\eta_{t}M}{2}  - 1}\\
    &\leq  \primaldual_{\fwt} \prod_{i=\fwt}^{t}  \left(1-\eta_i+\frac{\eta_i^2 M}{2}\right)\\
    & = \primaldual_{\fwt} \prod_{i=\fwt}^{t} (1-\eta_i)  \left(1 + \frac{\eta_i^2 M}{2(1-\eta_i)}\right) & \text{$\triangleright$ since $\eta_i \neq 1$ for $i \in\N_{\geq \fwt}$}\\
     & \le \primaldual_{\fwt} \srb{\frac{\eta_t}{\eta_{\fwt-1}}}^\ell \prod_{i=\fwt}^{t} \left(1 + \frac{\eta_i^2 M}{2(1-\eta_i)}\right). & \text{$\triangleright$ by Lemma~\ref{lemma:telescope.simple}}
\end{align*}
To finish, it suffices to bound the very last product.  Indeed, for all $t\in\N_{\geq \fwt}$ it holds that 
\begin{align*}
\prod_{i=\fwt}^{t} \left(1 + \frac{\eta_i^2 M}{2(1-\eta_i)}\right)
    & = \prod_{i=\fwt}^{t} \left(1 + \frac{\ell^2(i+\ell) M}{2(i+\ell)^2i}\right)\\
    & = \prod_{i=\fwt}^{t} \left(1 + \frac{\ell^2 M}{2(i+\ell)i}\right)\\
    & \le \prod_{i=\fwt}^{t} \exp\left(\frac{\ell^2 M}{2(i+\ell)i}\right)
    & \text{$\triangleright$ since $1+x\leq \exp(x)$ for all $\xx\in\R$}\\
      & = \exp\left(\frac{M\ell}{2}\sum_{i=S}^t \left(\frac{1}{i}-\frac{1}{i+\ell} \right)\right) \\
      & \le \exp\left(\frac{M\ell}{2}\sum_{i=S}^{S+\ell-1}\frac{1}{i}\right) \\
      & \leq  \exp\left({\frac{M\ell^2}{2S}}\right).
\end{align*}

It remains to prove the convergence guarantee of $\min_{i\in\{0,1,\ldots,t\}}\gap_i$. 
For any $\Rfwt \geq \left\lceil\frac{\ell M }{2(1-\epsilon)} - \ell\right\rceil$ and $i\in\N_{\geq \Rfwt}$, by \eqref{eq:1_strong_eq_for_gap}, it holds that
$
    \primaldual_{i+1} \leq \primaldual_{i} - \epsilon \eta_i \gap_i.
$
Thus, for any $t\in\N_{\geq\Rfwt +1}$, it holds that
\begin{align*}
    \primaldual_{t+1} &\leq \primaldual_{\Rfwt} -  \epsilon \ell \sum_{j=\Rfwt}^t\frac{1}{j+\ell} \cdot \min_{i \in \{\Rfwt, \Rfwt + 1, \ldots, t\}} \gap_i \\
    &\leq \primaldual_{\Rfwt} - \epsilon \ell \left(\int_{\Rfwt}^{t+1}\frac{1}{x+\ell} dx\right) \cdot \min_{i \in \{\Rfwt, \Rfwt + 1, \ldots, t\}} \gap_i,
    \\
    &\leq \primaldual_{\Rfwt} - \epsilon \ell \log\left(\frac{t+\ell}{\Rfwt+\ell}\right) \cdot \min_{i \in \{\Rfwt, \Rfwt + 1, \ldots, t\}} \gap_i
\end{align*}
Since $0 \leq \primaldual_{t+1}$, we have
\begin{align*}
    \min_{i \in \{\Rfwt, \Rfwt + 1, \ldots, t\}} \gap_i &\leq \frac{\primaldual_{\Rfwt}}{\epsilon\ell \log\left(\frac{t+\ell}{\Rfwt+\ell}\right)} 
    \\& \leq \frac{\primaldual_{\fwt}}{\epsilon \ell\log\left(\frac{t+\ell}{\Rfwt+\ell}\right)}  \srb{\frac{\eta_{\Rfwt-1}}{\eta_{\fwt-1}}}^\ell \exp\left({\frac{M\ell^2}{2S}}\right) \\
    & \leq \frac{\left(\frac{t+\ell-1}{\Rfwt+\ell-1}\right)^\ell \primaldual_{\fwt}}{\epsilon \ell\log\left(\frac{t+\ell}{\Rfwt+\ell}\right)}  \srb{\frac{\eta_{t-1}}{\eta_{\fwt-1}}}^\ell \exp\left({\frac{M\ell^2}{2S}}\right).
\end{align*}
We now choose $\Rfwt$ to roughly minimize the right-hand side. The derivative of $g(x) = \frac{x^\ell}{\log(x)}$ is $g'(x)= \frac{x^{\ell-1}(\ell (\log(x) - 1))}{\log^2(x)}$. Note that $g(x)$ is convex when restricted to the interval $[1, \infty[$. Thus, on this interval $g(x)$ is minimized at $\exp(\frac{1}{\ell})$. 
Let $c:= \exp(\frac{1}{\ell})$, $t\geq \left\lceil c\cdot\frac{\ell M}{2(1-\epsilon)}-\ell + 1 \right\rceil$, and $\Rfwt := \left\lfloor\frac{t+\ell}{c}-\ell\right\rfloor$. Then,
\begin{align*}
    c \cdot (\Rfwt + \ell +1) \geq  t+\ell \geq c \cdot (\Rfwt + \ell) \qquad \text{and} \qquad \Rfwt \geq  \left\lfloor\frac{\ell M}{2(1-\epsilon)} - \ell + 1\right\rfloor \geq \fwt.
\end{align*}
Thus, for the iterates of Algorithm~\ref{alg:fw} and $t\in\N$, $t\geq \left\lceil c\cdot\frac{\ell M}{2(1-\epsilon)}-\ell + 1 \right\rceil$, it holds that
\begin{align*}
    \min_{i \in \{0, 1, \ldots, t\}} \gap_i 
    & \leq \frac{ \left(c\cdot \frac{\Rfwt + \ell + 1 }{\Rfwt + \ell-1}\right)^\ell \primaldual_{\fwt}}{\epsilon}  \srb{\frac{\eta_{t-1}}{\eta_{\fwt-1}}}^\ell \exp\left({\frac{M\ell^2}{2S}}\right)\\
    & \leq \frac{ \left(\exp\left(\frac{1}{\ell}\right)\right)^\ell\left(1+\frac{2}{\ell}\right)^\ell \primaldual_{\fwt}}{\epsilon}  \srb{\frac{\eta_{t-1}}{\eta_{\fwt-1}}}^\ell \exp\left({\frac{M\ell^2}{2S}}\right)\\
    & \leq \primaldual_{\fwt}  \srb{\frac{\eta_{t-1}}{\eta_{\fwt-1}}}^\ell \frac{\exp\left(3+{\frac{M\ell^2}{2S}}\right)}{\epsilon},
\end{align*}
where the second inequality follows from $c^\ell = \left(\exp\left(\frac{1}{\ell}\right)\right)^\ell$ and $\frac{\Rfwt + \ell +1 }{\Rfwt + \ell-1} = \frac{\Rfwt + \ell - 1 }{\Rfwt + \ell-1} + \frac{2}{\Rfwt + \ell-1} \leq 1 + \frac{2}{\ell}$ and the third inequality follows from $\left(\exp\left(\frac{1}{\ell}\right)\right)^\ell = \exp(1)$ and $(1 + \frac{2}{\ell})^\ell \leq \exp\left(\frac{2}{\ell} \ell \right) = \exp(2)$.
\end{proof}
Theorem~\ref{thm:rate_strong} considers the case when the strong $(M, r)$-growth property holds for some $r\in[0,1[$.  In this case, the Frank-Wolfe algorithm with open-loop step-size $\frac{\ell}{t+\ell}$ for $\ell \in \N_{\geq 1}$ attains accelerated rate of convergence $\cO(t^{-\ell + \epsilon} + t^{-\frac{1}{1-r}})$ for any $\epsilon \in ]0,\ell[$. Note that the rate is capped at $\cO(t^{-\frac{1}{1-r}})$ when $r\in [0,1[$, in contrast to the uncapped rate 
$\cO(t^{-\ell})$ attainable when $r=1$ for sufficiently aggressive open-loop step-sizes.

\begin{theorem}[Strong $(M,r)$-growth, $r \in [0,1[$]\label{thm:rate_strong}
Let $\cC\subseteq \R^n$ be a compact convex set, let $f\colon\cC\to\R$ be convex and differentiable in an open set containing $\cC$, let $\eta_t = \frac{\ell}{t+\ell}$ for some $\ell\in\N_{\geq 1}$ and all $t\in\N$, and suppose that $(\cC, f)$ satisfies the strong $(M, r)$-growth property for some $M >0$ and $r\in [0, 1[$.
Let
$\epsilon \in ]0, \ell[$.  Then, for the iterates of Algorithm~\ref{alg:fw}, $k=\min\{\ell-\epsilon,\frac{1}{1-r}\}$ and $t\in\N_{\geq 1}$, it holds that
\begin{align}
    \primaldual_t 
     & \leq \exp\srb{\epsilon\ell}\max \left\{ \eta_{t-1}^{\ell-\epsilon} \frac{M}{2}, \eta_{t-1}^k \srb{\frac{\ell M}{2\epsilon}}^{\frac{1}{1-r}} 
    \right\} \label{eq:rate_strong_2}\\
    &=\cO(t^{-\ell + \epsilon} + t^{-\frac{1}{1-r}} ). \nonumber
\end{align}
Let $c = \exp\left(\frac{1}{\min\{\ell-\epsilon, \frac{1}{1-r}\}}\right)$. Then, for the iterates of Algorithm~\ref{alg:fw}, $t\in\N$, and $t\geq \left\lceil c\cdot(1+\ell) - \ell \right\rceil$, it holds that
\begin{align*}
    \min_{i \in \{0, 1, \ldots, t\}} \gap_i 
        &\leq \exp\srb{3+\epsilon\ell}\eta_{t-1}^k \max \left\{\frac{M}{2},  \srb{\frac{\ell M}{2\epsilon}}^{\frac{1}{1-r}} 
        \right\} =\cO(t^{-\ell + \epsilon} + t^{-\frac{1}{1-r}} ).
\end{align*}
\end{theorem}

The proof of Theorem~\ref{thm:rate_strong} relies on the following technical lemma that will also be used in the proof of Theorem~\ref{thm:rate-weak} in Section~\ref{sec:weak_growth} below.
\begin{lemma}\label{lemma:telescope}
    Let  $\eta_t = \frac{\ell}{t+\ell}$ for some $\ell\in\N_{\geq 1}$ and all $t\in\N$. Then, for any $\fwt, t\in \N_{\geq 1}$ such that $\fwt \leq t$ and $\epsilon \in [0,\ell]$, it holds that 
    \begin{align*}
        \prod_{i=\fwt}^{t} \left(1-\left(1-\frac{\epsilon}{\ell}\right)\eta_i\right)\leq \left(\frac{\eta_t}{\eta_{\fwt-1}}\right)^{\ell-\epsilon} \exp\left(\frac{\epsilon\ell}{\fwt}\right).
    \end{align*}
\end{lemma}
\begin{proof}
    The proof is presented in Appendix~\ref{app:proof_lemma_telescope}
\end{proof}

\begin{proof}[Proof of Theorem~\ref{thm:rate_strong}]
We will prove a stronger statement.  We will show that for the iterates of Algorithm~\ref{alg:fw} and $t\in\N_{\geq 1}$, it holds that
\begin{equation}\label{eq:rate_strong}
\primaldual_t \leq  \max \left\{  \eta_{t-1}^{\ell-\epsilon}\frac{M}{2}\exp\srb{\epsilon\ell}, \max_{\Rfwt\in\{1, \ldots, t\}} \eta_{\Rfwt-1}^{\frac{1}{1-r}} \srb{\frac{\eta_{t-1}}{\eta_{\Rfwt-1}}}^{\ell-\epsilon}
\srb{\frac{\ell M}{2\epsilon}}^{\frac{1}{1-r}}
\exp\srb{\frac{\epsilon\ell}{\Rfwt}}
    \right\},
\end{equation}
which in turn implies~\eqref{eq:rate_strong_2} as we detail below.

We proceed by induction.  The statement~\eqref{eq:rate_strong} readily holds for $t = 1$. For the main inductive step, suppose that~\eqref{eq:rate_strong} holds for $t\in\N_{\geq1 }$. We distinguish between several cases:
\begin{enumerate}
    \item \label{case:good_progress} Suppose that
    \begin{align}\label{eq:smaller}
    \gap_{t} \leq \srb{\frac{\eta_{t}\ell M}{2\epsilon}}^{\frac{1}{1-r}}.
\end{align}
In that case, Inequality~\eqref{eq:various.gaps}, Inequality~\eqref{eq:gap}, and the strong $(M, r)$-growth property imply that 
\begin{align}\label{eq:good_growth}
    \primaldual_{t+1} &\leq (1- \eta_{t}) \gap_{t} + \frac{\eta_{t}^2M}{2}\gap_{t}^r \notag\\
    & \leq (1-\eta_{t})\srb{\frac{\eta_{t}\ell M}{2\epsilon}}^{\frac{1}{1-r}} + \eta_{t}\cdot \frac{\eta_t M}{2} \srb{\frac{\eta_{t}\ell M}{2\epsilon}}^{\frac{r}{1-r}}  \notag\\
    & \leq (1-\eta_{t})\srb{\frac{\eta_{t}\ell M}{2\epsilon}}^{\frac{1}{1-r}} + \eta_{t}\cdot \frac{\eta_t \ell M}{2\epsilon} \srb{\frac{\eta_{t}\ell M}{2\epsilon}}^{\frac{r}{1-r}}  \notag\\
    & =\srb{\frac{\eta_{t}\ell M}{2\epsilon}}^{\frac{1}{1-r}}.
\end{align}
Observe that 
\[
\srb{\frac{\eta_{t}\ell M}{2\epsilon}}^{\frac{1}{1-r}} \le 
\max_{\Rfwt\in\{\change{1}, \ldots, t+1\}} \eta_{\Rfwt-1}^{\frac{1}{1-r}} \srb{\frac{\eta_{t}}{\eta_{\Rfwt-1}}}^{\ell-\epsilon}
\srb{\frac{\ell M}{2\epsilon}}^{\frac{1}{1-r}}
\exp\srb{\frac{\epsilon\ell}{\Rfwt}}.
\]
Thus, \eqref{eq:rate_strong} holds for $t+1.$
\item \label{case:bad_progress} Suppose that \eqref{eq:smaller} does not hold. Let $\Rfwt\in \{\change{1}, \ldots, t\}$ be the smallest index such that
\begin{align}\label{eq:greater}
    \gap_{i} \geq \srb{\frac{\eta_{i}\ell M}{2\epsilon}}^{\frac{1}{1-r}}
\end{align}
for all $i\in\{\Rfwt, \Rfwt + 1, \ldots, t\}$.
Then, Inequality~\eqref{eq:various.gaps}, Inequality~\eqref{eq:gap}, and the strong  $(M, r)$-growth property imply that 
\begin{align}\label{eq:rec}
   \primaldual_{i+1} &\leq \primaldual_{i}+\eta_{i}\gap_{i}\srb{\frac{\eta_{i}M}{2\gap_{i}^{1-r}}-1} 
   \leq  \primaldual_{i} \left(1 -\srb{1-\frac{\epsilon}{\ell}} \eta_{i}\right)
\end{align}
for all $i\in \{\Rfwt, \Rfwt + 1, \ldots, t\}$.
By Lemma~\ref{lemma:telescope},
\begin{align}\label{eq:calc_big}
    \primaldual_{t+1} & \leq \primaldual_{\Rfwt}\prod_{i=\Rfwt}^{t} \left(1 - \srb{1-\frac{\epsilon}{\ell}}\eta_i\right) \leq \primaldual_\Rfwt \srb{\frac{\eta_{t}}{\eta_{\Rfwt-1}}}^{\ell-\epsilon}\exp\srb{\frac{\epsilon\ell}{\Rfwt}}.
\end{align}
We distinguish between two subcases:
\begin{enumerate}
\item \label{case:bad_progress.a}  $\Rfwt = 1$.  In this case, \eqref{eq:calc_big} and \eqref{eq.pd1} imply that
\begin{align}\label{eq:fwt_does_not_exist}
    \primaldual_{t+1} &\leq   \primaldual_{1} \eta_{t}^{\ell-\epsilon}\exp\srb{\epsilon\ell}
    \leq  \eta_{t}^{\ell-\epsilon}\frac{M}{2}\exp\srb{\epsilon\ell}.
\end{align}
\item \label{case:bad_progress.b} $\Rfwt > 1$. In this case, 
$
\gap_{\Rfwt-1} \leq \srb{\frac{\eta_{\Rfwt-1}\ell M}{2\epsilon}}^{\frac{1}{1-r}}
$.  Hence the bound \eqref{eq:good_growth} in Case~\ref{case:good_progress} can be applied 
to get that $\primaldual_{\Rfwt} \le \srb{\frac{ \eta_{\Rfwt-1}\ell M}{2\epsilon}}^{\frac{1}{1-r}}$
and thus~\eqref{eq:calc_big} implies that
\begin{align}\label{eq:fwt_exists}
    \primaldual_{t+1} & 
    \leq \eta_{\Rfwt-1}^{\frac{1}{1-r}} \srb{\frac{\eta_{t}}{\eta_{\Rfwt-1}}}^{\ell-\epsilon}\srb{\frac{ \ell M}{2\epsilon}}^{\frac{1}{1-r}}\exp\srb{\frac{\epsilon\ell}{\Rfwt}}.
\end{align}
\end{enumerate}
Thus, in both Case~\ref{case:bad_progress.a} and Case~\ref{case:bad_progress.b} the bound~\eqref{eq:rate_strong} holds for $t+1$.
\end{enumerate}
We next show that~\eqref{eq:rate_strong} implies~\eqref{eq:rate_strong_2}.  To that end, first observe that
\begin{align*}
\max_{\Rfwt\in\{1, \ldots, t\}} \eta_{\Rfwt-1}^{\frac{1}{1-r}}\srb{\frac{\eta_{t-1}}{\eta_{\Rfwt-1}}}^{\ell-\epsilon} &=
\max_{\Rfwt\in\{1, \ldots, t\}}\frac{\eta_{t-1}^{\ell-\epsilon}}{\eta_{\Rfwt-1}^{\ell-\epsilon -\frac{1}{1-r}}} \\
&= 
\left\{\begin{array}{ll} \eta_{t-1}^{\frac{1}{1-r}}, &  \text{ if } \ell - \epsilon \geq \frac{1}{1-r} \\
\eta_{t-1}^{\ell - \epsilon} , & \text{ if } \ell - \epsilon \leq \frac{1}{1-r}
\end{array}\right. \\
&\le \eta_{t-1}^{\min\{\ell-\epsilon, \frac{1}{1-r}\}}.
\end{align*}
Therefore \eqref{eq:rate_strong_2} follows from~\eqref{eq:rate_strong} and the fact that
\begin{align*}
    \max_{\Rfwt\in\{1, \ldots, t\}} \eta_{\Rfwt-1}^{\frac{1}{1-r}}\srb{\frac{\eta_{t-1}}{\eta_{\Rfwt-1}}}^{\ell-\epsilon}\srb{\frac{ \ell M}{2\epsilon}}^{\frac{1}{1-r}}\exp\srb{\frac{\epsilon\ell}{\Rfwt}}
    &\leq \max_{\Rfwt\in\{1, \ldots, t\}} \eta_{\Rfwt-1}^{\frac{1}{1-r}}\srb{\frac{\eta_{t-1}}{\eta_{\Rfwt-1}}}^{\ell-\epsilon}\srb{\frac{ \ell M}{2\epsilon}}^{\frac{1}{1-r}}\exp\srb{\epsilon\ell}\\
    &\le\eta_{t-1}^{\min\{\ell-\epsilon, \frac{1}{1-r}\}} \srb{\frac{\ell M}{2\epsilon}}^{\frac{1}{1-r}} \exp\srb{\epsilon\ell}.
\end{align*}

It remains to prove the convergence guarantee of $\min_{i\in\{0,1,\ldots,t\}}\gap_i$.
Recall that $k=\min\{\ell-\epsilon, \frac{1}{1-r}\}$.  Let $c:= \exp\left(\frac{1}{k}\right)$, let $t\geq \left\lceil c\cdot(1+\ell) - \ell \right\rceil$, and let $\Rfwt := \left\lfloor\frac{t+\ell}{c}-\ell\right\rfloor$. Then,
\begin{align*}
    c \cdot (\Rfwt + \ell+1) \geq  t+\ell \geq c \cdot (\Rfwt + \ell) \qquad \text{and} \qquad \Rfwt \geq  \left\lfloor\frac{\left\lceil c\cdot(1+\ell) - \ell\right\rceil +  \ell}{c} -\ell\right\rfloor \geq 1.
\end{align*}
Thus,
\begin{align}\label{eq:rate_strong_gap_cons}
    &\left(\frac{t+\ell-1}{\Rfwt+\ell-1}\right)^k \leq \left(c\cdot\frac{\Rfwt+\ell+1}{\Rfwt+\ell-1}\right)^k\leq \left(\exp\left(\frac{1}{k}\right)\right)^k\left(1+\frac{2}{\ell}\right)^k\leq \exp(3)\nonumber \\
    &\text{and} \qquad \log\left(\frac{t+\ell}{\Rfwt+\ell}\right) \geq \frac{1}{\ell-\epsilon}.
\end{align}
First, suppose that there exists a $j\in\{\Rfwt,\Rfwt+1,\ldots,t\}$ such that 
$
    \gap_j \leq \left(\frac{\eta_j\ell M}{2\epsilon}\right)^{\frac{1}{1-r}}.
$
Thus, by \eqref{eq:rate_strong_gap_cons}, it holds that
\begin{align*}
    \min_{i\in\{0,1,\ldots,t\}} \gap_i  \leq \gap_j \leq \left(\frac{\eta_j\ell M}{2\epsilon}\right)^{\frac{1}{1-r}} & \leq \left(\frac{\eta_{\Rfwt-1}\ell M}{2\epsilon}\right)^{\frac{1}{1-r}}\\
    & \leq \eta_{\Rfwt-1}^k\left(\frac{\ell M}{2\epsilon}\right)^{\frac{1}{1-r}}\\
    & \leq \left(\frac{t+\ell-1}{\Rfwt+\ell-1}\right)^k\eta_{t-1}^k\left(\frac{\ell M}{2\epsilon}\right)^{\frac{1}{1-r}}\\
    &\leq \exp\srb{3+\epsilon\ell}\eta_{t-1}^k \max \left\{\frac{M}{2},  \srb{\frac{\ell M}{2\epsilon}}^{\frac{1}{1-r}} 
    \right\}.\nonumber
\end{align*}
Second, suppose that instead
$
    \gap_i \geq \left(\frac{\eta_i\ell M}{2\epsilon}\right)^{\frac{1}{1-r}}
$
for all $i\in\{\Rfwt, \Rfwt+1, \ldots, t\}$. Then, by Inequality~\eqref{eq:various.gaps}, Inequality~\eqref{eq:gap}, and the strong $(M,r)$-growth property, it holds that
\[
    \primaldual_{i+1} \leq \primaldual_{i} - \left(1-\frac{M\eta_i}{2}\gap_i^{r-1}\right)\eta_i\gap_i \le
    \primaldual_{i} - \left(1-\frac{\epsilon}{\ell}\right) \eta_i \gap_i.
\]
Thus,
\begin{align*}
    \primaldual_{t+1} &\leq \primaldual_{\Rfwt} -  (\ell-\epsilon) \sum_{j=\Rfwt}^t\frac{1}{j+\ell} \cdot \min_{i \in \{\Rfwt, \Rfwt + 1, \ldots, t\}} \gap_i \\
    &\leq \primaldual_{\Rfwt} - (\ell-\epsilon) \log\left(\frac{t+\ell}{\Rfwt+\ell}\right) \cdot \min_{i \in \{\Rfwt, \Rfwt + 1, \ldots, t\}} \gap_i.
\end{align*}
Since $0 \leq \primaldual_{t+1}$, by \eqref{eq:rate_strong_2} and \eqref{eq:rate_strong_gap_cons}, it holds that
\begin{align*}
    \min_{i \in \{0, 1, \ldots, t\}} \gap_i &\leq \frac{\primaldual_{\Rfwt}}{(\ell-\epsilon) \log\left(\frac{t+\ell}{\Rfwt+\ell}\right)} \\
     &\leq \frac{\exp\srb{\epsilon\ell}}{(\ell-\epsilon) \log\left(\frac{t+\ell}{\Rfwt+\ell}\right)} \max \left\{ \eta_{\Rfwt-1}^{\ell-\epsilon} \frac{M}{2}, \eta_{\Rfwt-1}^k \srb{\frac{\ell M}{2\epsilon}}^{\frac{1}{1-r}} 
    \right\}\\
   &\leq \frac{\exp\srb{\epsilon\ell}\eta_{\Rfwt-1}^k}{(\ell-\epsilon) \log\left(\frac{t+\ell}{\Rfwt+\ell}\right)} \max \left\{\frac{M}{2} ,  \srb{\frac{\ell M}{2\epsilon}}^{\frac{1}{1-r}} 
    \right\}\\
    & = \frac{\left(\frac{t+\ell-1}{\Rfwt+\ell-1}\right)^k\exp\srb{\epsilon\ell}\eta_{t-1}^k}{(\ell-\epsilon) \log\left(\frac{t+\ell}{\Rfwt+\ell}\right)} \max \left\{\frac{M}{2} ,  \srb{\frac{\ell M}{2\epsilon}}^{\frac{1}{1-r}} 
    \right\}\\
    &\leq \exp\srb{3+\epsilon\ell}\eta_{t-1}^k \max \left\{\frac{M}{2} ,  \srb{\frac{\ell M}{2\epsilon}}^{\frac{1}{1-r}} 
    \right\},
\end{align*}
where the last equality follows from  
$\left(\frac{t+\ell-1}{\Rfwt+\ell-1}\right)^k\eta_{t-1}^k  = \left(\frac{\ell}{\Rfwt+\ell-1}\right)^k = \eta_{\Rfwt-1}^k$ and the last inequality follows from \eqref{eq:rate_strong_gap_cons}.
\end{proof}

The results derived in this section characterize the acceleration of \fw{} with open-loop step-sizes when the strong $(M, r)$-growth property is satisfied. It remains to discuss sufficient conditions for the strong $(M, r)$-growth property to be satisfied.

\subsection{Sufficient conditions for strong $(M, r)$-growth}\label{sec.suff_strong_M_r}
A simple setting for which the strong $(M,r)$-growth property holds is when the feasible region is uniformly convex and the norm of the gradient of the objective is bounded from below by a positive constant. We recall the following result by \cite{pena2023affine}.  Although this result holds for general norms, we restrict ourselves to the intrinsic norm $\|\cdot\|_{\cC}$ to guarantee affine invariance.

\begin{proposition}[Proposition~3.3, \cite{pena2023affine}]\label{prop:exterior}
    For $\alpha > 0$ and $p \geq 2$, let $\cC \subseteq \R^n$ be a compact $(\alpha, p)$-uniformly convex set with respect to $\|\cdot\|_{\cC}$, for $L>0$ and $\lambda>0$, let $f\colon\cC\to\R$ be a convex and $L$-smooth function with respect to $\|\cdot\|_{\cC}$ with lower-bounded gradients, that is, $\|\nabla f(\xx) \|_{\cC, *} \geq \lambda$ for all $\xx\in \cC$. Then,
    $(\cC,f)$ satisfies the strong $(M,r)$-growth property with $M = L (\alpha\lambda)^{-\frac{2}{p}}$ and $r = \frac{2}{p}$.
\end{proposition}
The results established in this section address several open questions raised in \cite{wirth2023acceleration} and significantly advance the understanding of \fw{} with $\eta_t = \frac{\ell}{t+\ell}$ for $\ell\in\N_{\geq 1}$. In particular, our rates are affine-invariant, faster than those in \cite{wirth2023acceleration}, and in $\primaldual_t$, which bounds $\subopt_t$ from above. Furthermore, we derive bounds on the dual gap.

In the setting of Proposition~\ref{prop:exterior}, when the feasible region is strongly convex, Theorem~D.2 and Remark~D.3 in \cite{wirth2023acceleration} yields rates of order $\subopt_t = \cO(t^{-\ell/2})$. 
In this strong $(M, 1)$-growth setting, Theorem~\ref{thm:1_strong} yields rates of order $\primaldual_t = \cO(t^{-\ell})$ and $\min_{i\in\{1,\ldots, t\}} \gap_i = \cO(t^{-\ell})$.

In the setting of Proposition~\ref{prop:exterior}, when the feasible region is uniformly convex, Theorem~D.2 in \cite{wirth2023acceleration} yields rates of order $\subopt_t = \cO(t^{-2} + t^{-\frac{1}{1-r}})$ for \fw{} with $\eta_t = \frac{4}{t+4}$. 
In this strong $(M, r)$-growth setting, Theorem~\ref{thm:rate_strong} yields rates of order $\primaldual_t = \cO(t^{-\ell+\epsilon} + t^{-\frac{1}{1-r}})$ and $\min_{i\in\{1,\ldots, t\}} \gap_i = \cO(t^{-\ell+\epsilon} + t^{-\frac{1}{1-r}})$ for any $\epsilon > 0$. Thus, our analysis demonstrates that for small $r$, larger $\ell\in\N$ leads to faster convergence, an intricacy not captured by \cite{wirth2023acceleration}.

\subsection{Numerical experiments}\label{sec:strong_growth.examples}
\noindent
Proposition~\ref{prop:exterior} motivates the following example for which the strong $(M,r)$-growth property holds.  For ease of exposition, we take $\cC$ to be the unit $\ell_p$-norm ball for some $p >1$.  However, Proposition~\ref{prop:exterior} applies to the much broader class of uniformly convex domains.
\begin{example}[Strong $(M, r)$-growth]\label{ex:strong}
For $p> 1$, let $\cC = \{\xx\in\R^n \mid \|\xx\|_p \leq 1\}$ be an $\ell_p$-ball. 
Let $q\geq 1$ be such that $\frac{1}{p}+\frac{1}{q}=1$.
Let $\yy \in \R^n$ and $f(\xx):= \frac{1}{2}\|\xx-\yy\|_2^2$. 
We distinguish between two cases:
\begin{enumerate}
    \item When $p \in ]1, 2]$, 
    by Proposition~\ref{prop:unif_cvx}, $\cC$ is $(p-1)/4$-strongly convex with respect to $\|\cdot\|_{\cC} = \|\cdot\|_p$.
    Let $\yy$ be such that $\|\yy\|_q = 1 + \lambda$ for some $\lambda > 0$. Then, $\|\nabla f(\xx)\|_q = \|\xx - \yy\|_q \geq  \abs{\|\yy\|_q - \|\xx\|_q} \geq \|\yy\|_q - \|\xx\|_p = \lambda > 0$ for all $\xx\in\cC$.
    By Lemma~\ref{lemma:lp_balls}, $f$ is $1$-smooth with respect to $\|\cdot\|_{\cC}  $.
    Thus, by Proposition~\ref{prop:exterior}, $(\cC, f)$ satisfies the strong $(M, r)$-growth property with $M= \frac{4}{(p-1)\lambda}$ and $r=1$.
    \item When $p\geq 2$, 
    by Proposition~\ref{prop:unif_cvx}, $\cC$ is $(1/(p2^{p-1}), p)$-uniformly convex with respect to $\|\cdot\|_{\cC} = \|\cdot\|_p$.
    In addition, $q\le 2\le p.$
    Let $\yy$ be such that $\|\yy\|_q = n^{\frac{1}{q}-\frac{1}{p}} + \lambda$ for some $\lambda > 0$. Then, $\|\nabla f(\xx)\|_q = \|\xx - \yy\|_q \geq  \abs{\|\yy\|_q - \|\xx\|_q} \geq \|\yy\|_q - n^{\frac{1}{q}-\frac{1}{p}} = \lambda > 0$ for all $\xx\in\cC$. 
    By Lemma~\ref{lemma:lp_balls}, $f$ is $n^{\frac{1}{2}-\frac{1}{p}}$-smooth with respect to $\|\cdot\|_{\cC}$.
    Thus, by Proposition~\ref{prop:exterior}, $(\cC, f)$ satisfies the strong $(M, r)$-growth property with $M= n^{\frac{1}{2}-\frac{1}{p}} (\frac{p2^{p-1}}{\lambda} )^{\frac{2}{p}}$ and $r=\frac{2}{p}$.
\end{enumerate}
\end{example}

\subsubsection{Strong $(M, 1)$-growth experiment}\label{sec:strong_growth_setting}

\begin{figure}[t]
\captionsetup[subfigure]{justification=centering}
\begin{tabular}{c c c c} 
\begin{subfigure}{.31\textwidth}
    \centering
        \includegraphics[width=1\textwidth]{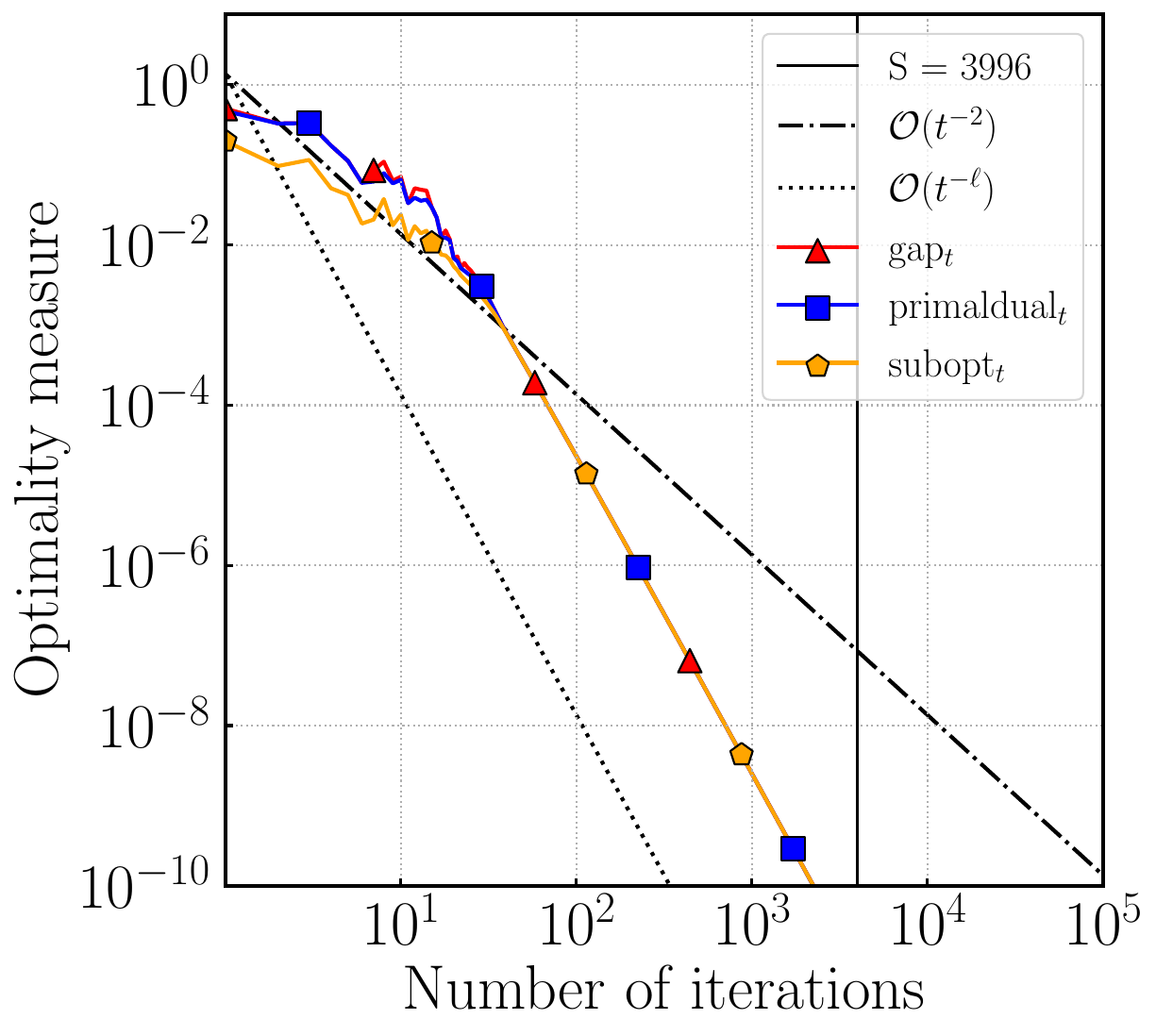}
        \caption{$\ell_{1.01}$-ball, $r=1$.}\label{fig:1_strong_1.01}
    \end{subfigure} & 
    \begin{subfigure}{.31\textwidth}
    \centering
        \includegraphics[width=1\textwidth]{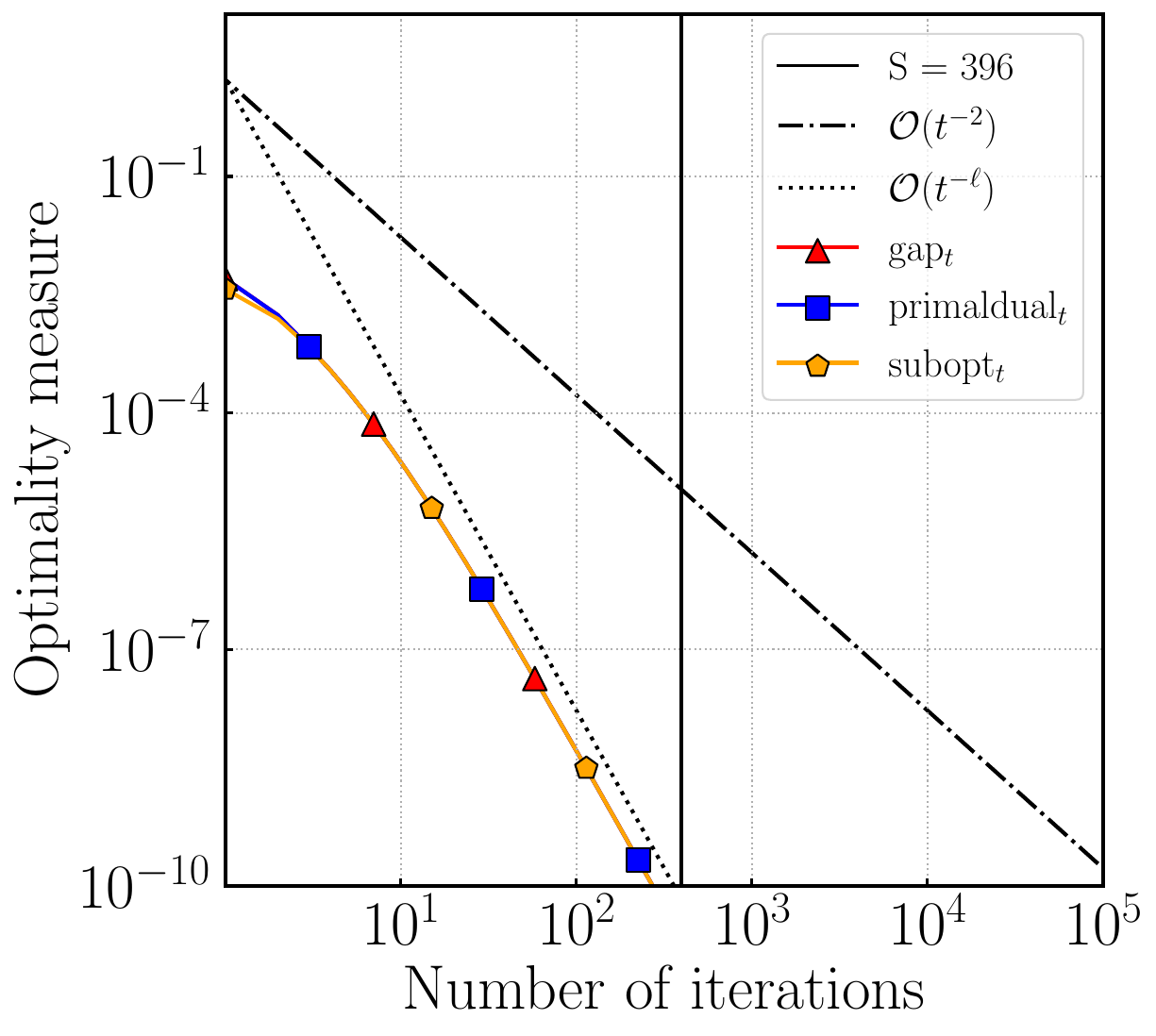}
        \caption{$\ell_{1.1}$-ball, $r=1$.}\label{fig:1_strong_1.1}
    \end{subfigure} & 
\begin{subfigure}{.31\textwidth}
    \centering
        \includegraphics[width=1\textwidth]{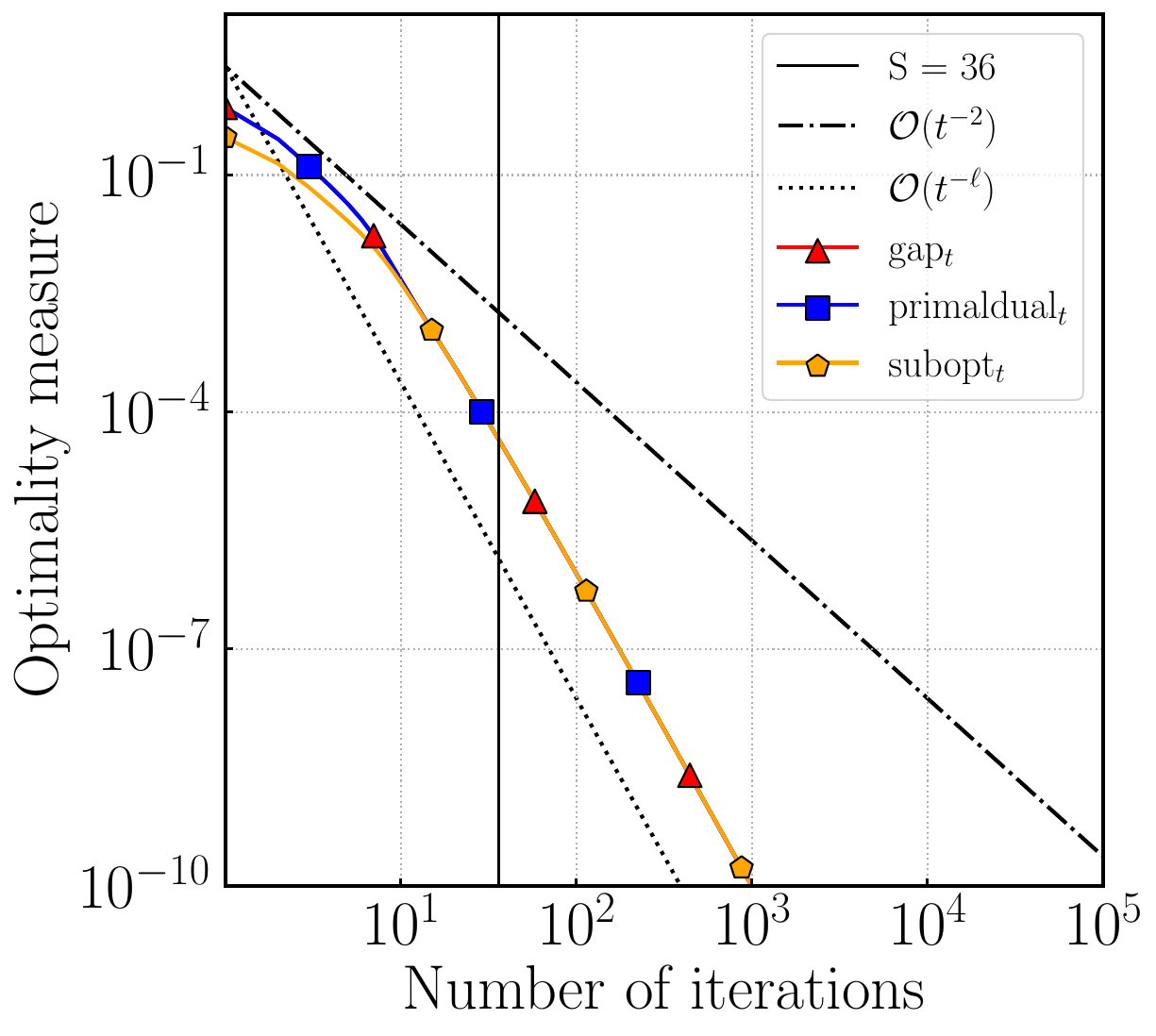}
        \caption{$\ell_{2}$-ball, $r=1$.}\label{fig:1_strong_2}
    \end{subfigure}
\end{tabular}
\caption{\textbf{Strong $(M, 1)$-growth.} Optimality measure comparison of \fw{} with step-size $\eta_t = \frac{\ell}{t+\ell}$ for $\ell = 4$ and all $t\in\N$ when the feasible region $\cC\subseteq\R^n$ for $n = 100$ is an $\ell_p$-ball and the objective is $f(\xx) =  \frac{1}{2}\|\xx - \yy\|_2^2$ for some random vector $\yy\in\R^n$ with $\|\yy\|_q = 1 + \lambda$, where $q$ is such that $\frac{1}{p}+\frac{1}{q}=1$ and $\lambda = 0.2$. Axes are in log scale.
}\label{fig:1_strong}
\end{figure}
In Figure~\ref{fig:1_strong}, in the setting of Example~\ref{ex:strong} for $n = 100$, $p \in \{1.01, 1.1, 2\}$, $\lambda=0.2$, and \fw{} with open-loop step-size $\eta_t = \frac{\ell}{t+\ell}$ for $\ell = 4$ and all $t\in\N$, we compare $\gap_t$, $\primaldual_t$, and $\subopt_t$. We also plot $\cO(t^{-\ell})$ and $\cO(t^{-2})$ for better visualization.
For these settings, strong $(M, 1)$-growth holds and Theorem~\ref{thm:1_strong} applies. Since Theorem~\ref{thm:1_strong} only predicts acceleration after a problem-dependent iteration $\fwt$ has been reached, we also add a vertical line indicating $\fwt$ to the plots.

We observe the expected rate of order $\cO(t^{-\ell})$ for $\primaldual_t$, $\subopt_t$, and $\min_{i\in\{0,1,\ldots, t\}} \gap_i$ for all three settings.
\subsubsection{Strong $(M, r)$-growth experiment}
\begin{figure}[t]
\captionsetup[subfigure]{justification=centering}
\begin{tabular}{c c c}
    \begin{subfigure}{.31\textwidth}
    \centering
        \includegraphics[width=1\textwidth]{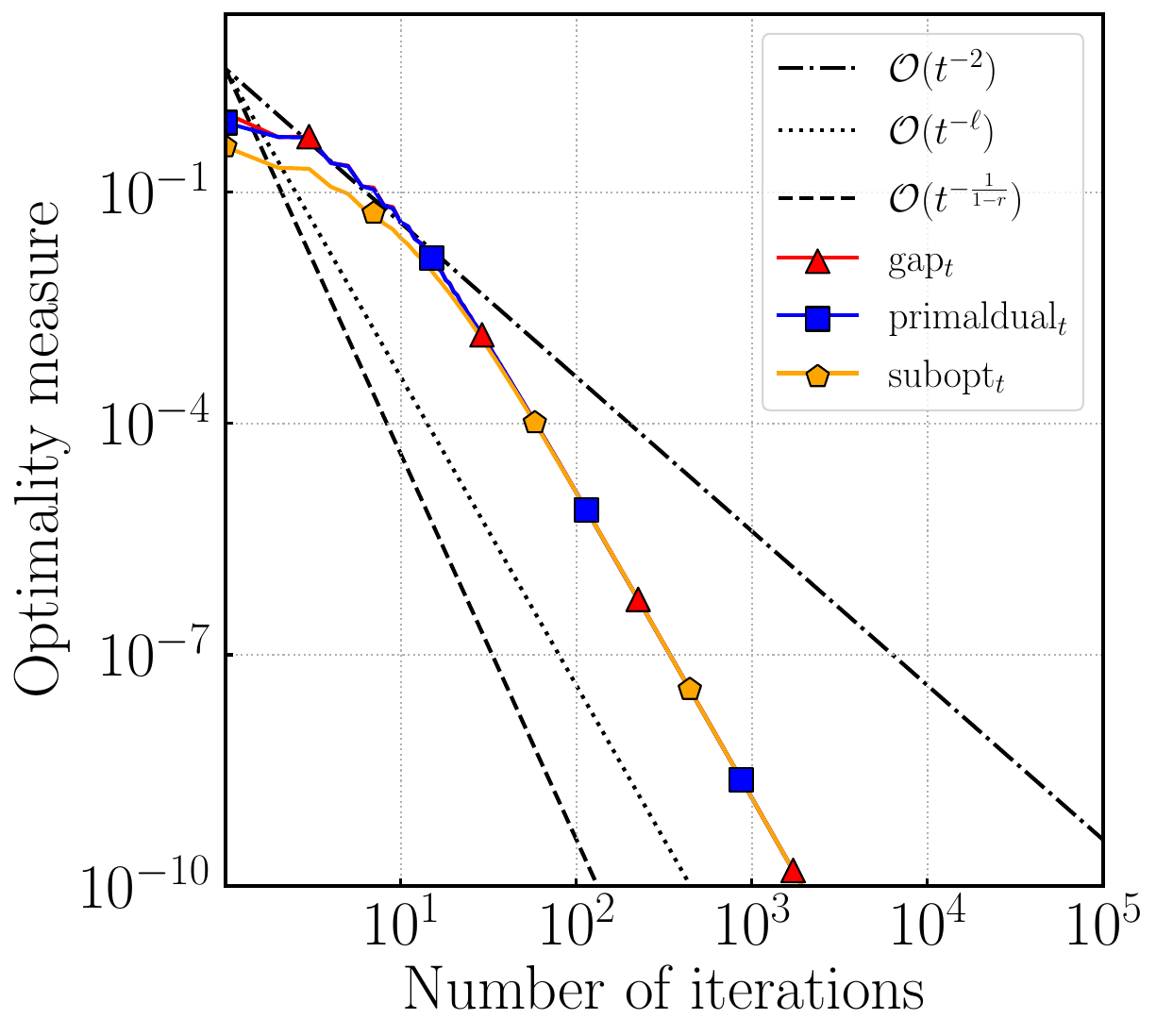}
        \caption{$\ell_{2.5}$-ball, $r=4/5$.}\label{fig:strong_2.5}
    \end{subfigure} & 
    \begin{subfigure}{.31\textwidth}
    \centering
        \includegraphics[width=1\textwidth]{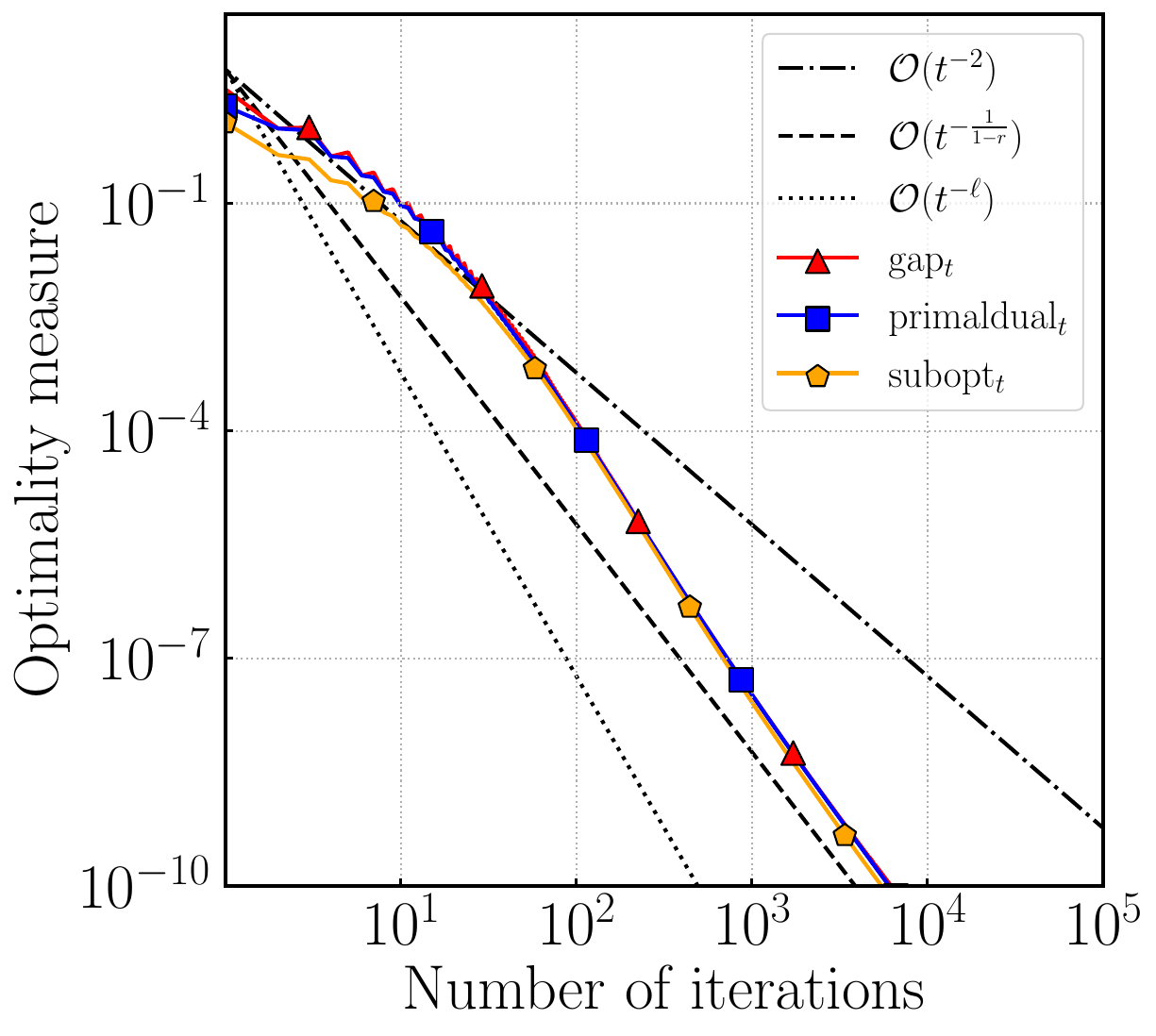}
        \caption{$\ell_3$-ball, $r=2/3$.}\label{fig:strong_3}
    \end{subfigure} & 
    \begin{subfigure}{.31\textwidth}
    \centering
        \includegraphics[width=1\textwidth]{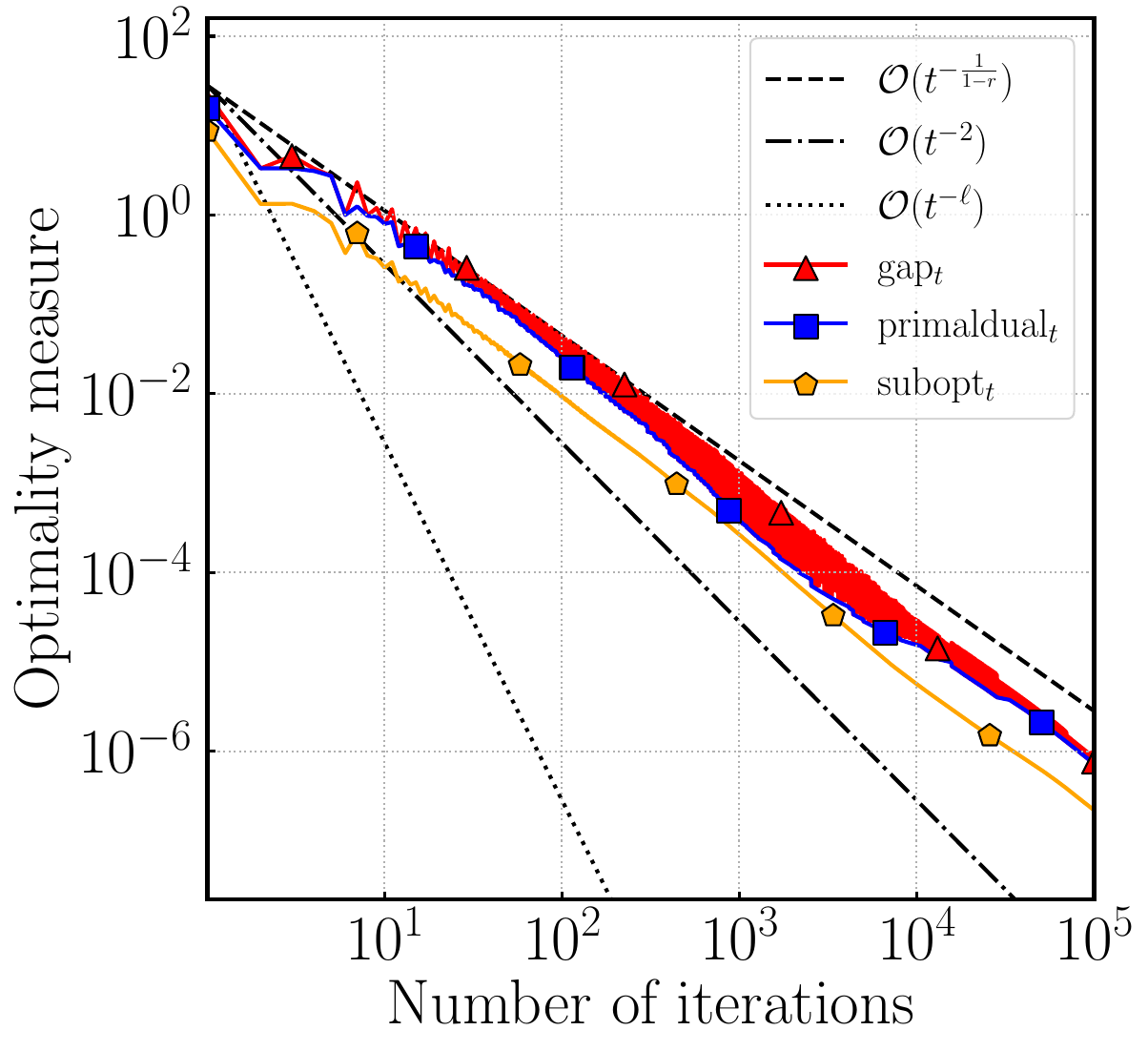}
        \caption{$\ell_7$-ball, $r=2/7$.}\label{fig:strong_7}
    \end{subfigure} 
\end{tabular}
\caption{\textbf{Strong $(M, r)$-growth.} 
Optimality measure comparison of \fw{} with  step-size $\eta_t = \frac{\ell}{t+\ell}$ for $\ell = 4$ and all $t\in\N$ when the feasible region $\cC\subseteq\R^n$ for $n = 100$ is an $\ell_p$-ball and the objective is $f(\xx) =  \frac{1}{2}\|\xx - \yy\|_2^2$ for some random vector $\yy\in\R^n$ with $\|\yy\|_q = n^{\frac{1}{q}-\frac{1}{p}} + \lambda$, where $q$ is such that $\frac{1}{p}+\frac{1}{q}=1$ and $\lambda = 0.2$. Axes are in log scale.
}\label{fig:strong}
\end{figure}
In Figure~\ref{fig:strong}, in the setting of Example~\ref{ex:strong} for $n = 100$, $p \in \{2.5, 3, 7\}$, $\lambda=0.2$, and \fw{} with open-loop step-size $\eta_t = \frac{\ell}{t+\ell}$ for $\ell = 4$ and all $t\in\N$, we compare $\gap_t$, $\primaldual_t$, and $\subopt_t$. We also plot $\cO(t^{-\ell})$, $\cO(t^{-\frac{1}{1-r}})$, and $\cO(t^{-2})$ for better visualization.
For these settings, strong $(M, r)$-growth, $r<1$, holds and Theorem~\ref{thm:rate_strong} applies. 
%

%
We observe the expected rate of order $\cO(t^{-\ell} + t^{-\frac{1}{1-r}})$ for $\primaldual_t$, $\subopt_t$, and $\min_{i\in\{0,1,\ldots, t\}} \gap_i$ for all three settings.

\section{Strong $(M, 0)$-growth and weak $(M, r)$-growth}
\label{sec:weak_growth}
This section presents affine-invariant convergence rates when the strong $(M, 0)$-growth and weak $(M, r)$-growth properties hold in Theorem~\ref{thm:rate-weak}.  As we detail in Section~\ref{sec.suff_weak_M_r} below, the results in this section extend those in \cite{wirth2023acceleration} to a broader selection of open-loop step-sizes.

\subsection{Convergence rates}\label{sec:weak_growth.results}
The following result is similar in spirit to Theorem~\ref{thm:rate_strong}.  However, instead of relying on the strong $(M, r)$-growth property, it relies on the typically more relaxed weak $(M, r)$-growth property.  In contrast to Theorem~\ref{thm:rate_strong} that bounds the convergence rate of the best duality gap, Theorem~\ref{thm:rate-weak} bounds the suboptimality gap and the rate is capped at $\cO(t^{-2})$. This is the tradeoff for the more relaxed growth property assumption in Theorem~\ref{thm:rate-weak}.  


\begin{theorem}[Strong $(M, 0)$-growth and weak $(M, r)$-growth]\label{thm:rate-weak}
Let $\cC\subseteq \R^n$ be a compact convex set, let $f\colon\cC\to\R$ be convex and differentiable in an open set containing $\cC$, let $\eta_t = \frac{\ell}{t+\ell}$ for some $\ell\in\N_{\geq 1}$ and all $t\in\N$, and suppose that $(\cC, f)$ satisfies the strong $(M,0)$- and weak $(M, r)$-growth properties for some $M >0$ and $r\in [0, 1[$.
Let
$\epsilon \in ]0, \ell[$.
Then, for the iterates of Algorithm~\ref{alg:fw} and $t\in\N_{\geq 1}$, it holds that
\begin{align}
    \subopt_t 
    & \leq \exp\srb{\epsilon\ell}  \max \left\{\eta_{t-1}^{\ell-\epsilon} \frac{M}{2}, \eta_{t-1}^{\min\{\ell-\epsilon,\frac{1}{1-r},2\}}\srb{\srb{\frac{\ell M}{2\epsilon}}^{\frac{1}{1-r}} +\frac{M}{2}}
    \right\} \label{eq:rate_weak_2}\\
    & = \cO\left(t^{-\ell+\epsilon}+t^{-\frac{1}{1-r}}+t^{-2}\right).\nonumber
\end{align}
Furthermore, either there  exists a subsequence $(t_i)_{i\in\N}$ for which
\begin{align}\label{eq:l_faster}
    \subopt_{t_i} 
    \leq \srb{\frac{\eta_{t_i}\ell  M}{2\epsilon}}^{\frac{1}{1-r}},
\end{align}
or there exists an iteration $\Rfwt\in\N$ such that for all $t\in\N_{\geq\Rfwt+1}$, it holds that
\begin{align}\label{eq:l_slower}
\subopt_t \leq \subopt_\Rfwt \srb{\frac{\eta_{t-1}}{\eta_{\Rfwt -1}}}^{\ell-\epsilon} \exp\srb{\frac{\epsilon\ell}{\Rfwt}}.
\end{align} 
\end{theorem}
\begin{proof}
This proof is very similar to that of Theorem~\ref{thm:rate_strong} but a few minor tweaks are necessary and we restrict ourselves to presenting those. 

We will prove a stronger statement.  We will show that for the iterates of Algorithm~\ref{alg:fw} and $t\in\N_{\geq 1}$, it holds that
\begin{align}\label{eq:rate_weak}
    \subopt_t \leq \max &\left\{\eta_{t-1}^{\ell-\epsilon}\frac{M}{2}\exp\srb{\epsilon\ell}, \right.\nonumber\\
    &\left. \max_{\Rfwt\in\{1, \ldots, t\}}
\eta_{R-1}^{\min\{\frac{1}{1-r},2\}}\srb{\frac{\eta_{t-1}}{\eta_{\Rfwt -1}}}^{\ell-\epsilon}\srb{\srb{\frac{\ell M}{2\epsilon}}^{\frac{1}{1-r}} +\frac{M}{2}}
 \exp\srb{\frac{\epsilon\ell}{\Rfwt}}    
    \right\},
\end{align}
which in turn implies~\eqref{eq:rate_weak_2} as we detail below.
We next prove~\eqref{eq:rate_weak} by induction.  The bound~\eqref{eq:rate_weak} readily holds for $t=1$. For the main inductive  step, we distinguish between two main cases:
\begin{enumerate}
    \item \label{case:good_progress.weak} Suppose that
    \begin{align}\label{eq:smaller.weak}
    \subopt_{t} \leq \srb{\frac{\eta_{t}\ell  M}{2\epsilon}}^{\frac{1}{1-r}}.
\end{align}
Hence, Equality~\eqref{eq:fw-step} and the strong $(M,0)$-growth property imply that
\begin{align}
 \label{eq:good_growth.weak}
    \subopt_{t+1} &\leq \subopt_{t} + \frac{\eta_{t}^2M}{2} \nonumber\\
    &\leq \srb{\frac{\eta_{t}\ell M}{2\epsilon}}^{\frac{1}{1-r}}+\frac{\eta_{t}^2M}{2}\nonumber\\
    &\le \eta_{t}^{\min\{\frac{1}{1-r},2\}}\srb{\srb{\frac{\ell M}{2\epsilon}}^{\frac{1}{1-r}} +\frac{M}{2}}.
\end{align}
Observe that
\begin{align*}
    &\eta_{t}^{\min\{\frac{1}{1-r},2\}}\srb{\srb{\frac{\ell M}{2\epsilon}}^{\frac{1}{1-r}} +\frac{M}{2}} \\
    &\le \max_{\Rfwt\in\{1, \ldots, t+1\}}
\eta_{R-1}^{\min\{\frac{1}{1-r},2\}}\srb{\frac{\eta_{t}}{\eta_{\Rfwt -1}}}^{\ell-\epsilon}\srb{\srb{\frac{\ell M}{2\epsilon}}^{\frac{1}{1-r}} +\frac{M}{2}}
 \exp\srb{\frac{\epsilon\ell}{\Rfwt}}.
\end{align*}
Thus, \eqref{eq:rate_weak} holds for $t+1.$
\item 
\label{case:bad_progress.weak} Suppose that \eqref{eq:smaller.weak} does not hold. Let $\Rfwt\in \{1, \ldots, t\}$ be the smallest index such that
\begin{align}\label{eq:greater.weak}
    \subopt_{i} \geq \srb{\frac{\eta_{i}\ell M}{2\epsilon}}^{\frac{1}{1-r}}
\end{align}
for all $i\in\{\Rfwt, \Rfwt + 1, \ldots, t\}$.
Then, Equality~\eqref{eq:fw-step}, 
Inequality~\eqref{eq:various.gaps}, and the
the weak $(M, r)$-growth property  imply that 
\begin{align}\label{eq:rec.weak}
   \subopt_{i+1} \leq \subopt_{i}+\eta_{i}\gap_{i}\srb{\frac{\eta_{i}M}{2\subopt_{i}^{1-r}}-1}
   \leq  \subopt_{i} \left(1 - \srb{1-\frac{\epsilon}{\ell}}\eta_{i}\right)
\end{align}
for all $i\in \{\Rfwt, \Rfwt + 1, \ldots, t\}$.
By repeatedly applying \eqref{eq:rec.weak} and by Lemma~\ref{lemma:telescope},
\begin{align}\label{eq:calc_big.weak}
    \subopt_{t+1} & \leq \subopt_{\Rfwt}\prod_{i=\Rfwt}^{t} \left(1 - \srb{1-\frac{\epsilon}{\ell}}\eta_i\right) \leq \subopt_\Rfwt \srb{\frac{\eta_t}{\eta_{\Rfwt -1}}}^{\ell-\epsilon} \exp\srb{\frac{\epsilon\ell}{\Rfwt}}.
\end{align}
Next we proceed as in Case~\ref{case:bad_progress} in Theorem~\ref{thm:rate_strong}.  If $R=1$ then~\eqref{eq:calc_big.weak}, \eqref{eq.pd1}, and Inequality~\eqref{eq:various.gaps} imply that
\[
\subopt_{t+1} \le \subopt_1 \eta_t^{\ell-\epsilon} \exp\srb{\epsilon\ell} \leq \eta_t^{\ell-\epsilon} \frac{M}{2}\exp\srb{\epsilon\ell},
\]
and if $R>1$ then the bound~\eqref{eq:good_growth.weak} in Case~\ref{case:good_progress.weak} and~\eqref{eq:calc_big.weak} imply that
\[
\subopt_{t+1} \le \eta_{R-1}^{\min\{\frac{1}{1-r},2\}}\srb{\frac{\eta_t}{\eta_{\Rfwt -1}}}^{\ell-\epsilon}\srb{\srb{\frac{\ell M}{2\epsilon}}^{\frac{1}{1-r}} +\frac{M}{2}}
 \exp\srb{\frac{\epsilon\ell}{\Rfwt}}.
\]
Thus, \eqref{eq:rate_weak} holds for $t+1$.
\end{enumerate}
We next show that~\eqref{eq:rate_weak} implies~\eqref{eq:rate_weak_2}.  Indeed, by proceeding as in the proof of Theorem~\ref{thm:rate_strong}, we have that
\begin{align*}
    &\max_{\Rfwt\in\{1, \ldots, t\}}
\eta_{R-1}^{\min\{\frac{1}{1-r},2\}}\srb{\frac{\eta_{t-1}}{\eta_{\Rfwt -1}}}^{\ell-\epsilon}\srb{\srb{\frac{\ell M}{2\epsilon}}^{\frac{1}{1-r}} +\frac{M}{2}}
 \exp\srb{\frac{\epsilon\ell}{\Rfwt}} \\
 & \le 
\eta_{t-1}^{\min\{\ell-\epsilon,\frac{1}{1-r},2\}}\srb{\srb{\frac{\ell M}{2\epsilon}}^{\frac{1}{1-r}} +\frac{M}{2}}
 \exp\srb{\epsilon\ell}.
\end{align*}
Thus, \eqref{eq:rate_weak_2} follows from~\eqref{eq:rate_weak}.
To prove the second statement of the theorem, suppose that there does not exist a subsequence $(t_i)_{i\in\N}$ for which \eqref{eq:l_faster} holds. Thus, there has to exist some iteration $\Rfwt\in\N_{\geq 1}$ such that 
$
    \subopt_t \geq \rb{\frac{\eta_{t-1}\ell M}{2\epsilon}}^{\frac{1}{1-r}}
$
for all $t\in\N_{\geq \Rfwt+1}$. Then, \eqref{eq:calc_big.weak} implies that \eqref{eq:l_slower} holds for all $t\in\N_{\geq \Rfwt+1}$.
\end{proof}

The astute reader may wonder if the analyses in Theorem~\ref{thm:1_strong} and Theorem~\ref{thm:rate_strong} could be adapted to derive a bound on $\min_{i\in\{0, 1, \ldots, t\}} \gap_i$ in the setting of Theorem~\ref{thm:rate-weak}.  We have not been able to do so and suspect that this is an inevitable tradeoff of relying on the more relaxed weak growth property.  Furthermore, the numerical experiments discussed in Section~\ref{sec:gaps_growth.examples} and in Section~\ref{sec:constrained regression} suggest that this is not possible.  Indeed, Figure~\ref{fig:gaps}(b) in Section~\ref{sec:gaps_growth.examples} and Figure~\ref{fig:reg}(e,f) in Section~\ref{sec:constrained regression} show experiments where $\subopt_t$ converges to zero at accelerated rate $\cO(t^{-2})$ while both $\primaldual_t$ and $\min_{i\in\{0, 1, \ldots, t\}} \gap_i$ converge to zero at the slower rate  $\cO(t^{-1})$.

\medskip

The following result complements the second statement in Theorem~\ref{thm:rate-weak} when $\ell$ is sufficiently large.  It gives a bound on the number of iterations before the rate~\eqref{eq:l_faster} is next observed.

\begin{lemma}\label{lemma:rate-weak}
Let $\cC\subseteq \R^n$ be a compact convex set, let $f\colon\cC\to\R$ be convex and differentiable in an open set containing $\cC$,  let $\eta_t = \frac{\ell}{t+\ell}$ for some $\ell\in\N_{\geq 1}$ and all $t\in\N$, and suppose that $(\cC, f)$ satisfies the strong $(M,0)$-growth property and the weak $(M, r)$-growth property for some $M >0$ and  $r\in [0, 1[$ with $\ell > \frac{1}{1-r}$. Let
$\epsilon \in ]0, \ell-\frac{1}{1-r}]$.
Then, for any $\Rfwt\in\N_{\geq 1}$ 
 there exists an $i \in \{\Rfwt, \Rfwt + 1, \ldots, U\}$ such that 
 $
    \subopt_{i} 
    \leq \srb{\frac{\eta_{i}\ell M}{2\epsilon} }^{\frac{1}{1-r}}
$, where
\begin{align*}
    U = \left\lceil\ell\left(\subopt_\Rfwt \eta_{\Rfwt}^{\epsilon -\ell} \srb{\frac{2\epsilon}{\ell M}}^{\frac{1}{1-r}} \exp\left(\frac{\epsilon \ell}{\Rfwt}\right)\right)^{\frac{1}{\ell-\epsilon-\frac{1}{1-r}}} -\ell  \right\rceil.
\end{align*}
\end{lemma}
\begin{proof}
Suppose towards a contradiction that there exists $\Rfwt\in\N_{\geq 1}$ such that
$
    \subopt_{i} 
    > \srb{\frac{\eta_{i}\ell M}{2\epsilon} }^{\frac{1}{1-r}}
$
for all $i \in \{\Rfwt, \Rfwt + 1, \ldots, U\}$. By the construction of $U$, it holds that
\begin{align*}
     & \left\lceil\ell \left(\subopt_\Rfwt \eta_{\Rfwt}^{\epsilon -\ell} \srb{\frac{2\epsilon}{\ell M}}^{\frac{1}{1-r}} \exp\left(\frac{\epsilon \ell}{\Rfwt}\right)\right)^{\frac{1}{\ell-\epsilon-\frac{1}{1-r}}} -\ell \right\rceil =  U\\
    \Rightarrow \qquad & \subopt_\Rfwt \eta_{\Rfwt}^{\epsilon -\ell} \srb{\frac{2\epsilon}{\ell M}}^{\frac{1}{1-r}} \exp\left(\frac{\epsilon \ell}{\Rfwt}\right) \leq  \eta_{U}^{\frac{1}{1-r}-\ell+\epsilon}  \\
    \Rightarrow \qquad & \subopt_\Rfwt \srb{\frac{\eta_{U}}{\eta_{\Rfwt}}}^{\ell-\epsilon}\exp\left(\frac{\epsilon\ell}{\Rfwt}\right) \leq  \srb{\frac{\eta_{U}\ell M}{2\epsilon} }^{\frac{1}{1-r}}.
    \end{align*}
By \eqref{eq:calc_big.weak}, it thus holds that
\begin{align*}
    \subopt_U  \leq \subopt_\Rfwt \srb{\frac{\eta_{U-1}}{\eta_{\Rfwt-1}}}^{\ell-\epsilon}\exp\left(\frac{\epsilon\ell}{\Rfwt}\right) \le 
    \subopt_\Rfwt \srb{\frac{\eta_{U}}{\eta_{\Rfwt}}}^{\ell-\epsilon}\exp\left(\frac{\epsilon\ell}{\Rfwt}\right)\leq  \srb{\frac{\eta_{U}\ell M}{2\epsilon} }^{\frac{1}{1-r}},
\end{align*}
a contradiction, proving the lemma.
\end{proof}
The results derived in this section characterize the acceleration of \fw{} with open-loop step-sizes when the weak $(M, r)$-growth property is satisfied. It remains to discuss sufficient conditions for the weak $(M, r)$-growth property to be satisfied.

\subsection{Sufficient conditions for weak $(M,r)$-growth}\label{sec.suff_weak_M_r}
A simple setting for which the weak $(M, r)$-growth property holds is when the feasible region is uniformly convex and the objective satisfies a Hölderian error bound. We recall the following result by \cite{pena2023affine}. 
Although this result holds for general norms, we restrict ourselves to the intrinsic norm $\|\cdot\|_{\cC}$ to guarantee affine invariance.

\begin{proposition}[Proposition~4.3, \cite{pena2023affine}]\label{prop:unif_heb}
    For $\alpha > 0$ and $p \geq 2$, let $\cC \subseteq \R^n$ be a compact $(\alpha, p)$-uniformly convex set with respect to $\|\cdot\|_{\cC}$, for $L>0$, $\mu > 0$, and $\theta \in [0, 1/2]$, let $f\colon\cC\to\R$ be a convex and $L$-smooth function satisfying a $(\mu, \theta)$-Hölderian error bound with respect to $\|\cdot\|_{\cC}$. Then,
    $(\cC,f)$ satisfies the weak $(M,r)$-growth property with $M = L \left(\frac{\mu}{\alpha}\right)^{\frac{2}{p}}$ and $r = \frac{2\theta}{p}$.
\end{proposition}

The results in this section improve over prior work in \cite{wirth2023acceleration} as follows: In the setting of Proposition~\ref{prop:unif_heb}, Theorem~E.1 in \cite{wirth2023acceleration} derived affine-dependent rates of order $\subopt_t = \cO(t^{-\frac{1}{1-r}}+t^{-2})$ for \fw{} with $\eta_t = \frac{4}{t+4}$. In contrast, Theorem~\ref{thm:rate-weak} yields affine-invariant rates of order $\subopt_t = \cO(t^{-\ell+\epsilon}+t^{-\frac{1}{1-r}}+t^{-2})$ for any $\ell\in\N_{\geq 1}$. That is, our main contributions for this setting are that our analysis is affine-invariant and captures a larger class of step-sizes than those in \cite{wirth2023acceleration}.

\subsection{Numerical experiments}\label{sec:weak_growth.examples}

\noindent
Proposition~\ref{prop:unif_heb} motivates the following example for which the weak $(M,r)$-growth property holds.  For ease of exposition, we take $\cC$ to be the unit $\ell_p$-norm ball for $p> 1$.  However, Proposition~\ref{prop:unif_heb} applies to the much broader class of uniformly convex domains.
\begin{example}[Strong $(M, 0)$-growth and weak $(M, r)$-growth]\label{ex:weak_boundary}
For $p > 1$, let $\cC = \{\xx\in\R^n \mid \|\xx\|_p \leq 1\}$ be an $\ell_p$-ball. 
Let $q\geq 1$ be such that $\frac{1}{p}+\frac{1}{q}=1$.
Let $\yy \in \R^n$ and $f(\xx):= \frac{1}{2}\|\xx-\yy\|_2^2$, where $\yy$ is a vector such that $\|\yy\|_p = 1$. 
We distinguish between two cases:
\begin{enumerate}
    \item When $p \in ]1, 2]$, 
    by Proposition~\ref{prop:unif_cvx}, $\cC$ is $(p-1)/4$-strongly convex with respect to $\|\cdot\|_{\cC} = \|\cdot\|_p$.
    By Lemma~\ref{lemma:lp_balls}, $f$ is $1$-smooth and satisfies a $(\sqrt{2}n^{\frac{1}{2}-\frac{1}{p}}, 1/2)$-Hölderian error bound with respect to $\|\cdot\|_{\cC}$.
    Thus, by Proposition~\ref{prop:unif_heb}, $(\cC, f)$ satisfies the weak $(M, r)$-growth property with $M= \frac{4\sqrt{2}n^{\frac{1}{2}-\frac{1}{p}}}{p-1}$ and $r=1$ and, by Lemma~\ref{lemma:lp_balls}, the strong $(M_0, 0)$-growth property with $M_0 = 4$.
    \item When $p\geq 2$, 
    by Proposition~\ref{prop:unif_cvx}, $\cC$ is $(1/(p2^{p-1}), p)$-uniformly convex with respect to $\|\cdot\|_{\cC}  = \|\cdot\|_p$.
    By Lemma~\ref{lemma:lp_balls}, $f$ is $n^{\frac{1}{2}-\frac{1}{p}}$-smooth and satisfies a $(\sqrt{2}, 1/2)$-Hölderian error bound with respect to $\|\cdot\|_p$.
    Thus, by Proposition~\ref{prop:unif_heb}, $(\cC, f)$ satisfies the strong $(M, r)$-growth property with $M= n^{\frac{1}{2}-\frac{1}{p}}(\sqrt{2}p2^{p-1})^{\frac{2}{p}}$ and $r=\frac{2}{p}$ and, by Lemma~\ref{lemma:lp_balls}, the strong $(M_0, 0)$-growth property with $M_0 = 4n^{\frac{1}{2}-\frac{1}{p}}$.
\end{enumerate}
\end{example}
\begin{figure}[t]
\captionsetup[subfigure]{justification=centering}
\begin{tabular}{c c c}
    \begin{subfigure}{.31\textwidth}
    \centering
        \includegraphics[width=1\textwidth]{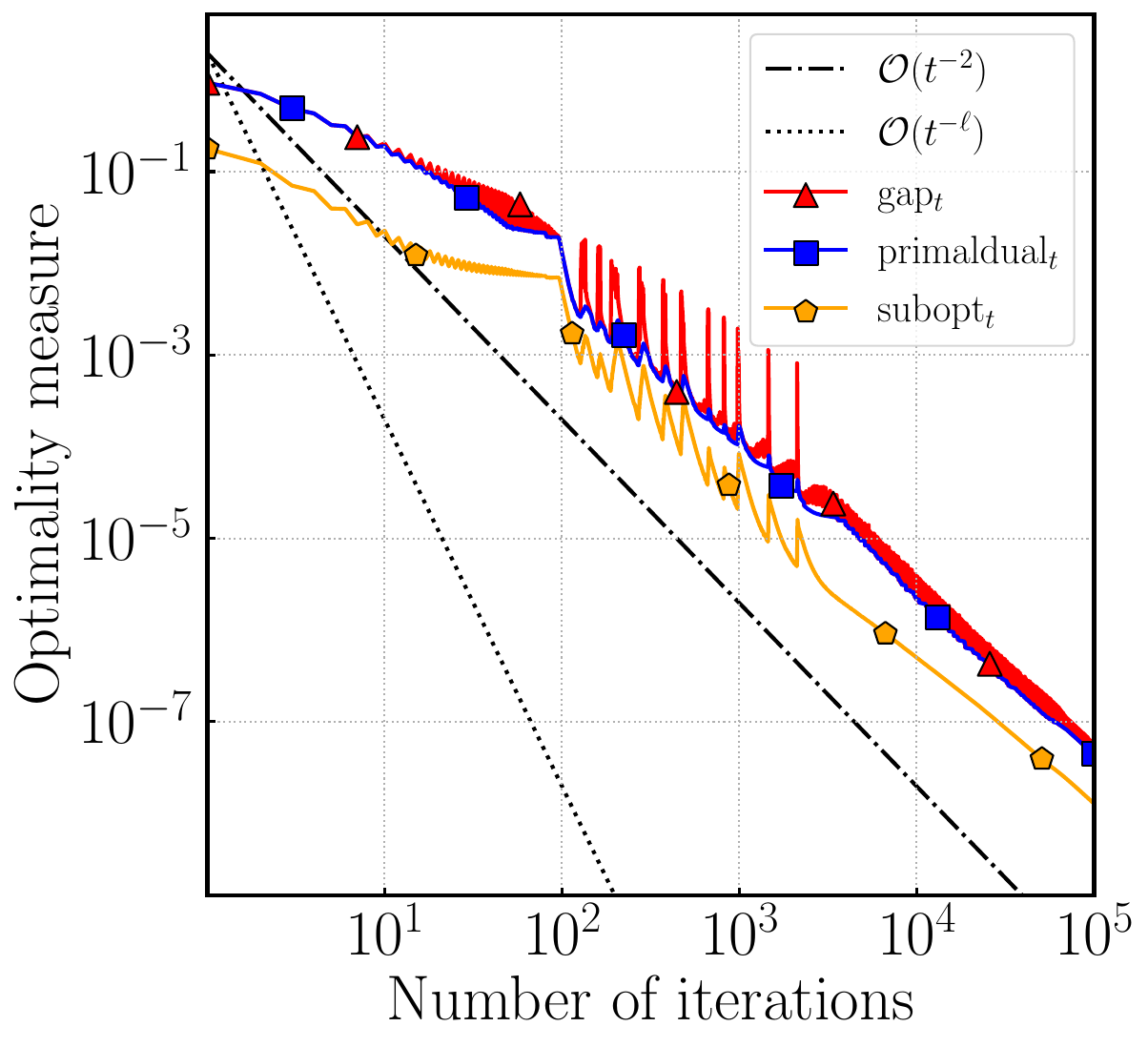}
        \caption{$\ell_{1.02}$-ball, $r=1/2$.}\label{fig:weak_boundary_1.02}
    \end{subfigure}& 
        \begin{subfigure}{.31\textwidth}
    \centering
        \includegraphics[width=1\textwidth]{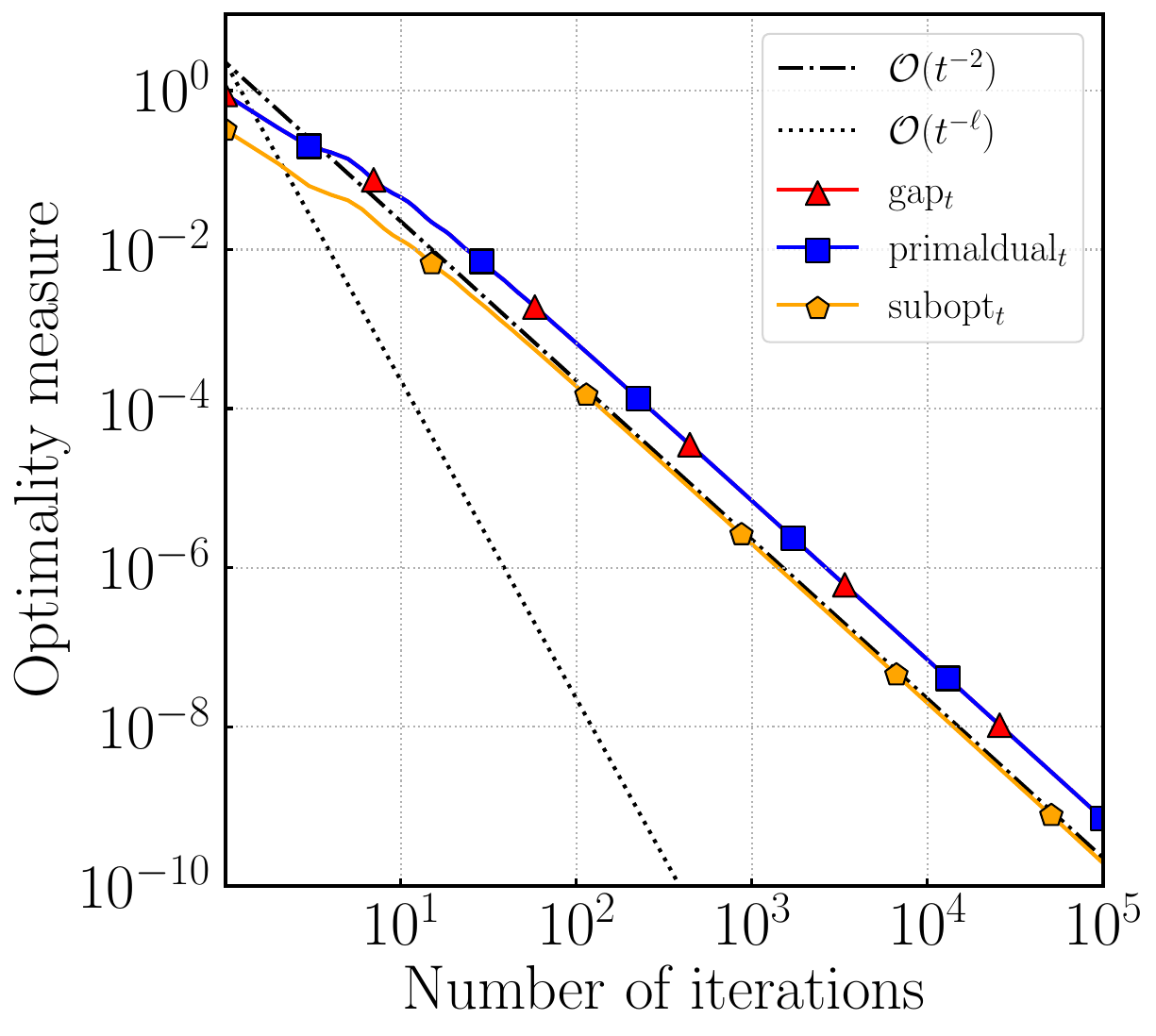}
        \caption{$\ell_{2}$-ball, $r=1/2$.}\label{fig:weak_boundary_2}
    \end{subfigure}& 
    \begin{subfigure}{.31\textwidth}
    \centering
        \includegraphics[width=1\textwidth]{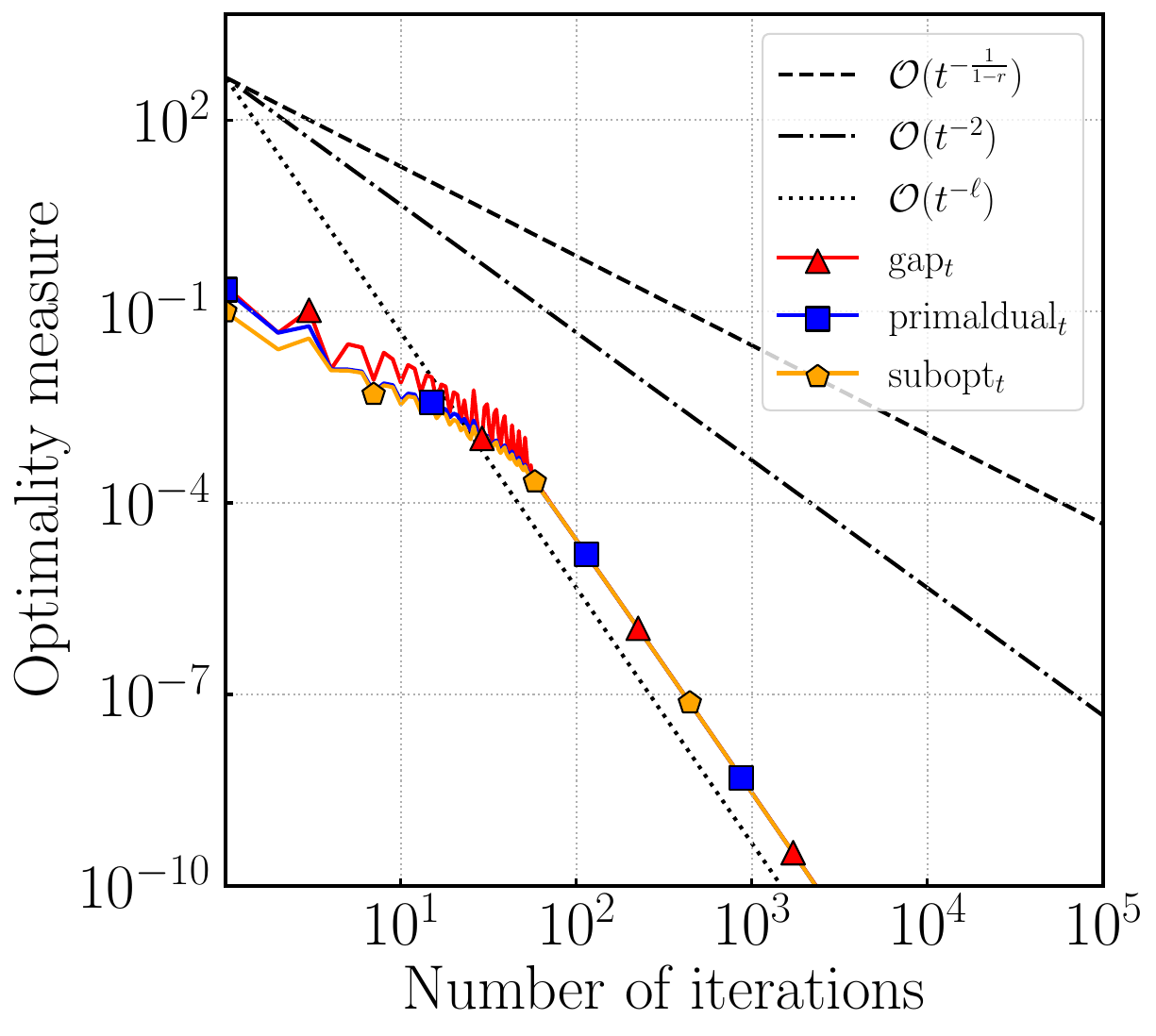}
        \caption{$\ell_7$-ball, $r=1/7$.}\label{fig:weak_boundary_7}
    \end{subfigure}
\end{tabular}
\caption{\textbf{Strong $(M, 0)$-growth and weak $(M, r)$-growth.} Optimality measure comparison of \fw{} with  step-size $\eta_t = \frac{\ell}{t+\ell}$ for $\ell = 4$ and all $t\in\N$ when the feasible region $\cC\subseteq\R^n$ for $n = 100$ is an $\ell_p$-ball and the objective is $f(\xx) =  \frac{1}{2}\|\xx - \yy\|_2^2$ for some random vector $\yy\in\R^n$ with $\|\yy\|_p = 1$. Axes are in log scale. 
}\label{fig:weak_unif_hold}
\end{figure}
In Figure~\ref{fig:weak_unif_hold}, in the setting of Example~\ref{ex:weak_boundary} for $n = 100$, $p\in\{1.02, 2, 7\}$, randomly chosen $\yy$ with $\|\yy\|_p=1$, and \fw{} with open-loop step-size $\eta_t = \frac{\ell}{t+\ell}$ for $\ell = 4$ and all $t\in\N$, we compare $\gap_t$, $\primaldual_t$, and $\subopt_t$. We also plot $\cO(t^{-\ell})$, $\cO(t^{-\frac{1}{1-r}})$, and $\cO(t^{-2})$ for better visualization, but we omit the second rate when $r = 1/2$. For these settings, weak $(M, r)$-growth, $r<1$, holds and Theorem~\ref{thm:rate-weak} applies. 
%

In Figure~\ref{fig:weak_boundary_1.02}, we observe that the convergence plot for $\gap_t$ forms a wedge, with the locally worst iterations converging at a rate of $\cO(t^{-1})$ and the locally best iterates converging at a rate of $\cO(t^{-2})$. The wedge ultimately collapses after which $\gap_t$ converges at a rate of $\cO(t^{-2})$.
In Figures~\ref{fig:weak_boundary_1.02} and~\ref{fig:weak_boundary_2}, all optimality measures converge at a rate of $\cO(t^{-2})$, whereas in Figure~\ref{fig:weak_boundary_7}, all optimality measures converge at a rate of order $\cO(t^{-\ell})$, which is better than predicted by Theorem~\ref{thm:rate-weak}. Note that the accelerated convergence rates for $\gap_t$ and $\primaldual_t$ are not explained by the current theory for any of the three figures.

\section{Strong $(M, 0)$-growth and (relaxed) gaps $(m, r)$-growth}\label{sec:gaps_growth}

We next discuss the convergence of open-loop \fw{} when the strong $(M,0)$-growth and the gaps $(m,r)$-growth properties hold. Although the combination of these two properties implies weak $(M/m,r)$-growth and hence is covered by Theorem~\ref{thm:rate-weak}, we discuss this case separately for two reasons.  First, Theorem~\ref{thm:rate-gaps} below gives a stronger convergence 
 result than Theorem~\ref{thm:rate-weak}. 
 Second, Theorem~\ref{thm:rate-gaps-relaxed} shows accelerated convergence when strong $(M,0)$-growth and the relaxed gaps $(m,r)$-growth properties hold. This latter case is not not covered by 
Theorem~\ref{thm:rate_strong} or Theorem~\ref{thm:rate-weak}.
As we detail in Section~\ref{sec.suff_gaps_m_r} below, the results in this section extend those in \cite{wirth2023acceleration} to a broader selection of open-loop step-sizes and under a more general context when $\cC$ is a polytope.

\subsection{Convergence rates}\label{sec:gaps_growth.results}

The following results relies on the gaps $(m, r)$-growth property. As in Theorem~\ref{thm:rate-weak}, the rates are only in suboptimality gap and capped by $\cO(t^{-2})$.

\begin{theorem}[Strong $(M, 0)$-growth and gaps $(m, r)$-growth]\label{thm:rate-gaps}
Let $\cC\subseteq \R^n$ be a compact convex set, let $f\colon\cC\to\R$ be convex and differentiable in an open set containing $\cC$, let $\eta_t = \frac{\ell}{t+\ell}$ for some $\ell\in\N_{\geq 1}$ and all $t\in\N$.  Suppose that $(\cC, f)$ satisfies the strong $(M,0)$-property for some $M>0$ and the gaps $(m, r)$-growth property for some $m>0$ and $r\in [0, 1[$.
Let $\fwt = \min\left\{i\in\N_{\geq 1} \mid \eta_i M\leq \frac{m^\frac{1}{r}}{2^\frac{1}{r}}\right\} = \max\left\{1, \left\lceil\frac{2^\frac{1}{r}\ell M }{m^\frac{1}{r}} - \ell\right\rceil\right\}$.
Then, for the iterates of Algorithm~\ref{alg:fw} and $t\in\N_{\geq \fwt}$, it holds that
\begin{align}\label{eq:rate_gaps_2}
    \subopt_t &\leq \max \left\{\subopt_{\fwt} \srb{\frac{\eta_{t-1}}{\eta_{\fwt-1}}}^{\ell}, \eta_{t-1}^{\min\{\ell, \frac{1}{1-r},2\}} \srb{\srb{\frac{M}{m}}^{\frac{1}{1-r}} + \frac{M}{2}}\right\}  \\ & = \cO(t^{-\ell} + t^{-\frac{1}{1-r}}+t^{-2}).\notag
\end{align}
\end{theorem}

Theorem~\ref{thm:rate-gaps} is a consequence of the stronger Theorem~\ref{thm:rate-gaps-relaxed} stated next.  The latter result provides the crux of the proof of accelerated convergence when $\cC$ is a polytope in cases when neither Theorem~\ref{thm:rate_strong} nor Theorem~\ref{thm:rate-weak} apply.

\begin{theorem}
[Strong $(M, 0)$-growth and relaxed gaps $(m, r)$-growth]
\label{thm:rate-gaps-relaxed}
Let $\cC\subseteq \R^n$ be a compact convex set, let $f\colon\cC\to\R$ be convex and differentiable in an open set containing $\cC$, let $\eta_t = \frac{\ell}{t+\ell}$ for some $\ell\in\N_{\geq 2}$ and all $t\in\N$. Suppose that $(\cC, f)$ satisfies the strong $(M,0)$-property for some $M>0$, and suppose there exists $r\in [0, 1[$ and $m>0$, and a threshold $\fwt \geq \max\left\{1, \left\lceil\frac{2^\frac{1}{r}\ell M }{m^\frac{1}{r}} - \ell\right\rceil\right\}$ 
such that 
$(\cC,f,(\xx_t)_{t\in\N},(\delta_t)_{t\in\N})$ satisfies the relaxed gaps property for the iterates $\xx_t$ generated by Algorithm~\ref{alg:fw}, $\delta_t = \srb{\frac{\eta_tM}{m}}^\frac{1}{1-r}$, and $t\in\N_{\geq\fwt}$.
Then, for the iterates of Algorithm~\ref{alg:fw} and $t\in\N_{\geq \fwt}$ it holds that
\begin{equation}\label{eq:rate_gaps_relaxed}
    \subopt_t \leq \max \left\{\subopt_{\fwt} \srb{\frac{\eta_{t-1}}{\eta_{\fwt-1}}}^{\ell}, \eta_{t-1}^{\min\{\ell, \frac{1}{1-r},2\}} \srb{\srb{\frac{M}{m}}^{\frac{1}{1-r}} + \frac{M}{2}}\right\}.
\end{equation}
\end{theorem}

\begin{proof}
This proof is very similar to that of Theorems~\ref{thm:rate_strong} and~\ref{thm:rate-weak} but a few minor tweaks are necessary and we restrict ourselves to presenting those. 
We will prove a stronger statement.  We will show that for the iterates of Algorithm~\ref{alg:fw} and $t\in\N_{\geq \fwt}$, it holds that
\begin{align}\label{eq:rate_gaps}
    \subopt_t \leq \max \left\{\subopt_{\fwt} \srb{\frac{\eta_{t-1}}{\eta_{\fwt-1}}}^{\ell},  \max_{\Rfwt\in\{\fwt, \ldots, t\}}  \eta_{R-1}^{\min\{\frac{1}{1-r},2\}}\srb{\frac{\eta_{t-1}}{\eta_{\Rfwt-1}}}^{\ell}  \srb{\left(\frac{M}{m}\right)^{\frac{1}{1-r}} +\frac{M}{2}}
    \right\},
    \end{align}
which in turn implies~\eqref{eq:rate_gaps_relaxed} as we detail below.
We next prove~\eqref{eq:rate_gaps} by induction.  The bound~\eqref{eq:rate_gaps} readily holds for $t=S$. For the main inductive  step, we distinguish between two main cases:
\begin{enumerate}
    \item\label{case:good_progress.gaps} Suppose that
    \begin{align}\label{eq:smaller.gaps}
    \subopt_{t} \leq \srb{\frac{\eta_{t}M}{m}}^{\frac{1}{1-r}}.
\end{align}
Hence, Equality~\eqref{eq:fw-step} and the strong $(M,0)$-growth property imply that
\begin{align}\label{eq:good_progress.gaps}
    \subopt_{t+1} &\leq \subopt_{t} + \frac{\eta_{t}^2M}{2} \leq \srb{\frac{\eta_{t}M}{m}}^{\frac{1}{1-r}}+\frac{\eta_{t}^2M}{2}\le \eta_{t}^{\min\{\frac{1}{1-r},2\}} \srb{\left(\frac{M}{m}\right)^{\frac{1}{1-r}} +\frac{M}{2}}.
\end{align}
Observe that
\[
\eta_{t}^{\min\{\frac{1}{1-r},2\}} \srb{\left(\frac{M}{m}\right)^{\frac{1}{1-r}} +\frac{M}{2}} \le \max_{\Rfwt\in\{\fwt, \ldots, t+1\}}  \eta_{R-1}^{\min\{\frac{1}{1-r},2\}}\srb{\frac{\eta_{t}}{\eta_{\Rfwt-1}}}^{\ell}  \srb{\left(\frac{M}{m}\right)^{\frac{1}{1-r}} +\frac{M}{2}}.
\]
Thus, \eqref{eq:rate_gaps} holds for $t+1$.
\item
\label{case:bad_progress.gaps} Suppose that \eqref{eq:smaller.gaps} does not hold. Let $\Rfwt\in \{\fwt, \ldots, t\}$ be the smallest index such that 
\begin{align}\label{eq:greater.gaps}
    \subopt_{i} \geq \srb{\frac{\eta_{i}M}{m}}^{\frac{1}{1-r}}
\end{align}
for all $i\in\{\Rfwt, \Rfwt + 1, \ldots, t\}$.
Then, Equality~\eqref{eq:fw-step}, Inequality~\eqref{eq:various.gaps}, and the relaxed gaps $(m, r)$-growth property imply that
\begin{align}\label{eq:rec.gaps}
   &\subopt_{i+1}  \leq \subopt_i -\eta_i \cdot \gap_i + \frac{\eta_i^2 M}{2}\nonumber\\
   & \leq \subopt_i -\eta_i \cdot m \cdot  \subopt_i^{1-r} + \frac{\eta_i^2 M}{2}\nonumber\\
   & \leq \subopt_i -\eta_i \cdot m \cdot \subopt_i^{1-r} + \frac{\eta_i \cdot m \cdot \subopt_i^{1-r}}{2}& \text{$\triangleright$ since $\subopt_i \geq \srb{\frac{\eta_{i}M}{m}}^{\frac{1}{1-r}}$}\nonumber\\
   & \leq \subopt_i -  \subopt_i^{1-r} \frac{\eta_i m}{2}\nonumber\\
   & \leq \subopt_i \srb{1 -\eta_i^{1-r} \frac{m}{2 M^r}} & \text{$\triangleright$ by Theorem~\ref{thm:slow_conv}, $\subopt_i \leq \eta_i M$} \nonumber \\
   & \leq \subopt_i (1 -\eta_i) & \text{$\triangleright$ since $\frac{m}{2M^r\eta_i^r}\geq 1$ for all $i\geq \fwt$} 
\end{align}
for all $i\in \{\Rfwt, \Rfwt + 1, \ldots, t\}$.
By repeatedly applying \eqref{eq:rec.gaps} and by Lemma~\ref{lemma:telescope.simple}, 
\begin{align}\label{eq:calc_big.gaps}
    \subopt_{t+1} & \leq \subopt_{\Rfwt}\left(\frac{\eta_t}{\eta_{\Rfwt-1}}\right)^\ell.
\end{align}
We next proceed as in Case~\ref{case:bad_progress.weak} in Theorem~\ref{thm:rate-weak}.  If $R=S$ then~\eqref{eq:calc_big.gaps} implies that
\[
\subopt_{t+1}  \leq \subopt_{\fwt}\left(\frac{\eta_t}{\eta_{\fwt-1}}\right)^\ell,
\]
and if $R>S$ then bound~\eqref{eq:good_progress.gaps} in Case~\ref{case:good_progress.gaps} and~\eqref{eq:calc_big.gaps} imply that
\[
\subopt_{t+1} \le
\eta_{R-1}^{\min\{\frac{1}{1-r},2\}}\srb{\frac{\eta_{t}}{\eta_{\Rfwt-1}}}^{\ell}  \srb{\left(\frac{M}{m}\right)^{\frac{1}{1-r}} +\frac{M}{2}}.
\]
Thus, \eqref{eq:rate_gaps} holds for $t+1$. 
\end{enumerate}
As in the proof of Theorems~\ref{thm:rate_strong} and~\ref{thm:rate-weak}, we next show that
\eqref{eq:rate_gaps} implies~\eqref{eq:rate_gaps_relaxed}. Indeed, by proceeding as in the proof of Theorem~\ref{thm:rate_strong}, we have that 
\begin{align*}
    \max_{\Rfwt\in\{\fwt, \ldots, t\}}
\eta_{R-1}^{\min\{\frac{1}{1-r},2\}}\srb{\frac{\eta_{t-1}}{\eta_{\Rfwt-1}}}^{\ell}  \srb{\left(\frac{M}{m}\right)^{\frac{1}{1-r}} +\frac{M}{2}} \leq \eta_{t-1}^{\min\{\ell, \frac{1}{1-r},2\}} \srb{\left(\frac{M}{m}\right)^{\frac{1}{1-r}} +\frac{M}{2}}.
\end{align*}
Thus, \eqref{eq:rate_gaps_relaxed} follows from~\eqref{eq:rate_gaps}.
\end{proof}

The following sanity check concerning the scaling behavior of the expressions in Theorem~\ref{thm:rate-gaps} is reassuring: the expression for $\fwt$ is scale invariant as  both sides of~\eqref{eq:rate_gaps_2} scale in the same fashion.  Indeed if $f$ is scaled by $\gamma > 0$ then $M$ scales by $\gamma$ whereas $m$ scales by $\gamma^r$. Thus, $\fwt$
is scale invariant and both sides of~\eqref{eq:rate_gaps_2} scale by $\gamma$.

\medskip

The results derived in this section characterize the acceleration of \fw{} with open-loop step-sizes when the gaps $(m,r)$-growth property is satisfied. It remains to discuss sufficient conditions for the gaps $(m,r)$-growth and relaxed 
gaps $(m,r)$-growth properties to be satisfied.

\subsection{Sufficient conditions for gaps and relaxed gaps growth}
\label{sec.suff_gaps_m_r}

We next present sufficient conditions for both the gaps $(m, r)$-growth property and for the relaxed gaps $(m, r)$-growth property.  The sufficient conditions for the latter one are substantially more involved but they also lead to one of our most interesting developments, namely Corollary~\ref{cor:rel_gap}.  This result establishes the first affine-invariant accelerated convergence rate guarantee in the challenging Wolfe's lower bound setting~\cite{wolfe1970convergence}.

\subsubsection{Sufficient conditions for gaps $(m,r)$ growth}

A simple setting for which the gaps $(m, r)$-growth property holds is when \change{the} feasible region is compact and the objective satisfies a Hölderian error bound. The following result that gives sufficient conditions for gaps $(m, r)$-growth is  analogous to 
Proposition~\ref{prop:exterior} and
Proposition~\ref{prop:unif_heb} which give sufficient conditions for strong and weak growth.  Observe that this result is affine-invariant as all of our previous developments.

\begin{proposition}\label{prop:gaps_interior}
Let $\cC \subseteq \R^n$ be a compact convex set. For $\mu > 0$, and $\theta \in [0, 1/2]$, let $f\colon\cC\to\R$ be a convex function satisfying a $(\mu, \theta)$-Hölderian error bound with respect to $\|\cdot\|_{\cC}$. Suppose the optimal set $X^* = \argmin_{\xx\in \cC}f(\xx)$ satisfies $X^*\subseteq \relativeinterior(\cC)$.
Let 
\begin{equation}\label{eq.def.rho}
\rho:=\min\{\|\xx-\xx^*\|_{\cC} \mid \xx^* \in X^*, \xx \in \relativeboundary(\cC)\}.
\end{equation}
Then,
$(\cC,f)$ satisfies the gaps $(m,r)$-growth property with $m = \frac{\rho}{\mu} > 0$ and $r = \theta$.
\end{proposition}
\begin{proof}
The construction~\eqref{eq.def.rho} of $\rho$ implies that for all \change{$\xx \in \cC\setminus X^*$}
\[
\xx^*+\frac{\rho}{\|\xx-\xx^*\|_{\cC}} (\xx-\xx^*) \in \cC
\]
where $\xx^*\in X^*$ is the optimal solution closest to $\xx$.
Thus
\begin{align*}
\gap(\xx) &= \argmax_{\yy\in \cC} \left\langle\nabla f\left(\xx\right), \xx-\yy \right\rangle \\ &\geq \srb{1+\frac{\rho}{\|\xx-\xx^*\|_{\cC}}} \left\langle\nabla f\left(\xx\right), \xx-\xx^* \right\rangle 
\\ &\geq \frac{\rho}{\|\xx-\xx^*\|_{\cC}} (f(\xx) - f(\xx^*)).
\end{align*}
Thus, the $(\mu,\theta)$ H\"olderian error bound implies that
\[
\|\xx-\xx^*\|_{\cC}\cdot \gap(\xx) \geq \rho \cdot\subopt(\xx) \geq \rho \cdot \left(\frac{\|\xx-\xx^*\|_{\cC}}{\mu} \right)^\frac{1}{\theta}.
\]
Therefore, by taking the $(1-\theta, \theta)$ geometric mean of the last two quantities, it follows that
\[
\|\xx-\xx^*\|_{\cC}\cdot \gap(\xx) \geq \frac{\rho}{\mu}\cdot \subopt(\xx)^{1-\theta} \cdot \|\xx-\xx^*\|_{\cC}.
\]
Hence, $(\cC,f)$ satisfies the gaps $(m,r)$-growth property with $m = \frac{\rho}{\mu}$ and $r = \theta$.  Observe that $m>0$ because $\xx^*\in\relativeinterior(\cC^*)$ implies that $\rho > 0.$
\end{proof}
%

The results in this section improve over prior work in \cite{wirth2023acceleration} as follows: In the setting of Proposition~\ref{prop:gaps_interior}, Theorem~3.6 in \cite{wirth2023acceleration} derived affine-dependent rates of order $\subopt_t = \cO(t^{-\frac{1}{1-r}}+t^{-2})$ for \fw{} with $\eta_t = \frac{4}{t+4}$. In contrast, Theorem~\ref{thm:rate-gaps} yields affine-invariant rates of order $\subopt_t = \cO(t^{-\ell}+t^{-\frac{1}{1-r}}+t^{-2})$ for any $\ell\in\N_{\geq 1}$. That is, our main contributions for this setting are that our analysis is affine-invariant and captures a larger class of step-sizes than those in \cite{wirth2023acceleration}.

\subsubsection{Sufficient conditions for relaxed gaps $(m, r)$-growth}\label{sec.suff_relaxed_gaps_m_r}

A simple setting for which the relaxed gaps $(m, r)$-growth property holds is when the feasible region is a polytope and the objective satisfies a Hölderian error bound. Proposition~\ref{prop:polytope} below is analogous to  
Proposition~\ref{prop:gaps_interior} and gives sufficient conditions for relaxed gaps growth. We should highlight that Proposition~\ref{prop:polytope} relies heavily on the feasible region being a polytope and on the choice of open-loop step-sizes of the form $\frac{\ell}{t+\ell}$ for $\ell \in \N_{\geq 2}.$  As a result, we get accelerated  rate of convergence even for a class of problems where it is well-known that exact line-search and short-step do not accelerate beyond $\cO(t^{-1})$ such as in the case when the minimizer lies in the interior of a proper face of the polytope of dimension at least one as shown in~\cite{wolfe1970convergence}.

To rigorously address this special polytope setting, we first introduce several additional definitions and assumptions that we will rely on throughout this subsection. 

\begin{enumerate}
\item\label{assum.cont} The set $\cC\subseteq \R^n$ is a polytope, and $f\colon \cC \to \R$ is convex and continuously differentiable in an open set containing $\cC$.
\item\label{assum.relativeinterior}
The point $\xx^* := \argmin_{\xx\in\cC}f(\xx)$ exists and is unique, and $\cC^*\in\faces(\cC)$ is the unique face of $\cC$ such that $\xx^*\in\relativeinterior(\cC^*)$.  For $\xx\in \cC$ let $\bar \xx:=\argmin_{\yy\in \cC^*}\|\xx-\yy\|_{\cC}\in \cC^*$.
\item\label{assum.B.rho} Let  $B$ and $\rho$ be as follows
\begin{align*}
B&:=\max_{\yy\in \cC, \xx\in \cC\setminus \cC^*} \frac{|\ip{\nabla f(\yy)}{\xx-\bar\xx}|}{\|\xx-\bar\xx\|_{\cC}}\\
\rho &:= \min\{\|\xx-\xx^*\|_{\cC} \mid \xx\in \relativeboundary(\cC^*)\}.
\end{align*}
Assumption~\ref{assum.cont} implies that $B<\infty$ and 
Assumption~\ref{assum.relativeinterior} implies that $\rho > 0$.
\item\label{assum.active} There exists $\Afwt\in \N$ such that 
\[
\argmin_{\yy\in\cC} \ip{\nabla f(\xx_t)}{\yy} \in \cC^* \qquad \text{for all} \ t\geq \Afwt.
\]

\end{enumerate}

We consider the optimization problem \eqref{eq:opt} with Algorithm~\ref{alg:fw} and step-size rule $\eta_t = \frac{\ell}{t+\ell}$ for some $\ell\in\N_{\geq 2}$.
In this context, Theorem~\ref{thm:rate-gaps-relaxed} and Proposition~\ref{prop:polytope} below yield accelerated rate of convergence.

\begin{proposition}\label{prop:polytope}
Suppose that Assumptions~\ref{assum.cont}--~\ref{assum.active} hold.  
Suppose also that $\nabla f$ is $L$-Lipschitz continuous on $C$ with respect to $\|\cdot\|_{\cC}$ and $f$ satisfies a $(\mu, \theta)$-Hölderian error bound with respect to $\|\cdot\|_{\cC}$ for some $L > 0$, $\mu > 0$, and $\theta\in[0,1/2]$. In addition, suppose that $\eta_t = \frac{\ell}{t+\ell}$ for $\ell \in \N_{\geq 2}$ in  Algorithm~\ref{alg:fw}.
Let \[
m = \frac{\rho}{2^{2-\theta}\mu} \; \text{ and } \;
M = \max\left\{\frac{m}{\eta_\Afwt}(4B)^{1-\theta},2\frac{2L+B}{\eta_\Afwt},4L\right\}.
\]
Then, $(\cC, f)$ satisfies the strong $(M,0)$-growth property, and 
$(\cC,f,(\xx_t)_{t\in\N},(\delta_t)_{t\in\N})$ satisfies the relaxed gaps property for the iterates $\xx_t$ generated by Algorithm~\ref{alg:fw}, $\delta_t =\left(\frac{\eta_t M}{m}\right)^\frac{1}{1-\theta}$, and $t\in\N_{\geq \Afwt}$.  More precisely,
if $t\in\N_{\geq \Afwt}$ and $\subopt_t  \geq \left(\frac{\eta_t M}{m}\right)^\frac{1}{1-\theta}$, then 
\begin{equation}\label{eq.relaxed}
m\cdot \subopt_t^{1-\theta} \leq \gap_t.
\end{equation}
\end{proposition} 

The results in this section improve over prior work in \cite{wirth2023acceleration} as follows: In the setting of Proposition~\ref{prop:polytope}, Theorem~4.3 in \cite{wirth2023acceleration} derived affine-dependent rates of order $\subopt_t = \cO(t^{-2})$ for \fw{} with $\eta_t = \frac{4}{t+4}$, however, only when the objective is strongly convex. In contrast, Theorem~\ref{thm:rate-gaps-relaxed} yields affine-invariant rates of order $\subopt_t = \cO(t^{-\ell}+t^{-\frac{1}{1-r}}+t^{-2})$ for any $\ell\in\N_{\geq 1}$. Thus, our main contributions for this setting are that our analysis is affine-invariant, captures a larger class of step-sizes than those in \cite{wirth2023acceleration}, and is more general in that we capture any objectives satisfying Hölderian error bounds as opposed to only strongly convex functions. Note, this setting is highly significant to \fw{} research as it is one of the few settings in which open-loop \fw{} outperforms \fw{} with line-search and short-step. Indeed, \fw{} with the latter two step-sizes converges at rates of order at most $\Omega(t^{-1-\epsilon})$ for any $\epsilon > 0$ as shown in the proof of Wolfe's lower bound \cite{wolfe1970convergence}. This is also illustrated in the experiments of Section~\ref{sec.num_w_lb}.

Theorem~\ref{thm:rate-gaps-relaxed} and Proposition~\ref{prop:polytope} readily imply the accelerated convergence rate $\cO(t^{-\frac{1}{1-\theta}})$  as formally stated in the following corollary.

\begin{corollary}\label{cor:rel_gap}
Suppose that Assumptions~\ref{assum.cont}--~\ref{assum.active} hold.  
Suppose also that $\nabla f$ is $L$-Lipschitz continuous on $C$ with respect to $\|\cdot\|_{\cC}$ and $f$ satisfies a $(\mu, \theta)$-Hölderian error bound with respect to $\|\cdot\|_{\cC}$ for some $L > 0$, $\mu > 0$, and $\theta\in[0,1/2]$. In addition, suppose that $\eta_t = \frac{\ell}{t+\ell}$ for $\ell \in \N_{\geq 2}$ in  Algorithm~\ref{alg:fw}.
Let \[
m = \frac{\rho}{2^{2-\theta}\mu}, \;
M = \max\left\{\frac{m}{\eta_\Afwt}(4B)^{1-\theta},2\frac{2L+B}{\eta_\Afwt},4L\right\},  \; 
S:= \max\left\{1,\left\lceil\frac{2^\frac{1}{\theta}\ell M}{m^\frac{1}{\theta}}-\ell\right\rceil,\Afwt\right\}.
\]
Then, for the iterates of Algorithm~\ref{alg:fw} and $t\in\N_{\geq \fwt}$ it holds that
\[
    \subopt_t \leq \max \left\{\subopt_{\fwt} \srb{\frac{\eta_{t-1}}{\eta_{\fwt-1}}}^{\ell}, \eta_{t-1}^{\min\{\ell, \frac{1}{1-\theta},2\}} \srb{\srb{\frac{M}{m}}^{\frac{1}{1-\theta}} + \frac{M}{2}}\right\} = \cO(t^{-\frac{1}{1-\theta}}).
\]
\end{corollary}

The proof of Proposition~\ref{prop:polytope} 
relies on the following lemma that bounds the distance of iterate $\xx_t$ to the active set in affine-invariant fashion.  The bound~\eqref{eq.distance} relies crucially on the choice of open-loop step-size $\eta_t = \frac{\ell}{t+\ell}$ with $\ell\in \N_{\geq 2}$.

\begin{lemma}[Distance to active set] 
\label{lemma:distance_to_active_set}
Let $\cC\subseteq \R^n$ be a polytope, let $f\colon\cC\to\R$ be convex and differentiable in an open set containing $\cC$, let $\eta_t = \frac{\ell}{t+\ell}$ for some $\ell\in\N_{\geq 2}$ and all $t\in\N$, and suppose $\Afwt<\infty$. Then, for the iterates of Algorithm~\ref{alg:fw} and $t\in\N_{\geq \Afwt}$, it holds that 
\begin{align*}
    \|\bar \xx_t-\xx_t\|_{\cC} \leq 2\prod_{i=\Afwt}^{t-1}(1-\eta_i) = \frac{2\Afwt\cdots(\Afwt-1+\ell)}{t\cdots (t-1+\ell)}.
\end{align*}
In particular, for $t\in\N_{\geq  \Afwt}$, it holds that
\begin{equation}\label{eq.distance}
     \|\bar \xx_t-\xx_t\|_{\cC} \leq \frac{2\Afwt(\Afwt-1+ \ell)}{t(t-1+\ell)} \leq \frac{2(\Afwt+ \ell)^2}{(t+\ell)^2} = \frac{2\eta_t^2}{\eta_\Afwt^2}.
\end{equation}
\end{lemma}
\begin{proof}
We proceed by induction.  For $t=\Afwt$ we readily have $\|\xx_{\Afwt} - \bar \xx_{\Afwt}\|_{\cC} \le 2$.
Suppose the bound holds for $t\geq \Afwt$. Thus, for some $\vv\in \cC^*$, we have
\begin{align*}
\xx_{t+1} = (1-\eta_{t})\xx_{t} + \eta_{t} \vv \qquad \Leftrightarrow \qquad  
\xx_{t+1}- ((1-\eta_{t})\bar \xx_{t} + \eta_{t} \vv) = (1-\eta_{t})(\xx_t - \bar \xx_{t}).
\end{align*}
The latter equation and the induction hypothesis imply that
\begin{align*}
\|\bar \xx_{t+1}-\xx_{t+1}\|_{\cC} \leq \|(1-\eta_{t})\bar \xx_{t} + \eta_{t} \vv-\xx_{t+1}\|_{\cC} \leq 
(1-\eta_{t})\|\bar \xx_t-\xx_t\|_{\cC} \leq 2\prod_{i=\Afwt}^{t}(1-\eta_i).
\end{align*}
\end{proof}
\begin{proof}[Proof of Proposition~\ref{prop:polytope}]
Since $M\geq 4L$, the strong $(M,0)$-growth property follows from the $L$-Lipschitz continuity of $f$ and Proposition~\ref{prop.strong.M.0}.  
 We next show~\eqref{eq.relaxed} when $t\in \N_{\geq \Afwt}$ and $\subopt_t  \geq \left(\frac{\eta_tM}{m}\right)^\frac{1}{1-\theta}$. 
The construction of $\rho$ implies that
\begin{align*}
\max_{\yy\in \cC^*}\langle\nabla f(\bar \xx_t),\bar \xx_t-\yy\rangle &\ge
\frac{\rho}{\|\bar \xx_t-\xx^*\|_{\cC}}\ip{\nabla f(\bar \xx_t)}{\bar \xx_t - \xx^*}
\\&\geq \frac{\rho}{\|\bar \xx_t-\xx^*\|_{\cC}}(f(\bar \xx_t) - f(\xx^*)) 
\\&\geq \frac{\rho}{\|\bar \xx_t-\xx^*\|_{\cC}}\left(\frac{\|\bar \xx_t - \xx^*\|_{\cC}}{\mu}\right)^\frac{1}{\theta}.
\end{align*}
Thus, the first expression is bounded from below by the $(1-\theta,\theta)$ geometric mean of the last two quantities, that is,
\begin{equation}\label{eq:lowerbound.g}
\max_{\yy\in \cC^*}\langle\nabla f(\bar \xx_t),\bar \xx_t-\yy\rangle \geq \frac{\rho}{\mu}(f(\bar \xx_t) - f(\xx^*))^{1-\theta}.
\end{equation}
Since $t\in \N_{\geq \Afwt}$, we have
\[
\gap_t = \max_{\yy\in\cC}\ip{\nabla f(\xx_t)}{\xx_t-\yy} \\
= \max_{\yy\in \cC^*}\ip{\nabla f(\xx_t)}{\xx_t-\yy}.
\]
Therefore the construction of $B$, $L$-Lipschitz continuity of $\nabla f$, Inequality~\eqref{eq:lowerbound.g}, convexity of $f$ imply that
\begin{align}\label{eq:bound.gap}
\gap_t 
&= \max_{\yy\in \cC^*}\langle\nabla f(\bar \xx_t),\bar \xx_t-\yy\rangle + 
\langle\nabla f(\bar \xx_t),\xx_t - \bar \xx_t\rangle + \langle\nabla f(\xx_t)-\nabla f(\bar \xx_t),\xx_t-\yy\rangle \notag
\\
&\geq \max_{\yy\in \cC^*}\langle\nabla f(\bar \xx_t),\bar \xx_t-\yy\rangle - (2L+B)\|\bar \xx_t-\xx_t\|_{\cC}\notag
\\
&\geq \frac{\rho}{\mu}(f(\bar \xx_t) - f(\xx^*))^{1-\theta} - (2L+B)\|\bar \xx_t-\xx_t\|_{\cC} \notag\\
&\geq \frac{\rho}{\mu}(f(\xx_t) - f(\xx^*) + \ip{\nabla f(\xx_t)}{\bar \xx_t-\xx_t})^{1-\theta}- (2L+B)\|\bar \xx_t-\xx_t\|_{\cC}
\notag  \\
&\geq \frac{\rho}{\mu}\left(\subopt_t -B \|\bar \xx_t-\xx_t\|_{\cC}\right)^{1-\theta}- (2L+B)\|\bar \xx_t-\xx_t\|_{\cC}.
\end{align}
Next we bound each of the two terms in~\eqref{eq:bound.gap}.
Lemma~\ref{lemma:distance_to_active_set},  $t\in \N_{\geq \Afwt},$ and $\theta\in [0,1/2]$ imply that
\[
\|\bar \xx_t-\xx_t\|_{\cC} \le 2\frac{\eta_t^2}{\eta_\Afwt^2} \le 2\left(\frac{\eta_t}{\eta_\Afwt}\right)^\frac{1}{1-\theta}.
\]
Hence, the construction of $M$ implies that 
\begin{equation}\label{eq:bound.norm.1}
B\|\bar \xx_t-\xx_t\|_{\cC} \le 2B \frac{\eta_t^2}{\eta_\Afwt^2}\leq 2B \left(\frac{\eta_t}{\eta_\Afwt}\right)^\frac{1}{1-\theta} \leq \frac{1}{2}\left(\frac{\eta_t M}{m} \right)^\frac{1}{1-\theta},
\end{equation}
and
\begin{equation}\label{eq:bound.norm.2}
(2L+B)\|\bar \xx_t-\xx_t\|_{\cC} \le 2(2L+B) \frac{\eta_t^2}{\eta_\Afwt^2} \le 2(2L+B) \frac{\eta_t}{\eta_\Afwt} \leq  M\eta_t.
\end{equation}
Finally, by putting together~\eqref{eq:bound.gap},~\eqref{eq:bound.norm.1},~\eqref{eq:bound.norm.2}, the construction of $m$, and $\subopt_t  \geq \left(\frac{\eta_t M}{m} \right)^\frac{1}{1-\theta}$, we get
\begin{align*}
\gap_t  & \geq \frac{\rho}{\mu}\left(\frac{\subopt_t}{2}\right)^{1-\theta}- M \eta_t 
\geq 2m \cdot \subopt_t^{1-\theta} - M \eta_t \geq m \cdot \subopt_t^{1-\theta}.
\end{align*}
\end{proof}

The following sanity check is reassuring: the two sides of each of the  inequalities $\subopt_t  \geq   \left(\frac{\eta_t M}{m}\right)^\frac{1}{1-\theta}$ and $m\cdot\subopt_t^{1-\theta} \leq \gap_t$ scale consistently. Indeed, if $f$ is scaled by $\gamma > 0$, then $B$ scales by $\gamma$ whereas $m$ scales by $\gamma^\theta$. Thus, $M$ scales by $\gamma$ and $M/m$ scales by $\gamma^{1-\theta}$.
Consequently, both $\subopt_t$ and $\left(\frac{\eta_t M}{m}\right)^\frac{1}{1-\theta} $  scale by $\gamma$ as do both $m\cdot\subopt_t^{1-\theta}$ and $\gap_t$.

The following lemma shows that Assumption~\ref{assum.active} holds provided that some suitable strict complementarity condition holds.

\begin{lemma}[Active-set identification]\label{lemma.active}
Suppose Assumption~\ref{assum.cont} and Assumption~\ref{assum.relativeinterior} hold. 
Suppose also that $\nabla f$ is $L$-Lipschitz continuous on $C$ with respect to $\|\cdot\|_{\cC}$ and $f$ satisfies a $(\mu, \theta)$-Hölderian error bound with respect to $\|\cdot\|_{\cC}$ for some $L,\mu > 0$, and $\theta\in[0,1/2]$.
In addition suppose that the following strict complementarity condition holds for some $\kappa > 0$:
\begin{equation}\label{eq:strict_comp}
\langle\nabla f(\xx^*),\vv - \xx^*\rangle \geq \kappa \text{ for }
\vv\in \vertices(\cC) \setminus \vertices(\cC^*)\; \text{ and } \;
\langle\nabla f(\xx^*),\vv - \xx^*\rangle =0 \text{ for }
\vv\in \vertices(\cC^*).
\end{equation}
Then, for all $t \geq \left\lceil 4L\ell\cdot (\frac{4L\mu}{\kappa})^\frac{1}{\theta}\right\rceil$ 
the iterates generated by Algorithm~\ref{alg:fw} satisfy the following active-set identification property
\begin{equation}\label{eq:active_id}
\argmin_{\yy\in\cC} \ip{\nabla f(\xx_t)}{\yy} \in \cC^*.
\end{equation}
\end{lemma}
\begin{proof}
For $\vv\in \vertices(\cC)$, we have
\begin{align*}
\ip{\nabla f(\xx_t)}{\vv-\xx_t} = \ip{\nabla f(\xx^*)}{\vv-\xx^*}+\ip{\nabla f(\xx^*)}{\xx^*-\xx_t} + \ip{\nabla f(\xx_t) - \nabla f(\xx^*)}{\vv-\xx_t}.
\end{align*}
Thus, the strict complementarity assumption implies that
\begin{align}\label{eq.active.step1}
\ip{\nabla f(\xx_t)}{\vv-\xx_t}   \left\{\begin{array}{ll} \geq \ip{\nabla f(\xx^*)}{\xx^*-\xx_t}  -|\ip{\nabla f(\xx_t) - \nabla f(\xx^*)}{\vv-\xx_t}| + \kappa & \text{ if } \vv \not\in \cC^*\\
\leq \ip{\nabla f(\xx^*)}{\xx^*-\xx_t} + |\ip{\nabla f(\xx_t) - \nabla f(\xx^*)}{\vv-\xx_t}| & \text{ if } \vv\in \cC^*.
\end{array}\right.
\end{align}

The $L$-Lipschitz continuity of $\nabla f$ and Theorem~\ref{thm:slow_conv}, implies that $\subopt_t \leq \frac{4L \ell}{t+\ell}$.  Next, $L$-Lipschitz continuity of $\nabla f$ again and the H\"olderian error bound assumption of $f$ imply that
\begin{align*}
    \left\vert\ip{\nabla f(\xx_t) - \nabla f(\xx^*)}{\vv-\xx_t}\right\vert \leq 2L \|\xx_t - \xx^*\|_{\cC} \leq 2L\mu\cdot\subopt_t^\theta &\leq 
   2L\mu\left(\frac{4L\ell}{t+\ell}\right)^\theta\\ &< 2L\mu\left(\frac{4L\ell}{t}\right)^\theta.
\end{align*}
Thus, for $t \geq 4L\ell\cdot (\frac{4L\mu}{\kappa})^\frac{1}{\theta}$, we have
\begin{align}\label{eq.active.step2}
\left\vert\ip{\nabla f(\xx_t) - \nabla f(\xx^*)}{\vv-\xx_t}\right\vert < \frac{\kappa}{2}.
\end{align}
Combining~\eqref{eq.active.step1} and~\eqref{eq.active.step2}, it follows that 
$\argmin_{\yy\in\cC} \ip{\nabla f(\xx_t)}{\yy} \in \cC^*$ whenever $t \geq \left\lceil 4L\ell\cdot (\frac{4L\mu}{\kappa})^\frac{1}{\theta}\right\rceil$.
\end{proof}

Once again, the following sanity check is reassuring: the lower bound $4L\ell\cdot (\frac{4L\mu}{\kappa})^\frac{1}{\theta}$ on $\Afwt$ is scale invariant.  Indeed if $f$ is scaled by $\gamma > 0$ then both $L$ and $\kappa$ scale by $\gamma$ whereas $\mu$ scales by $1/\gamma^\theta$.  Thus, $4L\ell\cdot (\frac{4L\mu}{\kappa})^\frac{1}{\theta}$ is scale invariant.

\begin{figure}[t]
\captionsetup[subfigure]{justification=centering}
\begin{tabular}{c c c}
    \begin{subfigure}{.31\textwidth}
    \centering
        \includegraphics[width=1\textwidth]{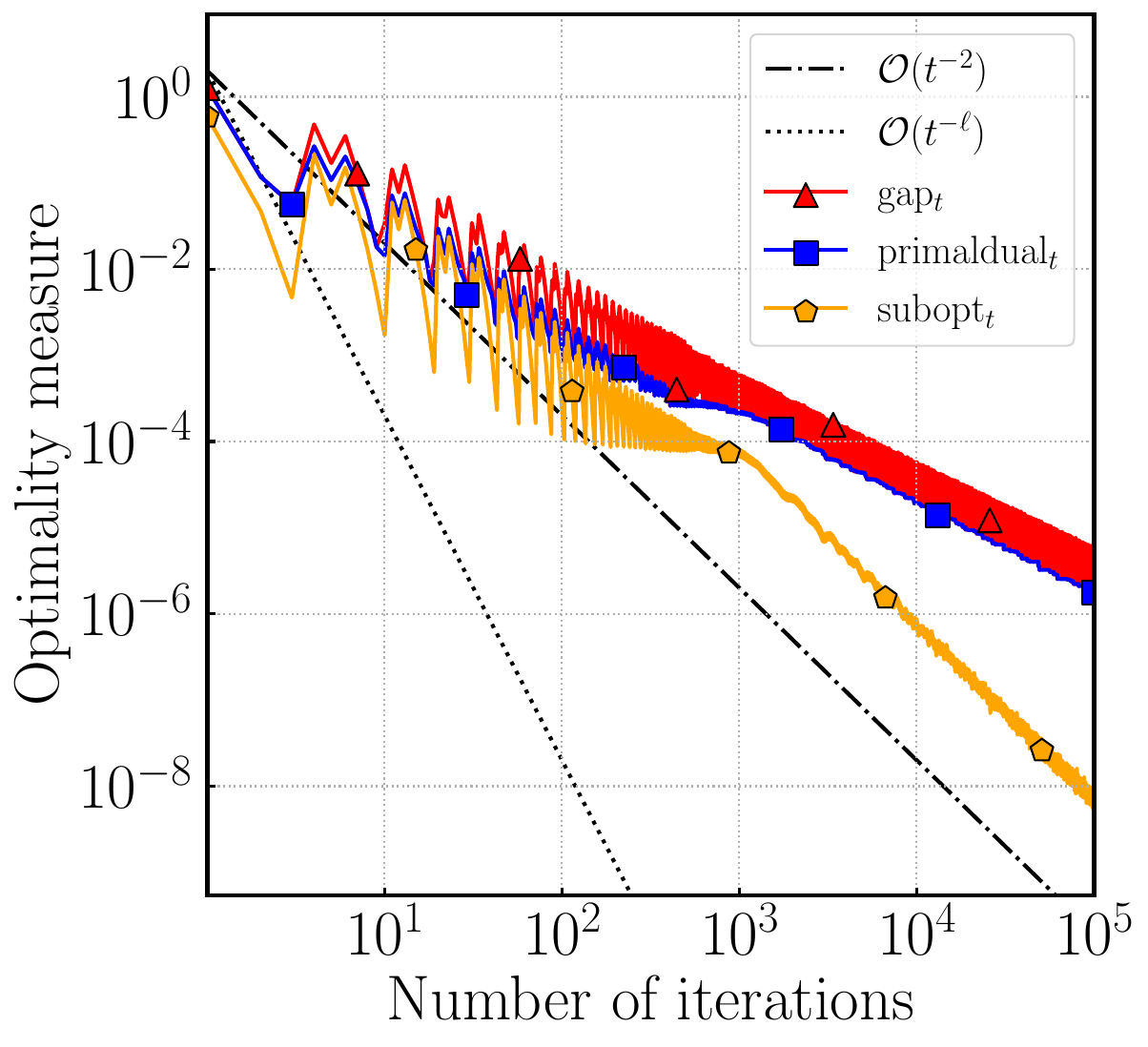}
        \caption{$\kappa = 0.0001$.}\label{fig:polytope.0001}
    \end{subfigure}& 
    \begin{subfigure}{.31\textwidth}
    \centering
        \includegraphics[width=1\textwidth]{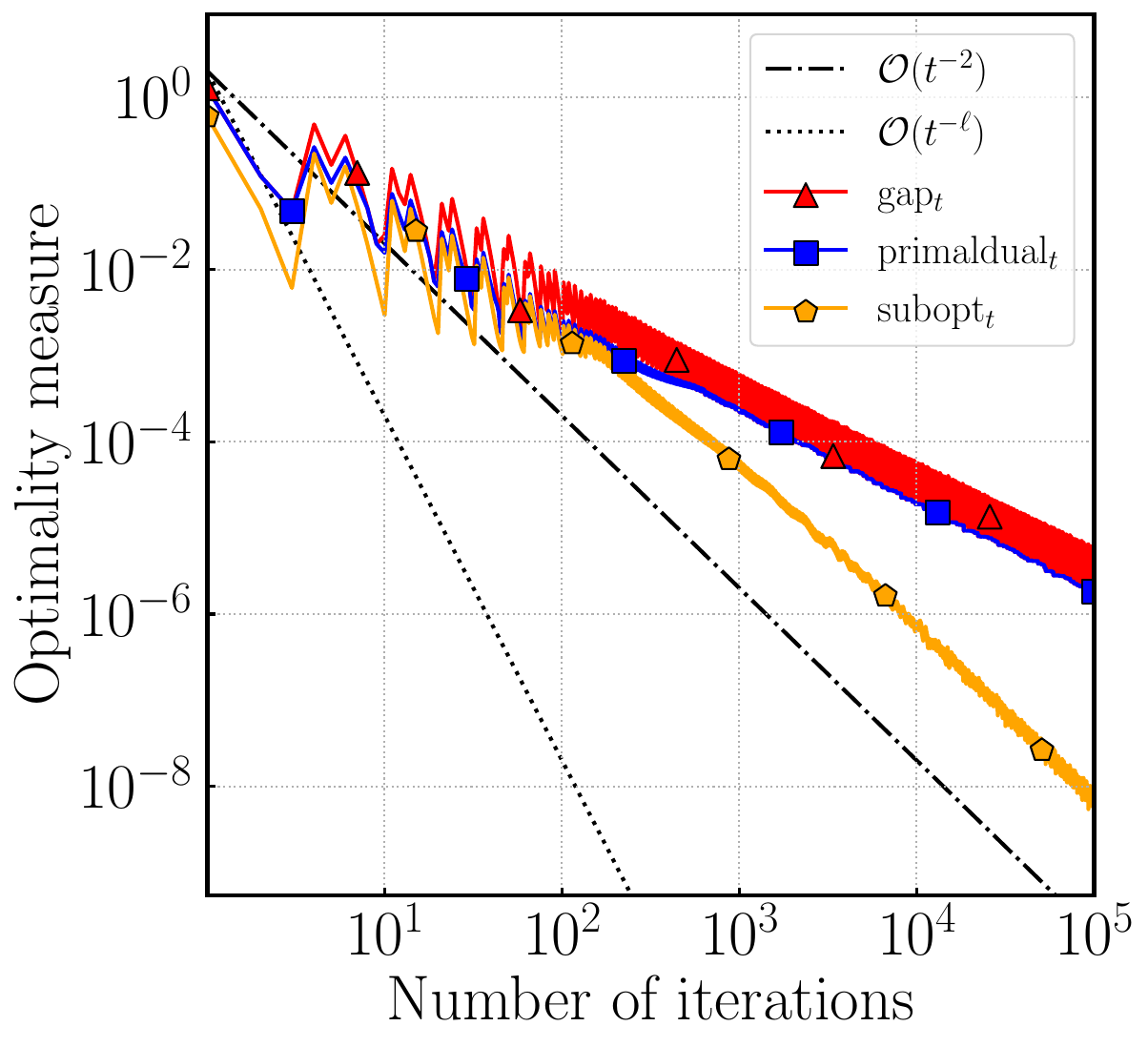}
        \caption{$\kappa = 0.01$.}\label{fig:polytope.01}
    \end{subfigure}& 
    \begin{subfigure}{.31\textwidth}
    \centering
        \includegraphics[width=1\textwidth]{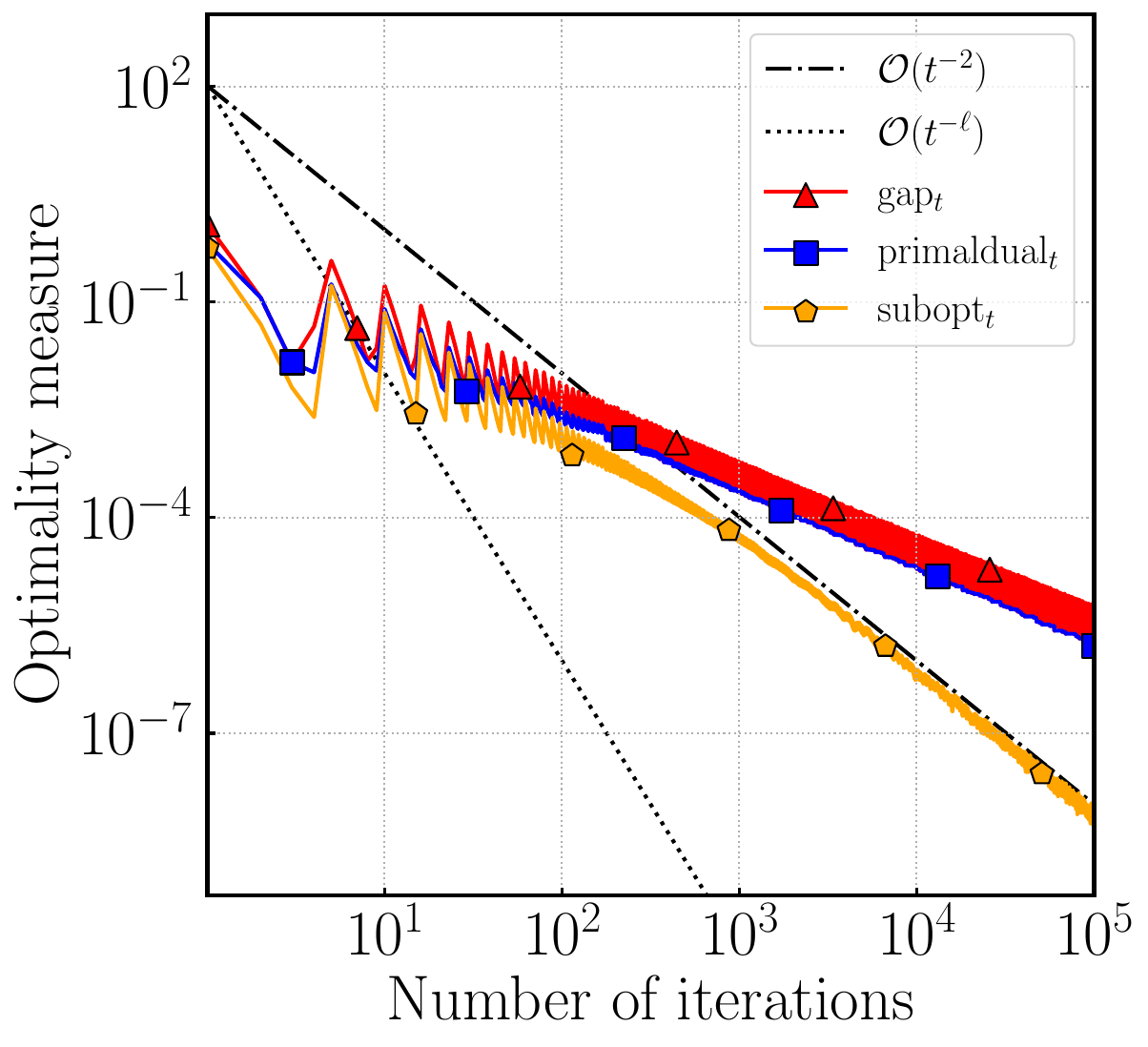}
        \caption{$\kappa=100$.}\label{fig:polytope.1}
    \end{subfigure}
\end{tabular}
\caption{\textbf{Strong $(M, 0)$-growth and relaxed gaps $(m, r)$-growth.} Optimality measure comparison of \fw{} with  step-size $\eta_t = \frac{\ell}{t+\ell}$ for $\ell = 4$ and all $t\in\N$ when the feasible region $\cC\subseteq\R^{n}$ for $n = 100$ is the $\ell_1$-ball and the objective is $f(\xx) =  \frac{1}{2}\|\xx - \yy\|_2^2$ for some random vector $\yy\in\R^n$ such that $\|\xx - \xx^*\|_1 \geq \rho = 0.1 $ for all $\xx\in\cV = \vertices(\cC)$, where $\xx^*=\argmin_{\xx\in\cC}f(x)$, and strict complementarity constant $\kappa > 0$ as in Lemma~\ref{lemma.active}. Axes are in log scale. 
}\label{fig:polytope}
\end{figure}

\subsection{Numerical experiments}\label{sec:gaps_growth.examples}
\noindent
Below, we characterize an example setting for which \fw{} with open-loop step-size rule admits accelerated rates according to Proposition~\ref{prop:polytope}.  For ease of exposition, we take $\cC$ to be the unit $\ell_1$-norm ball.  However, Proposition~\ref{prop:polytope} applies to the much broader class of polytopes.

\begin{example}[Strong $(M, 0)$-growth and relaxed gaps $(m, r)$-growth]\label{ex:polytope}
Let $\cC = \{\xx\in\R^n \mid \|\xx\|_1 \leq 1\}$ be the $\ell_1$-ball. 
Let $f(\xx):= \frac{1}{2}\|\xx-\yy\|_2^2$, where for $\rho\in ]0, 1[$ and $\kappa > 0$, $\yy \in \R^n$ is constructed as follows:
Construct a vector $\zz \in\cC$, $\|\zz\|_\cC = \|\zz\|_1 = 1$ with nonnegative entries such that $z_0 = 0$, $z_1 = 1- \rho$, $z_i \in ]0, 1- \rho]$ for all $i\in \{3, \ldots, n\}$.\footnote{Our implementation of \fw{} starts at the first unit vector when optimizing over the $\ell_1$-ball. To avoid initializing the algorithm in the optimal face, we set $z_0 = 0$.} Then, let $\yy := \kappa \vv + \zz$, where $\vv$ has entry $0$ for $i=1$ and $1$ for all $i\in \{2, \ldots, \lceil n/2\rceil\}$.
By Lemma~\ref{lemma:lp_balls}, $f$ is $1$-smooth and satisfies a $(\sqrt{2n}, 1/2)$-Hölderian error bound with respect to $\|\cdot\|_{\cC}$.
Furthermore,
$\xx^* = \zz$. Thus, $\|\xx - \xx^*\|_\cC \geq \rho$ for all $\xx\in \relativeboundary(\cC^*)$. Furthermore, we have
\begin{align*}
B&=\max_{\yy\in \cC, \xx\in \cC\setminus \cC^*} \frac{|\ip{\nabla f(\yy)}{\xx-\bar\xx}|}{\|\xx-\bar\xx\|_{\cC}} \leq \max_{\yy\in\cC}\|\nabla f (\yy)\|_\infty = 2 - \rho.
\end{align*}
Finally, the strict complementarity condition as in Lemma~\ref{lemma.active} holds, that is, $Q = \left\lceil 4L\ell\cdot (\frac{4L\mu}{\kappa})^\frac{1}{\theta}\right\rceil $ in Assumption~\ref{assum.active}.  
By Proposition~\ref{prop:polytope}, $(\cC, f)$ satisfies the strong $(M, 0)$-growth property and the relaxed gaps growth property, with
$m = \frac{\rho}{2^{2-\theta}\mu}$, $M = \max\left\{\frac{m}{\eta_\Afwt}(4B)^{1-\theta},2\frac{2L+B}{\eta_\Afwt},4L\right\}$, and $Q = \left\lceil 4L\ell\cdot (\frac{4L\mu}{\kappa})^\frac{1}{\theta}\right\rceil$.
\end{example}
In Figure~\ref{fig:polytope}, in the setting of Example~\ref{ex:polytope} for $n = 100$, $\rho = 0.1$, $\kappa \in\{0.0001, 0.01, 100\}$, and \fw{} with open-loop step-size $\eta_t = \frac{\ell}{t+\ell}$ for $\ell = 4$ and all $t\in\N$, we compare $\gap_t$, $\primaldual_t$, and $\subopt_t$.
We also plot $\cO(t^{-\ell})$ and $\cO(t^{-2})$ for better visualization. 
For this setting, as was already the case for the gaps growth setting in Figure~\ref{fig:gaps}, $\fwt$ is not a good predictor for when acceleration kicks in due to the dependence of $\fwt$ on the dimension. We therefore omit $\fwt$ from the plots.

%
For all three settings, we observe that only $\subopt_t$ enjoys the accelerated convergence rate of order $\cO(t^{-2})$ as predicted by Theorem~\ref{thm:rate-gaps-relaxed} unlike in previous settings, see, for example, Figures~\ref{fig:1_strong}, \ref{fig:strong}, and~\ref{fig:weak_unif_hold}, where also $\gap_t$ and $\primaldual_t$ enjoyed the accelerated rates.

Proposition~\ref{prop:gaps_interior} motivates the following example for which the gaps $(m,r)$-growth property holds.  For ease of exposition, we take $\cC$ to be the unit $\ell_1$-norm ball.  However, Proposition~\ref{prop:gaps_interior} applies to the much broader class of compact convex sets.
\begin{example}[Strong $(M, 0)$-growth and gaps $(m,r)$-growth]\label{ex:gaps}
Let $\cC = \{\xx\in\R^n \mid \|\xx\|_1 \leq 1\}$ be the $\ell_1$-ball. 
Let $f(\xx):= \frac{1}{2}\|\xx-\yy\|_2^2$, where for $\rho > 0$, $\yy \in \R^n$ is a normalized vector such that $\|\yy\|_\cC = 1 - \rho$ and thus, $\|\xx - \xx^*\|_\cC=\|\xx - \yy\|_\cC \geq \rho$ for all $\xx\in \relativeboundary(\cC)$.
By Lemma~\ref{lemma:lp_balls}, $f$ is $1$-smooth and satisfies a $(\sqrt{2n}, 1/2)$-Hölderian error bound with respect to $\|\cdot\|_{\cC} = \|\cdot\|_1$.
By Proposition~\ref{prop:gaps_interior}, $(\cC, f)$ satisfies the gaps $(m,r)$-growth property with $m=\frac{\rho}{\sqrt{2n}}$ and $r=\frac{1}{2}$ and, by Lemma~\ref{lemma:lp_balls}, the strong $(M_0, 0)$-growth property with $M_0 = 4$.
\end{example}

%

\begin{figure}[t]
\captionsetup[subfigure]{justification=centering}
\begin{tabular}{c c c}
    \begin{subfigure}{.31\textwidth}
    \centering
        \includegraphics[width=1\textwidth]{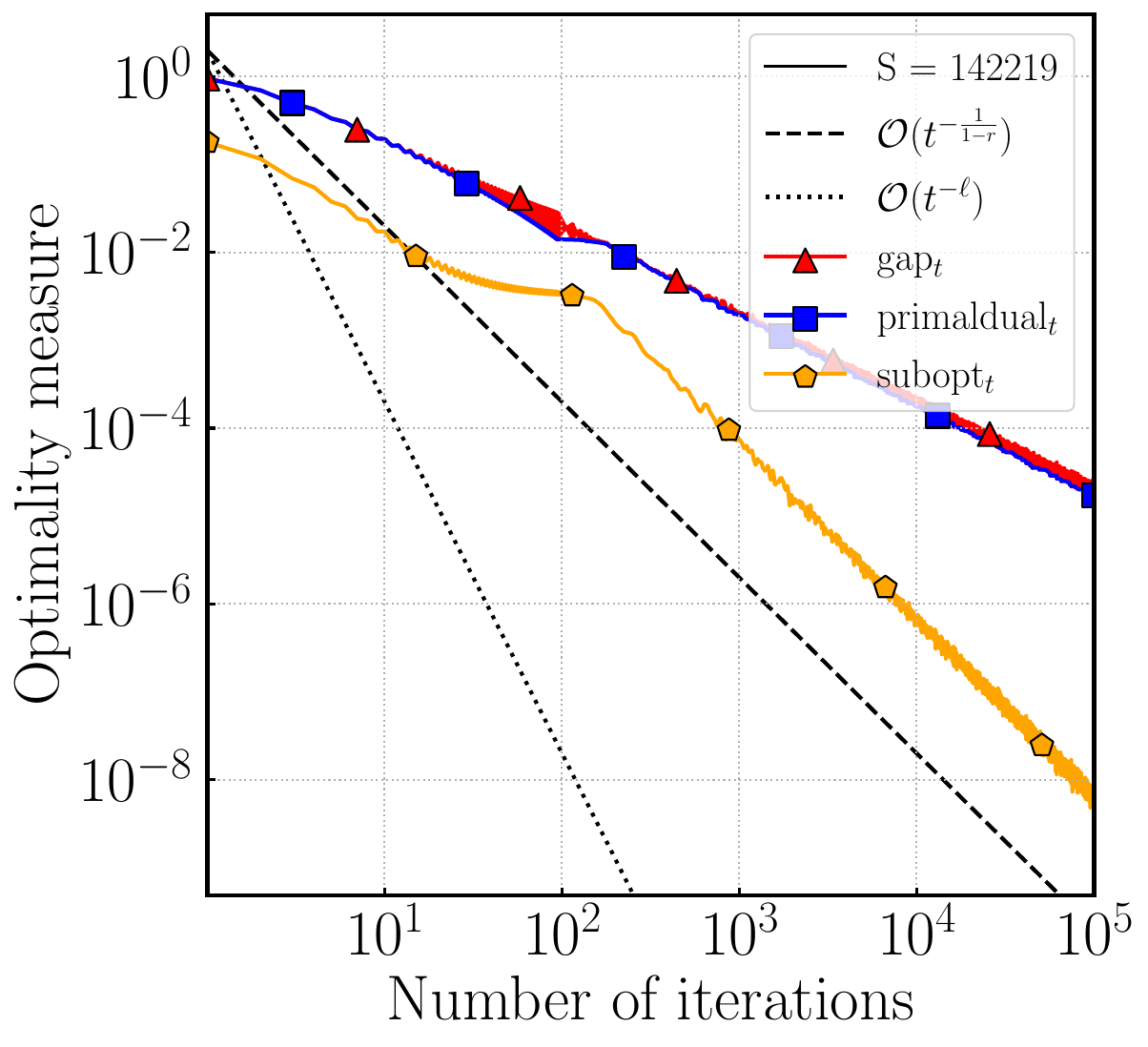}
        \caption{$\rho = 0.3$.}\label{fig:gaps_0.3}
    \end{subfigure}& 
    \begin{subfigure}{.31\textwidth}
    \centering
        \includegraphics[width=1\textwidth]{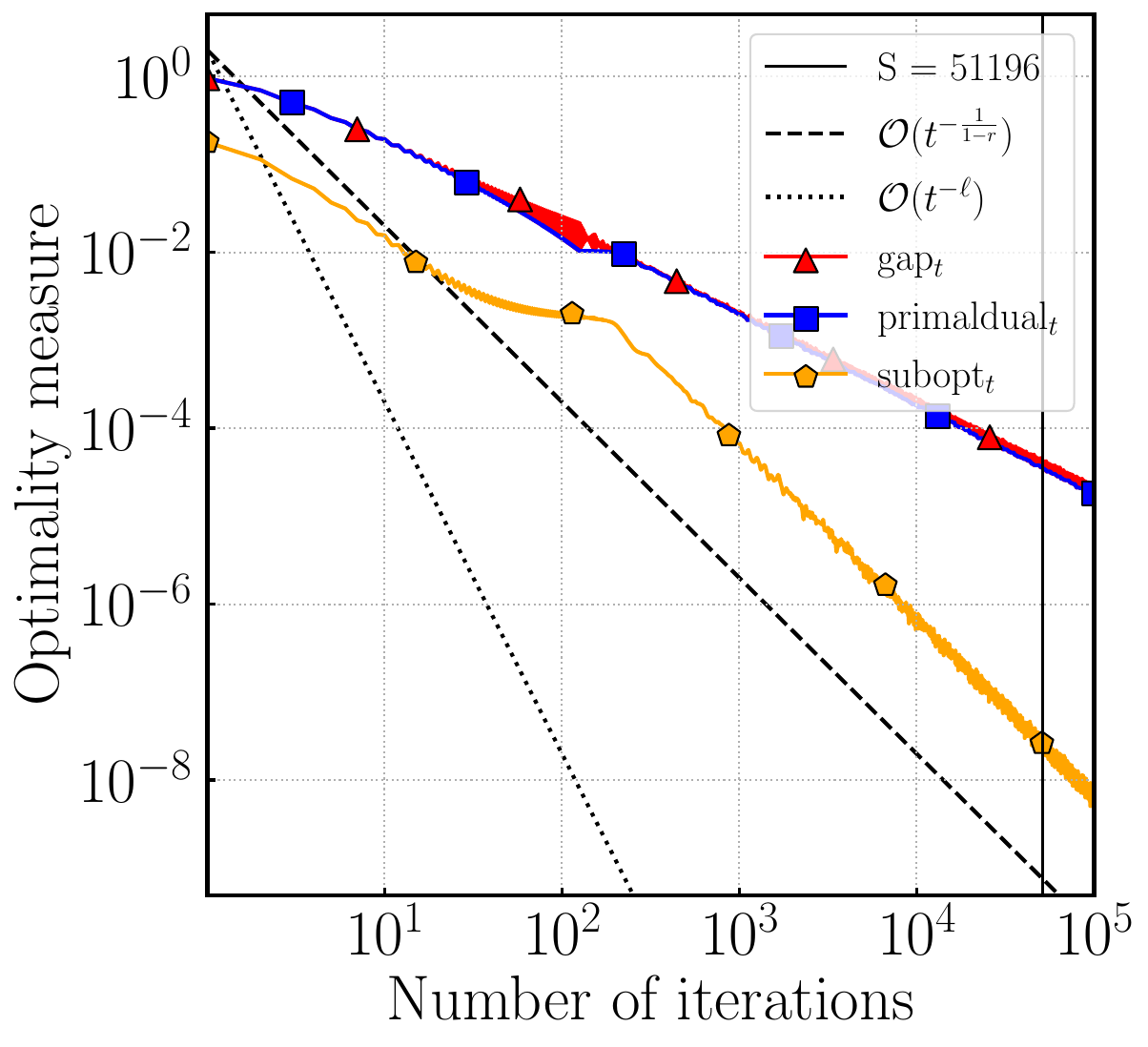}
        \caption{$\rho = 0.5$.}\label{fig:gaps_0.5}
    \end{subfigure}& 
    \begin{subfigure}{.31\textwidth}
    \centering
        \includegraphics[width=1\textwidth]{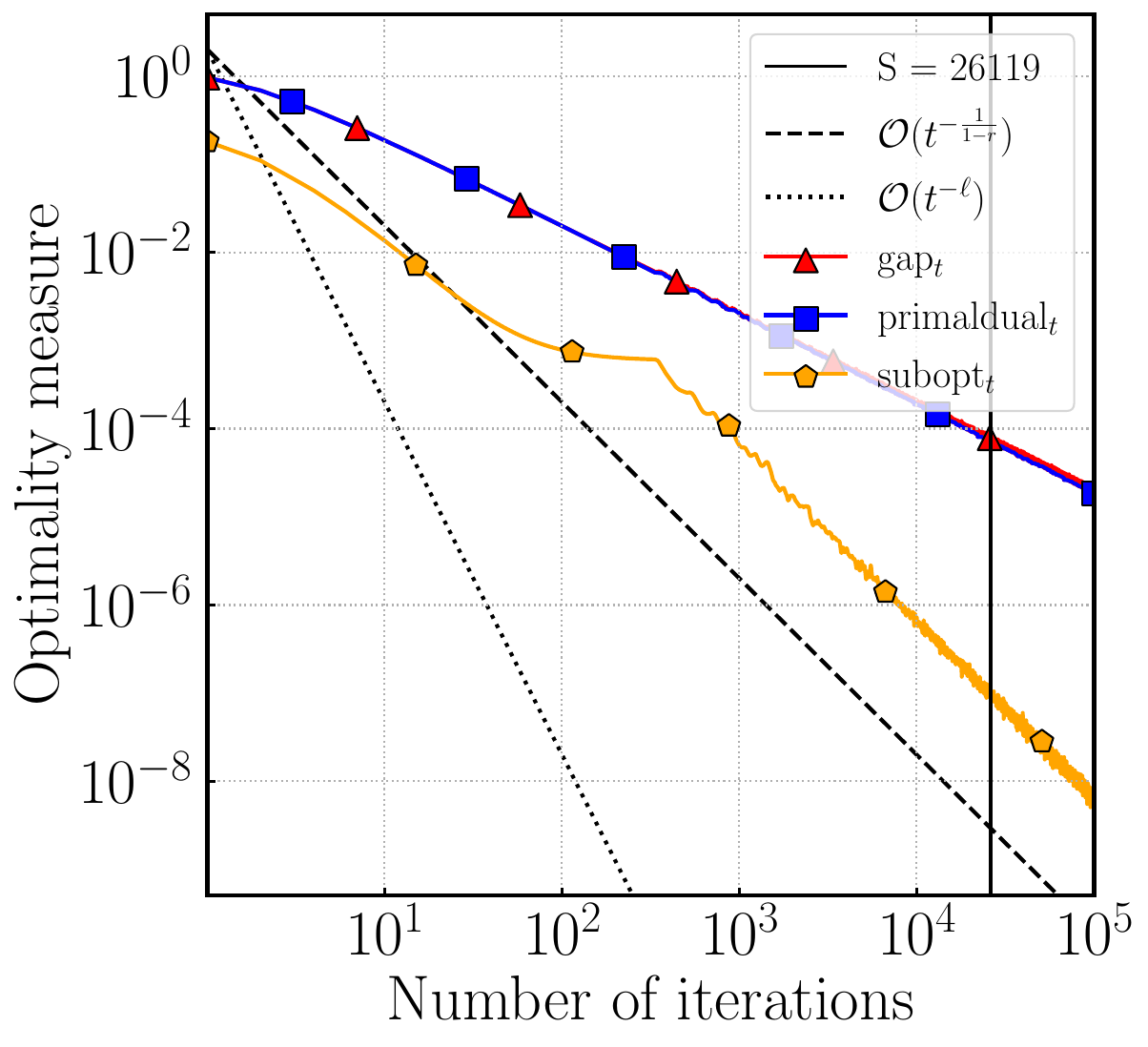}
        \caption{$\rho = 0.7$.}\label{fig:gaps_0.7}
    \end{subfigure}
\end{tabular}
\caption{\textbf{Strong $(M, 0)$-growth and gaps $(m, r)$-growth.} Optimality measure comparison of \fw{} with  step-size $\eta_t = \frac{\ell}{t+\ell}$ for $\ell = 4$ and all $t\in\N$ when the feasible region $\cC\subseteq\R^n$ for $n = 100$ is the $\ell_1$-ball and the objective is $f(\xx) =  \frac{1}{2}\|\xx - \yy\|_2^2$ for some random vector $\yy\in\R^n$ such that $\|\yy\|_1 = 1 -\rho$ for some $\rho\in ]0, 1[$. Axes are in log scale. 
}\label{fig:gaps}
\end{figure}

In Figure~\ref{fig:gaps}, in the setting of Example~\ref{ex:gaps} for $n = 100$, $\rho\in\{0.3, 0.5, 0.7\}$, and \fw{} with open-loop step-size $\eta_t = \frac{\ell}{t+\ell}$ for $\ell = 4$ and all $t\in\N$, we compare $\gap_t$, $\primaldual_t$, and $\subopt_t$. We also plot $\cO(t^{-\ell})$ and $\cO(t^{-\frac{1}{1-r}})$ for better visualization. For these settings, gaps $(m, r)$-growth, $r = 1/2$, holds and Theorem~\ref{thm:rate-gaps} applies. Since Theorem~\ref{thm:rate-gaps} only predicts acceleration after a problem-dependent iteration $\fwt$ has been reached, we also add a vertical line indicating $\fwt$ to the plots.

For all three settings, observe that only $\subopt_t$ enjoys the accelerated convergence rate of order $\cO(t^{-2})$ as predicted by Theorem~\ref{thm:rate-gaps}, unlike in previous settings, see, for example, Figures~\ref{fig:1_strong}, \ref{fig:strong}, and~\ref{fig:weak_unif_hold}, where also $\gap_t$ and $\primaldual_t$ enjoyed the accelerated rates.

Note that $\fwt$ is a very coarse predictor for when acceleration kicks in. This is explained by the fact that, in our affine-invariant formulation, the Hölderian error bound depends on the dimension.
\section{Other Numerical Experiments}\label{sec:numerical_experiments}
In this section, we present additional numerical experiments on problem instances from data science applications. The aim of our empirical comparisons is to demonstrate that the accelerated rates derived in this paper a) occur on real data and not just in toy examples and b) study how \fw{} with open-loop step-sizes admits very different convergence rates for problems that differ only by very few parameters. All of our numerical experiments are implemented in \textsc{Python} and performed on an Nvidia GeForce RTX 3080 GPU with 10GB RAM and an Intel Core i7 11700K 8x CPU at 3.60GHz with 64 GB RAM. Our code is publicly available on 
\href{https://github.com/ZIB-IOL/affine_invariant_open_loop_fw}{GitHub}.

\subsection{Ablation study for open-loop step-sizes}

We study the impact of different values of $\ell\in\N_{\geq 1}$ on the performance of $\fw$ with open-loop step-size $\eta_t = \frac{\ell}{t+\ell}$ when the unconstrained optimizer is in the relative exterior of a uniformly convex feasible region.

In Figure~\ref{fig:ablation}, in the setting of Example~\ref{ex:strong} for $n = 100$, $p =2$, and $\lambda = 0.2$, we compare the convergence rate of \fw{} with open-loop step-sizes $\eta_t = \frac{\ell}{t+\ell}$ for $\ell \in\{1, 2, 5, 10\}$ and all $t\in\N$. The comparison is performed for the three different optimality measures $\gap_t$, $\primaldual_t$, $\subopt_t$. We also plot $\cO(t^{-2})$ for better visualization.

\begin{figure}[ht]
\captionsetup[subfigure]{justification=centering}
\begin{tabular}{c c c}
    \begin{subfigure}{.31\textwidth}
    \centering
        \includegraphics[width=1\textwidth]{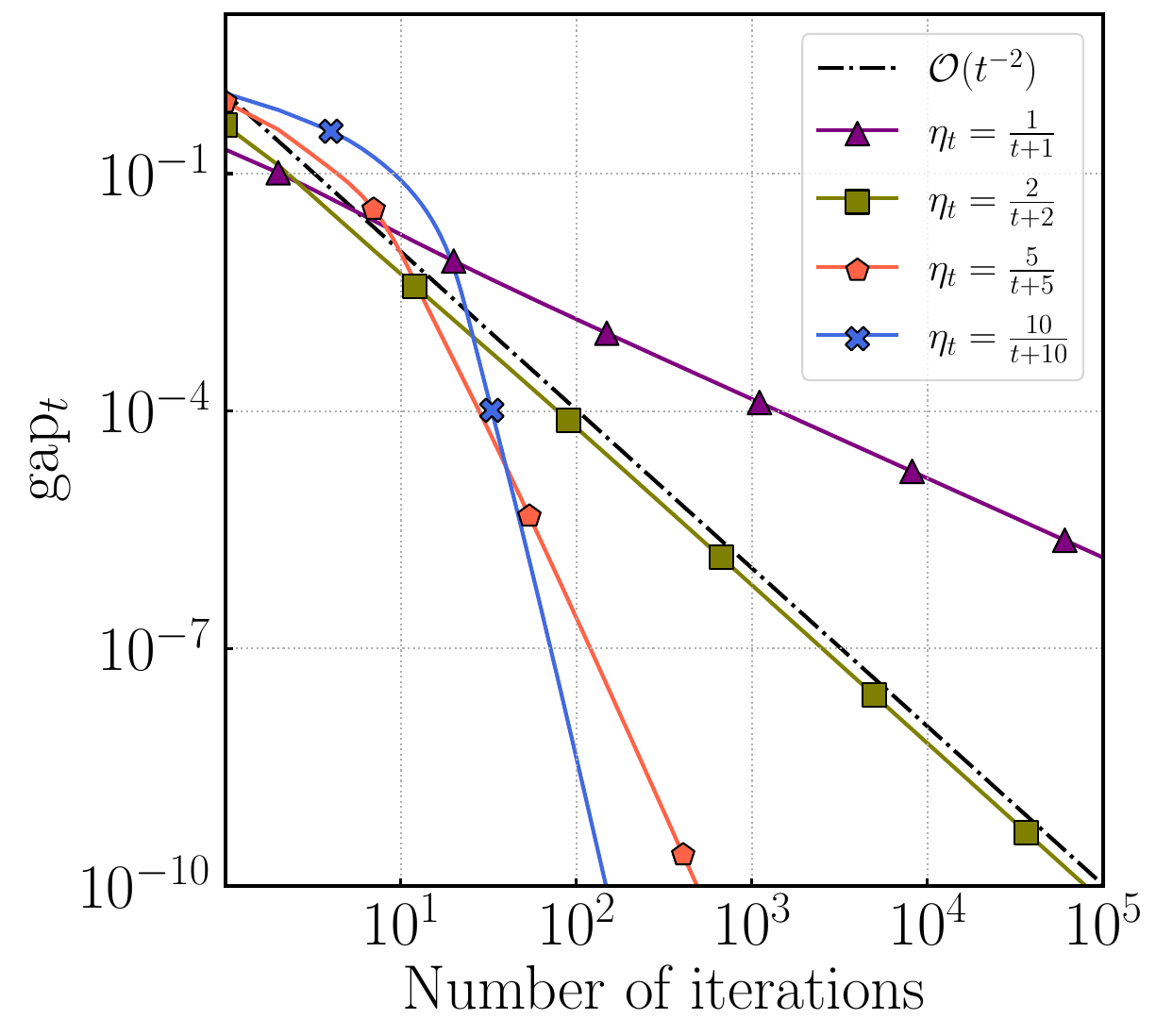}
        \caption{$\gap_t$.}\label{fig:ablation_gap}
    \end{subfigure}& 
    \begin{subfigure}{.31\textwidth}
    \centering
        \includegraphics[width=1\textwidth]{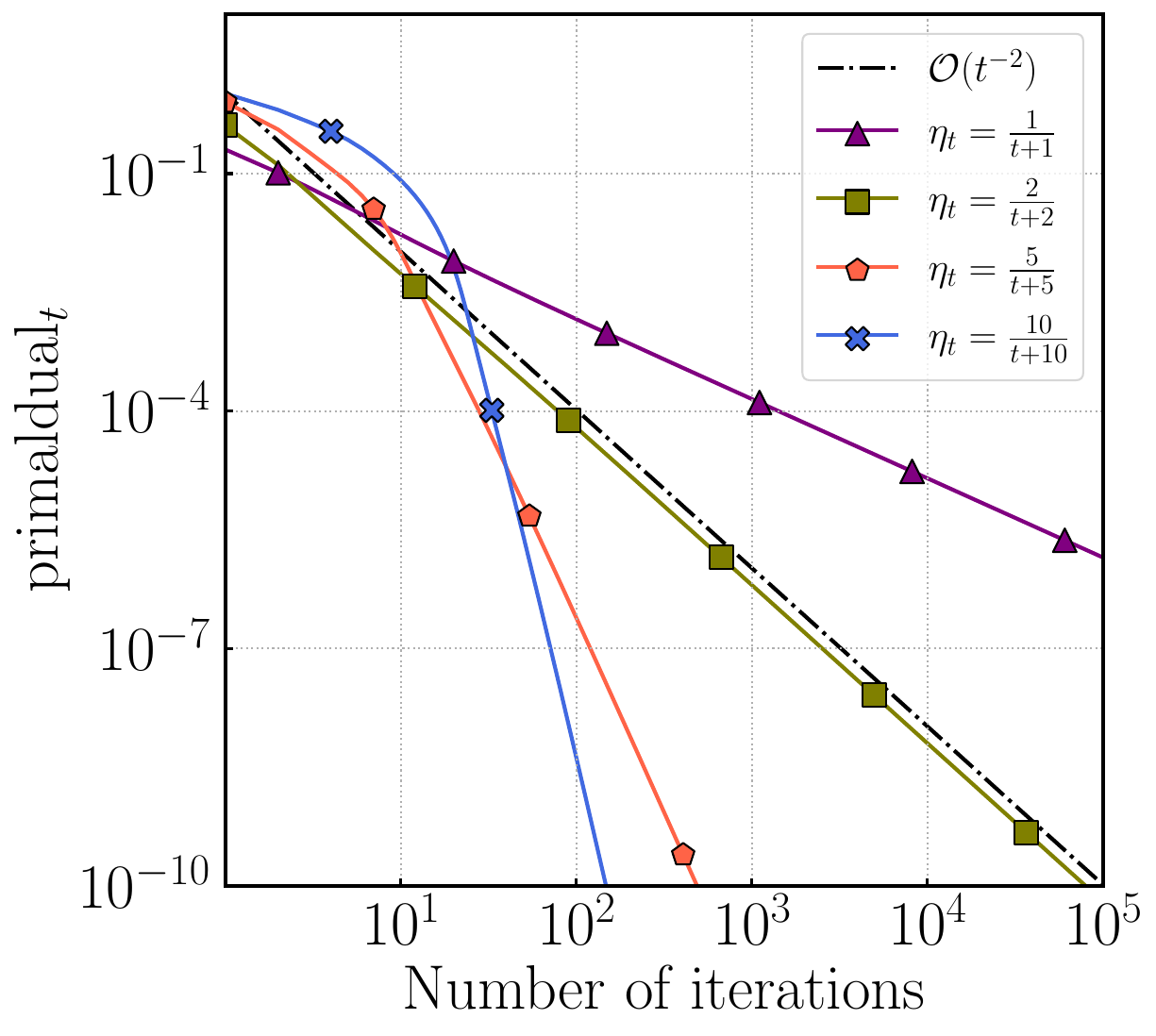}
        \caption{$\primaldual_t$.}\label{fig:ablation_primaldual}
    \end{subfigure}& 
    
    \begin{subfigure}{.31\textwidth}
    \centering
        \includegraphics[width=1\textwidth]{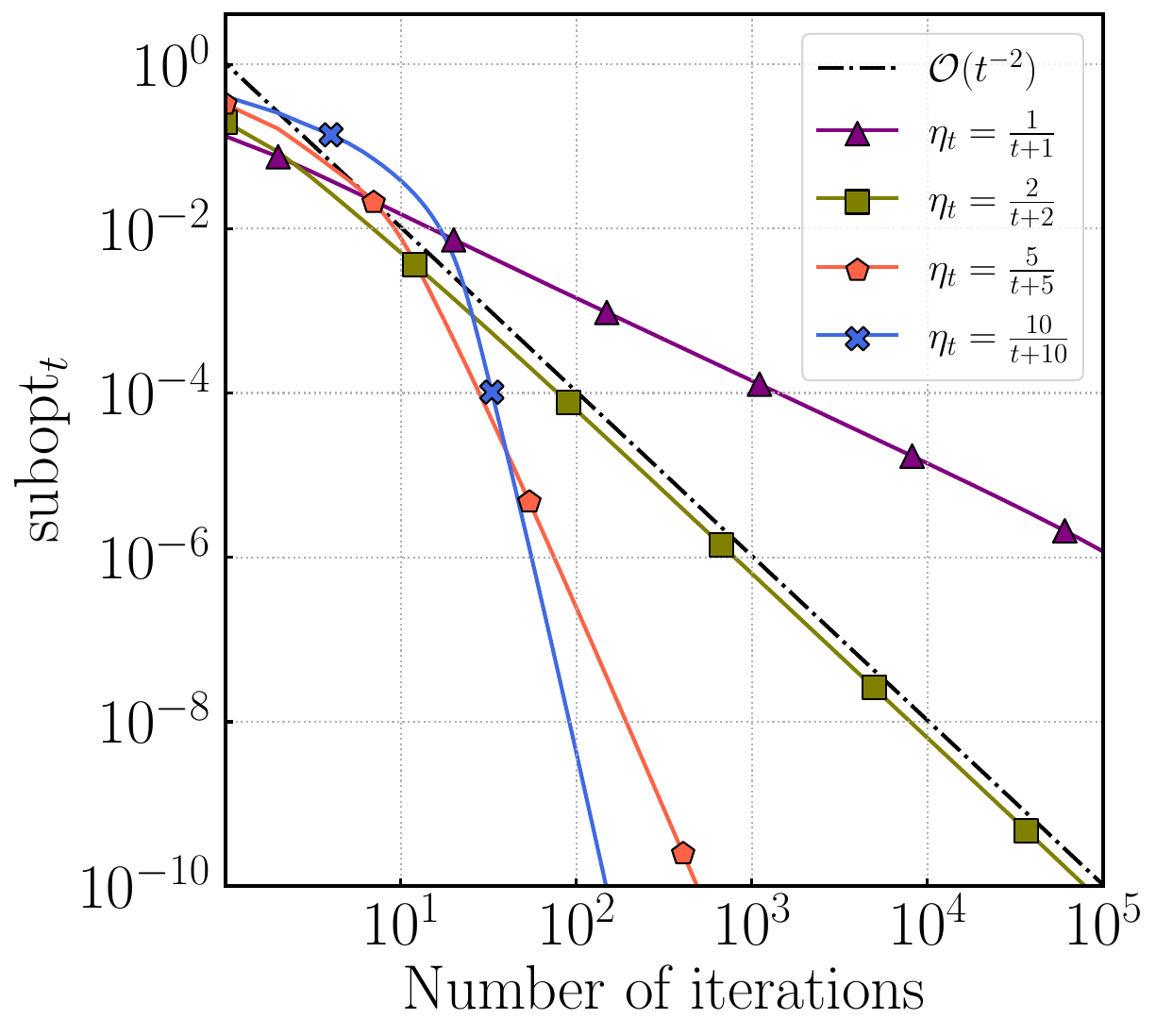}
        \caption{$\subopt_t$.}\label{fig:ablation_subopt}
    \end{subfigure}
\end{tabular}
\caption{\textbf{Ablation study for open-loop step-sizes.} Convergence rate comparison of \fw{} with step-sizes $\eta_t = \frac{\ell}{t+\ell}$ for $\ell \in\{1, 2, 5, 10\}$ and all $t\in\N$ when the feasible region $\cC\subseteq\R^n$ for $n = 100$ is the $\ell_2$-ball and the objective is $f(\xx) =  \frac{1}{2}\|\xx - \yy\|_2^2$ for some random vector $\yy\in\R^n$ with $\|\yy\|_2 = 1 + \lambda$, where $\lambda = 0.2$. The comparison is performed for the three different optimality measures $\gap_t$, $\primaldual_t$, and $\subopt_t$. Axes are in log scale.
}\label{fig:ablation}
\end{figure}


In Figures~\ref{fig:ablation_gap}, \ref{fig:ablation_primaldual}, and~\ref{fig:ablation_subopt}, we observe that the larger $\ell\in\N$, the later acceleration of order $\cO(t^{-\ell})$ kicks in, as predicted by Theorem~\ref{thm:1_strong}.

\subsection{\fw{} with line-search can be slower than with open-loop step-sizes}
\label{sec.num_w_lb}
As established by \cite{wolfe1970convergence, bach2021effectiveness, wirth2023acceleration} and again discussed in Section~\ref{sec.suff_relaxed_gaps_m_r} of this paper, \fw{} with open-loop step-size $\eta_t = \frac{\ell}{t+\ell}$ for $\ell\in\N_{\geq 2}$ sometimes converges faster than \fw{} with line-search in the setting of the lower bound due to \cite{wolfe1970convergence}. As it turns out, Example~\ref{ex:polytope} characterizes such a setting.

In Figure~\ref{fig:polytope_ls_ol}, in the setting of Example~\ref{ex:polytope} for $n = 100$, $\rho = 0.1$, $\kappa = 0.0001$, and \fw{} with step-sizes $\eta_t = \frac{\ell}{t+\ell}$ for $\ell\in\{1, 2, 4\}$ and all $t\in\N$ and line-search, we compare $\gap_t$, $\primaldual_t$, and $\subopt_t$. We also plot $\cO(t^{-2})$ for better visualization.

\begin{figure}[ht]
\captionsetup[subfigure]{justification=centering}
\begin{tabular}{c c c}
    \begin{subfigure}{.31\textwidth}
    \centering
        \includegraphics[width=1\textwidth]{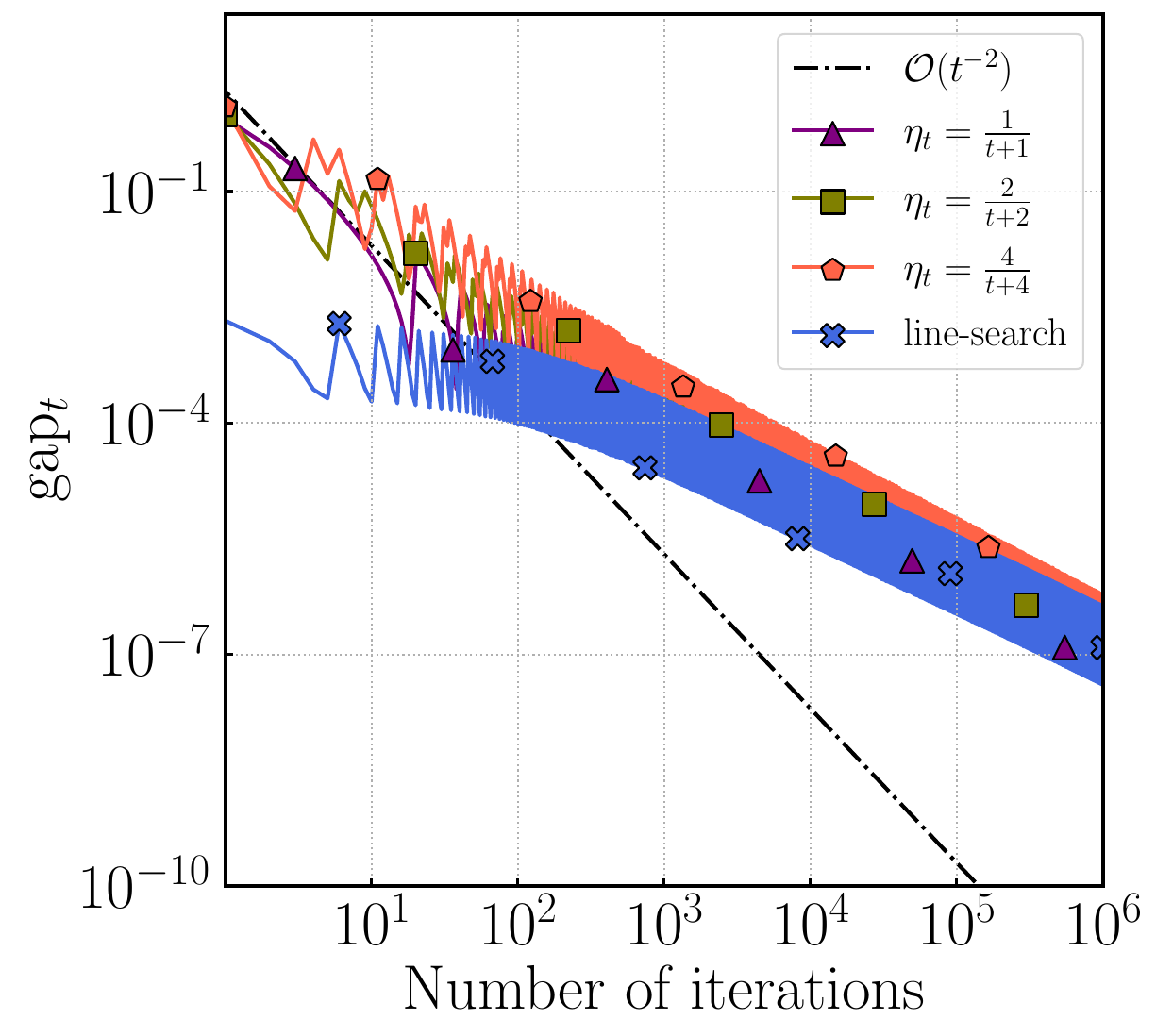}
        \caption{$\gap_t$.}\label{fig:polytope_ls_ol_gap}
    \end{subfigure}& 
    \begin{subfigure}{.31\textwidth}
    \centering
        \includegraphics[width=1\textwidth]{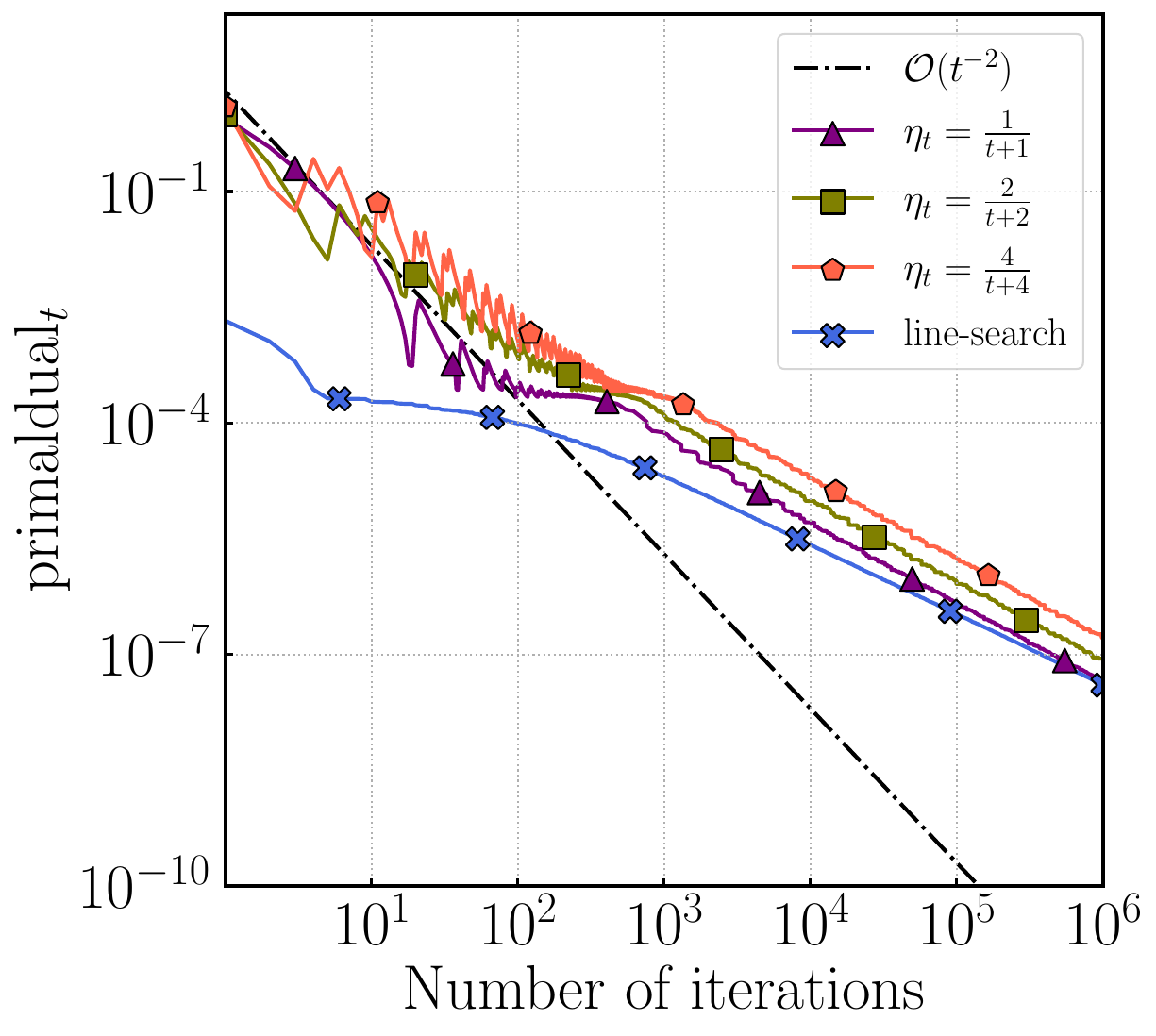}
        \caption{$\primaldual_t$.}\label{fig:polytope_ls_ol_primaldual}
    \end{subfigure}& 
    
    \begin{subfigure}{.31\textwidth}
    \centering
        \includegraphics[width=1\textwidth]{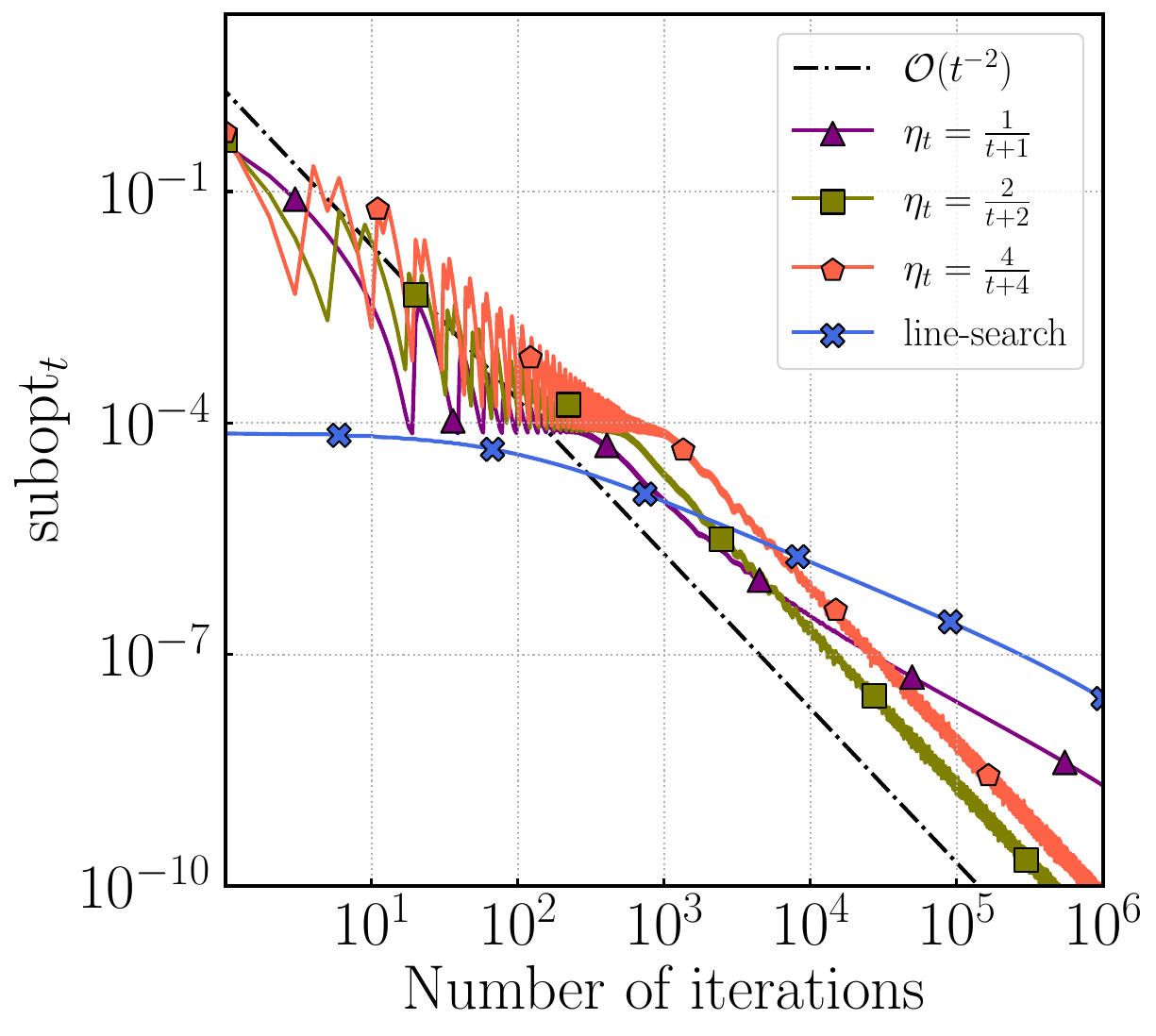}
        \caption{$\subopt_t$.}\label{fig:polytope_ls_ol_subopt}
    \end{subfigure}
\end{tabular}
\caption{\textbf{\fw{} with line-search can be slower than with open-loop step-sizes.} Convergence rate comparison of \fw{} with step-sizes $\eta_t = \frac{\ell}{t+\ell}$ for $\ell\in\{1, 2, 4\}$ and all $t\in\N$ and line-search when the feasible region $\cC\subseteq\R^n$ for $n = 100$ is the $\ell_1$-ball and the objective is $f(\xx) =  \frac{1}{2}\|\xx - \yy\|_2^2$ for some random vector $\yy\in\R^n$ such that $\|\xx - \xx^*\|_1 \geq \rho = 0.1 $ for all $\xx\in\cV = \vertices(\cC)$, where $\xx^*=\argmin_{\xx\in\cC}f(x)$, and strict complementarity constant $\kappa =0.0001$ as in Lemma~\ref{lemma.active}. The comparison is performed for the three different optimality measures $\gap_t$, $\primaldual_t$, and $\subopt_t$.  Axes are in log scale.
}\label{fig:polytope_ls_ol}
\end{figure}

In Figures~\ref{fig:polytope_ls_ol_gap} and \ref{fig:polytope_ls_ol_primaldual}, we observe that $\gap_t$ and $\primaldual_t$ both converge at a rate of order $\cO(t^{-1})$ for \fw{} with any of the step-size rules. In Figure~\ref{fig:polytope_ls_ol_subopt}, we observe that \fw{} with step-sizes $\eta_t = \frac{1}{t+1}$ and line-search converges at a rate of order $\cO(t^{-1})$ and \fw{} with step-sizes $\eta_t = \frac{\ell}{t+\ell}$ for $\ell\in\{2, 4\}$ converges at a rate of order $\cO(t^{-2})$.

\subsection{Constrained regression}\label{sec:constrained regression}
\begin{figure}[ht]
\captionsetup[subfigure]{justification=centering}
\begin{tabular}{c c c}
    \begin{subfigure}{.31\textwidth}
    \centering
        \includegraphics[width=1\textwidth]{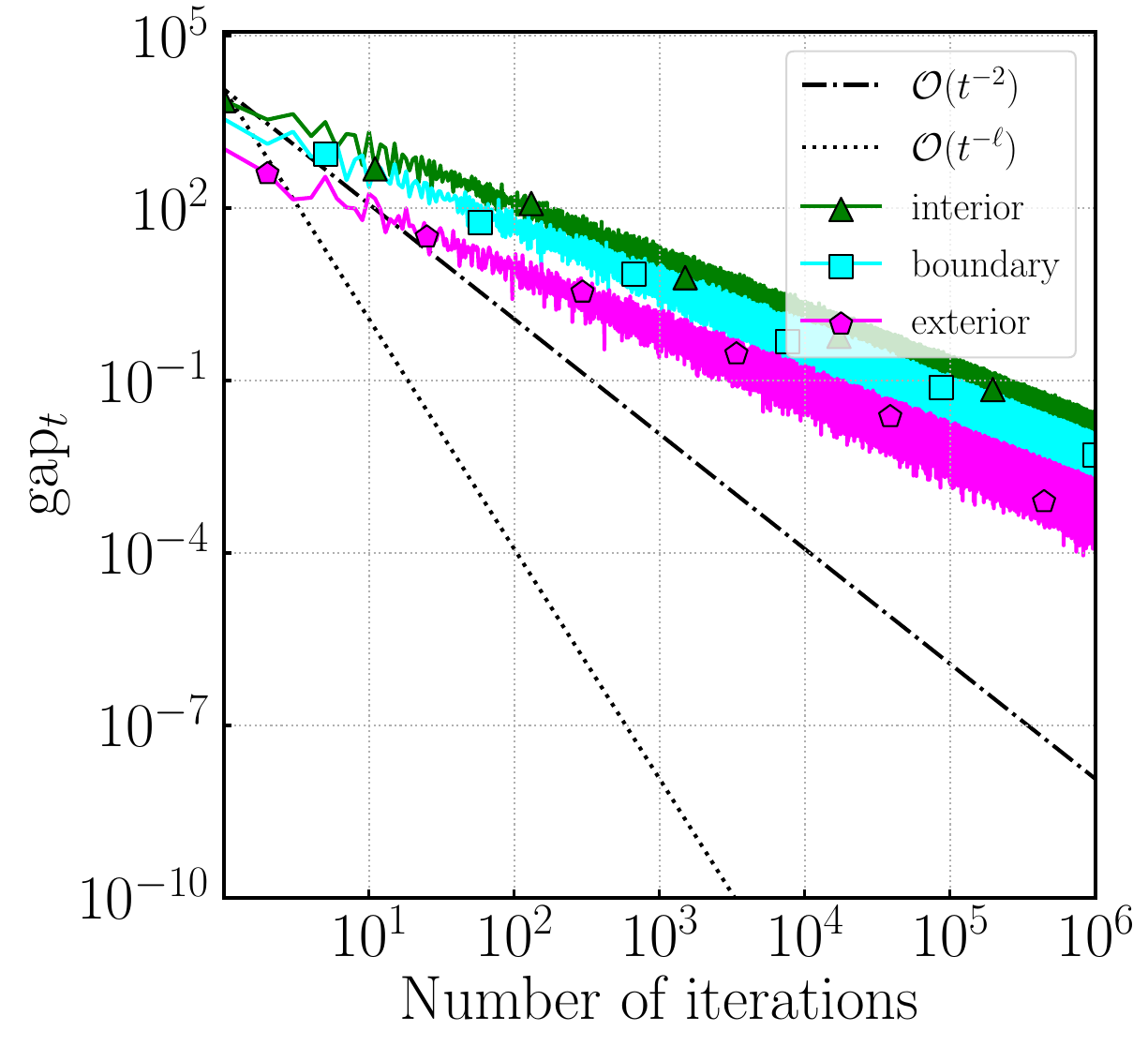}
        \caption{$\ell_1$-ball, $\gap_t$.}\label{fig:reg_1_gap}
    \end{subfigure}& 
    \begin{subfigure}{.31\textwidth}
    \centering
        \includegraphics[width=1\textwidth]{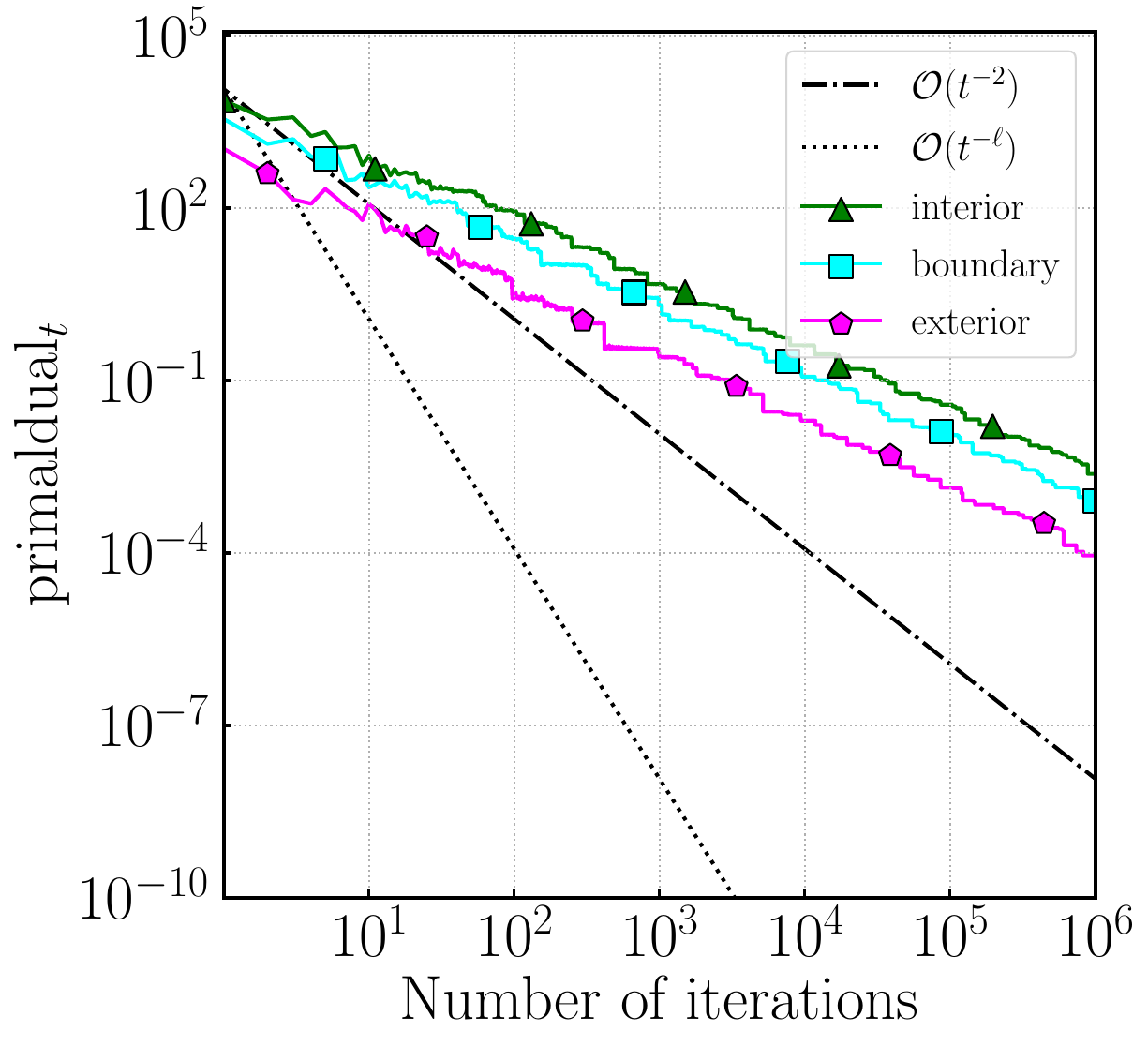}
        \caption{$\ell_1$-ball, $\primaldual_t$.}\label{fig:reg_1_primaldual}
    \end{subfigure}& 
    \begin{subfigure}{.31\textwidth}
    \centering
        \includegraphics[width=1\textwidth]{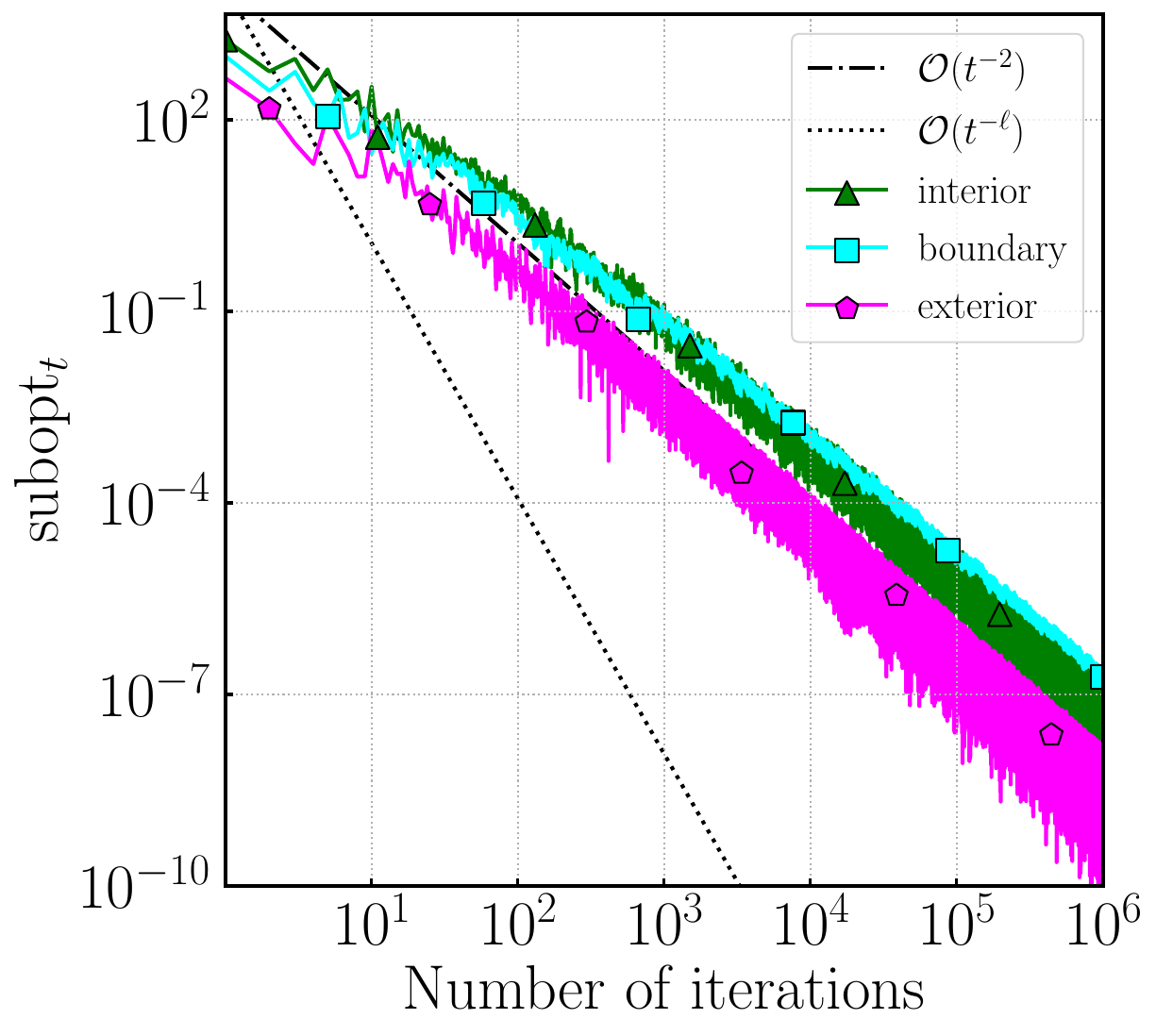}
        \caption{$\ell_1$-ball, $\subopt_t$.}\label{fig:reg_1_subopt}
    \end{subfigure}\\
        \begin{subfigure}{.31\textwidth}
    \centering
        \includegraphics[width=1\textwidth]{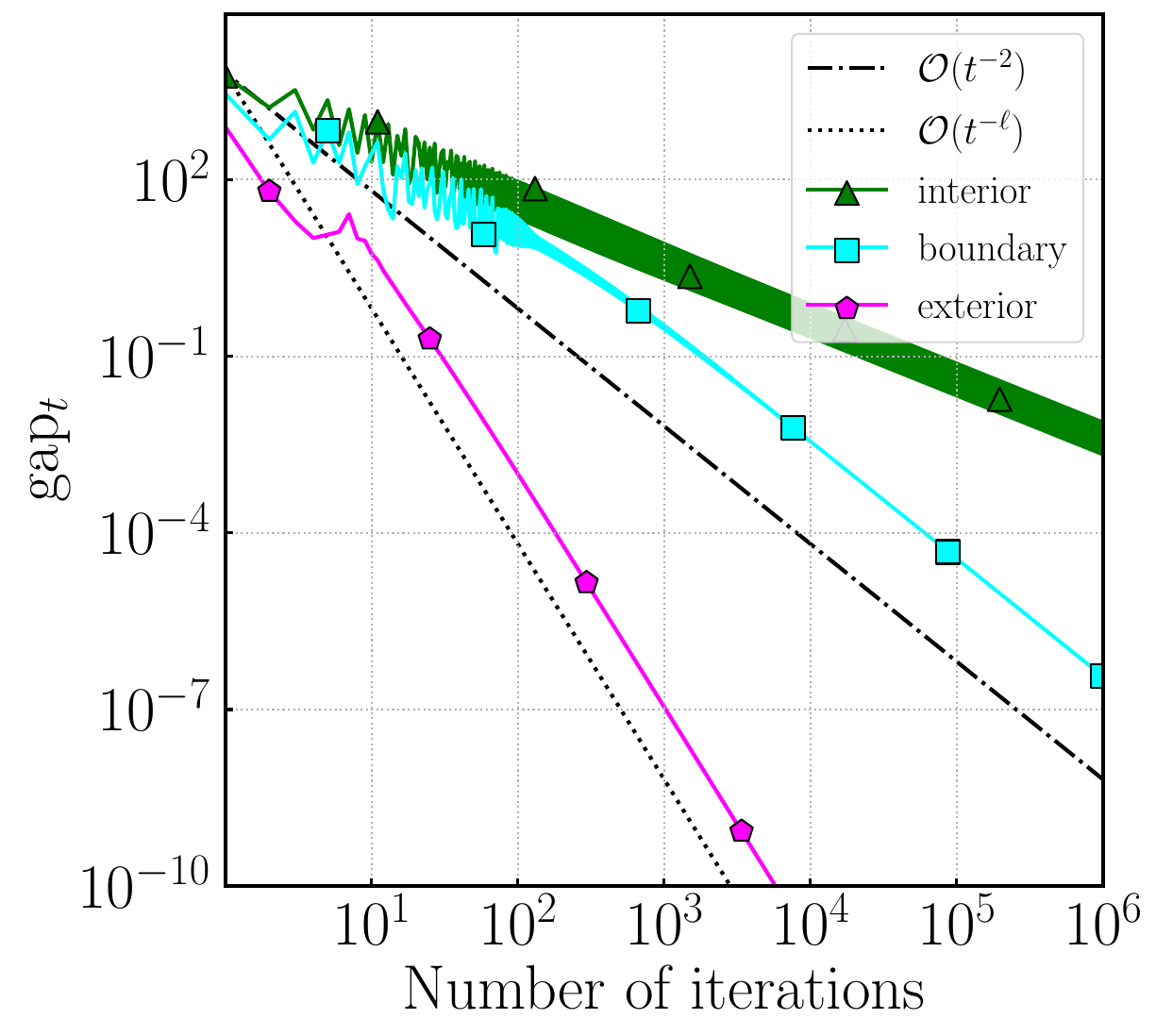}
        \caption{$\ell_2$-ball, $\gap_t$.}\label{fig:reg_2_gap}
    \end{subfigure}& 
    \begin{subfigure}{.31\textwidth}
    \centering
        \includegraphics[width=1\textwidth]{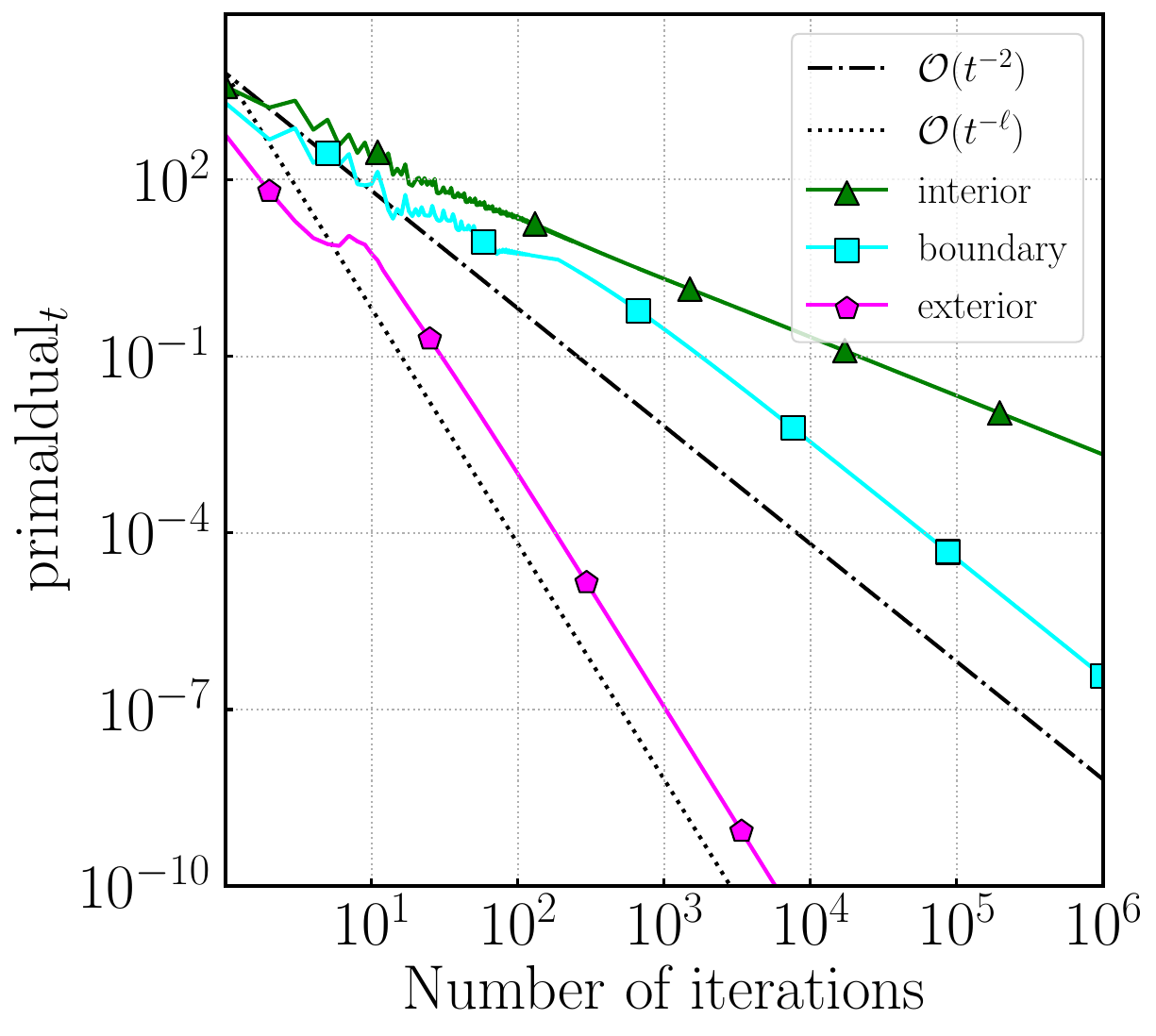}
        \caption{$\ell_2$-ball, $\primaldual_t$.}\label{fig:reg_2_primaldual}
    \end{subfigure}& 
    \begin{subfigure}{.31\textwidth}
    \centering
        \includegraphics[width=1\textwidth]{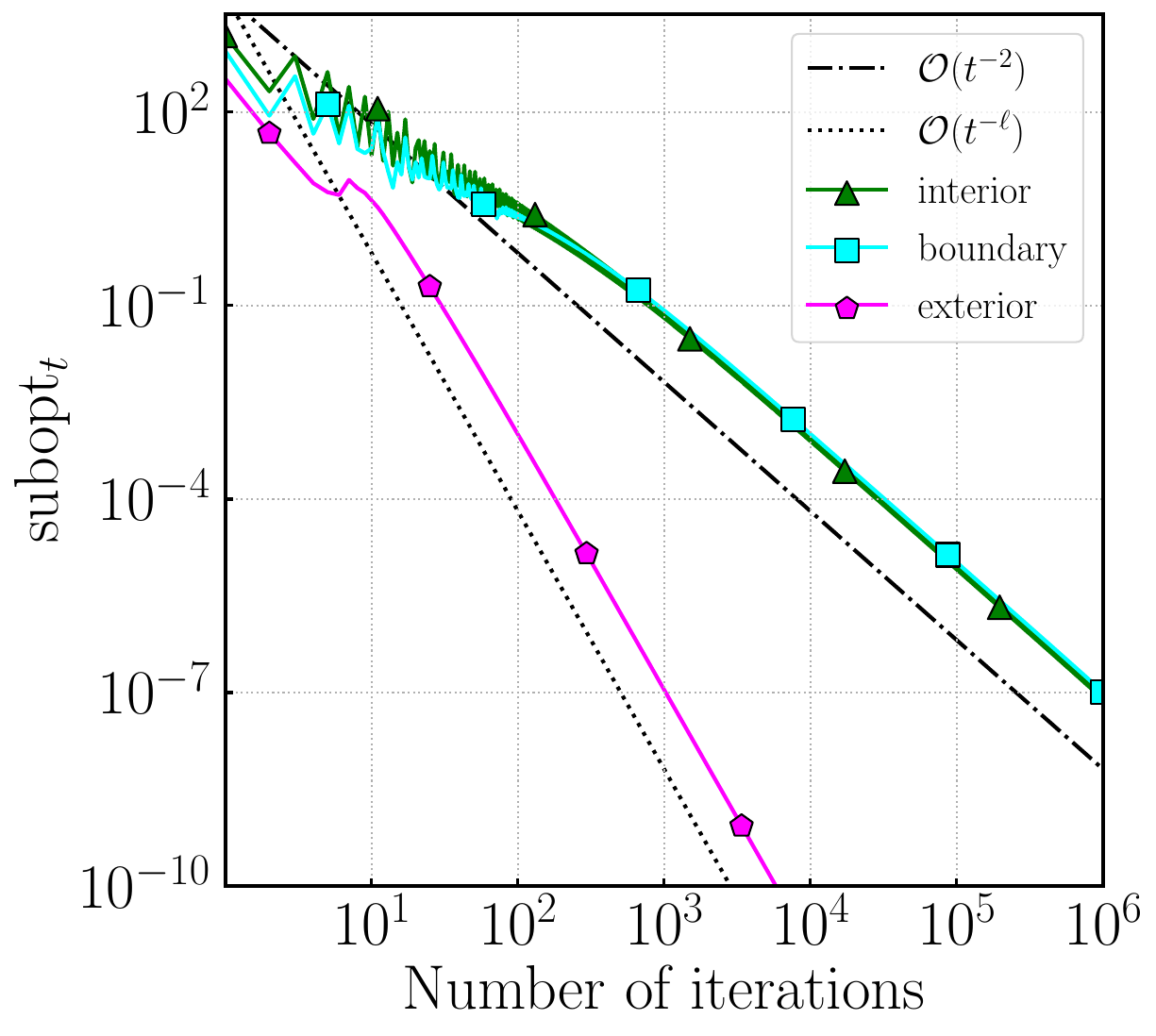}
        \caption{$\ell_2$-ball, $\subopt_t$.}\label{fig:reg_2_subopt}
    \end{subfigure}\\
            \begin{subfigure}{.31\textwidth}
    \centering
        \includegraphics[width=1\textwidth]{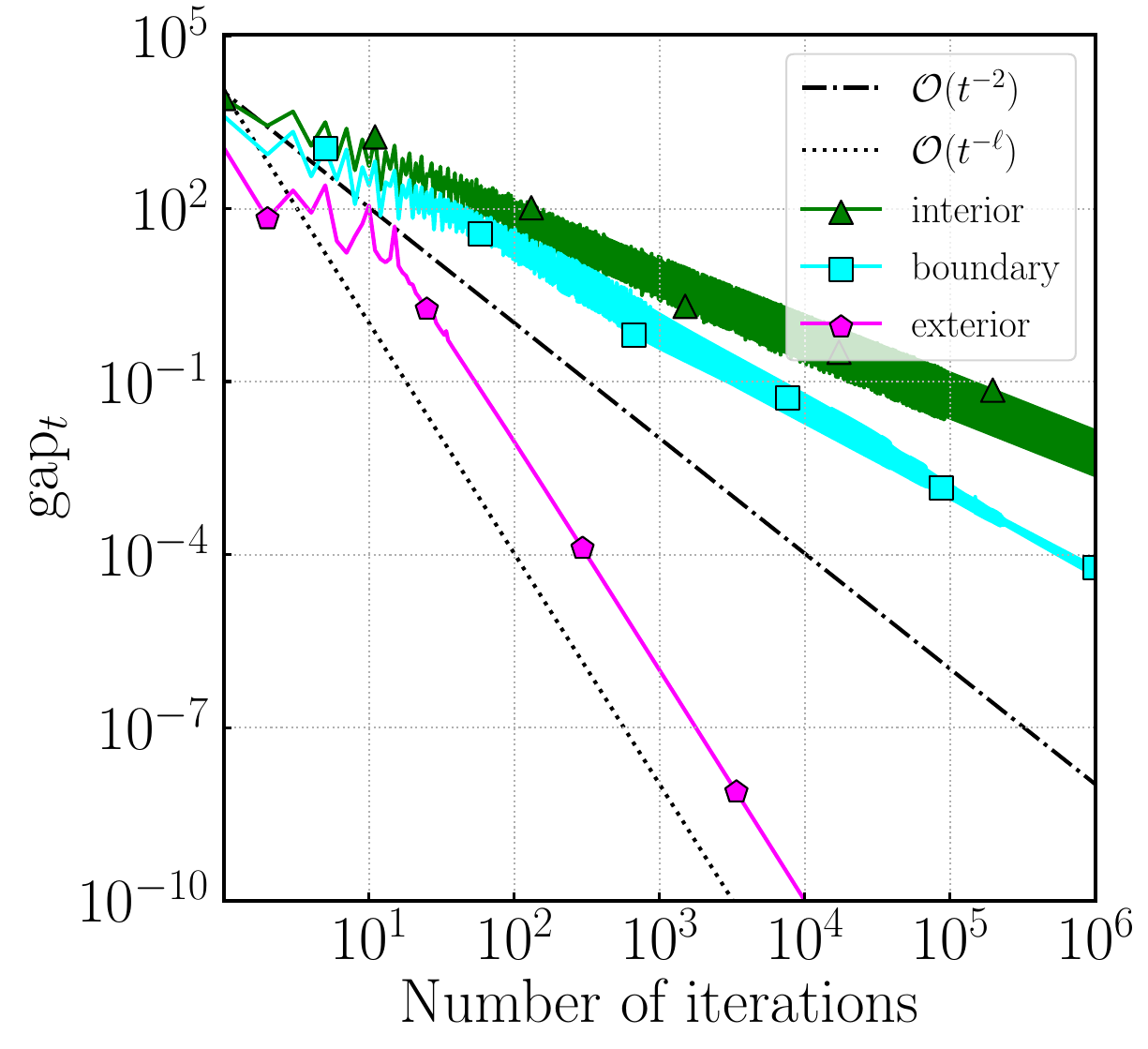}
        \caption{$\ell_5$-ball, $\gap_t$.}\label{fig:reg_5_gap}
    \end{subfigure}& 
    \begin{subfigure}{.31\textwidth}
    \centering
        \includegraphics[width=1\textwidth]{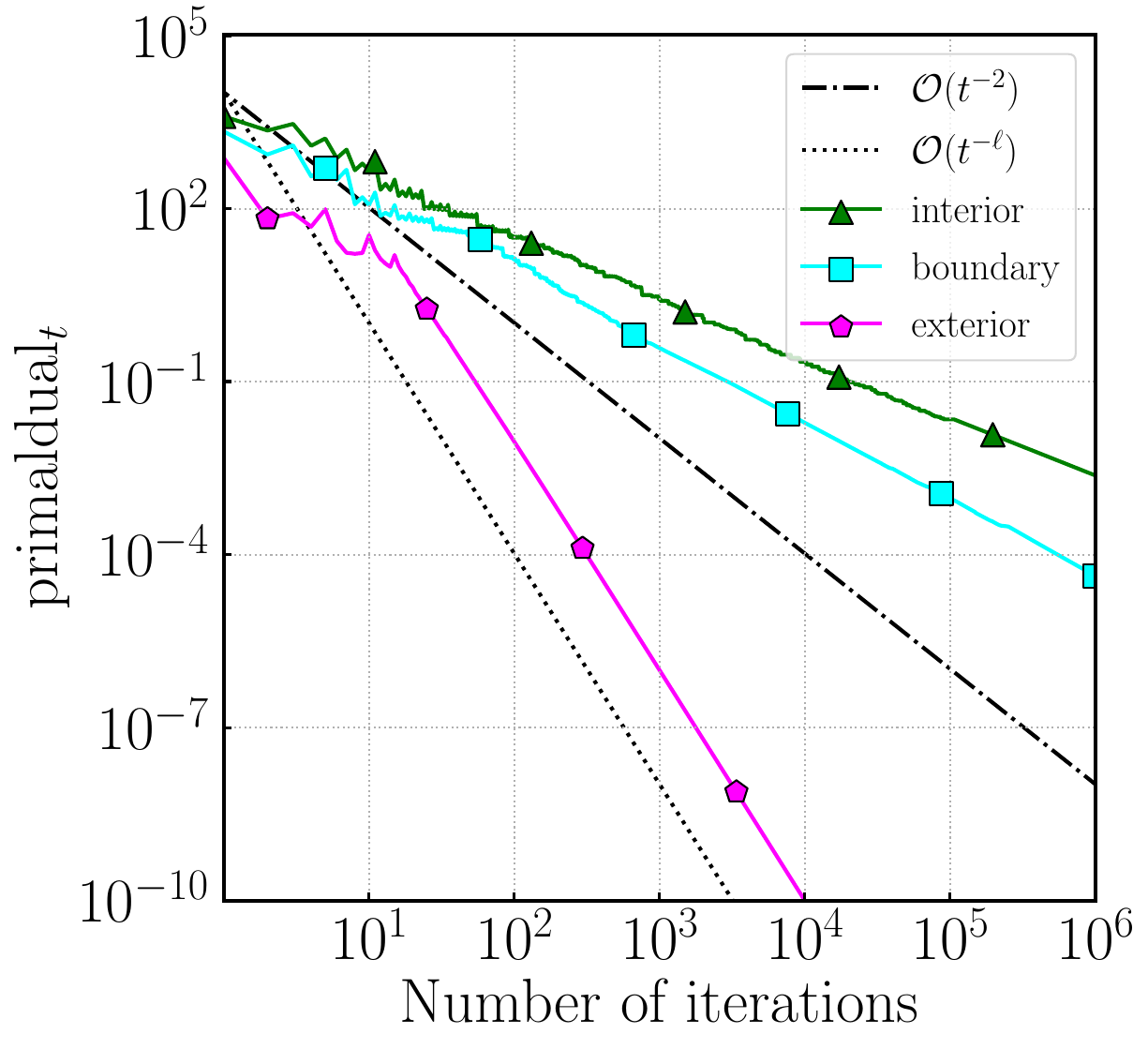}
        \caption{$\ell_5$-ball, $\primaldual_t$.}\label{fig:reg_5_primaldual}
    \end{subfigure}& 
    \begin{subfigure}{.31\textwidth}
    \centering
        \includegraphics[width=1\textwidth]{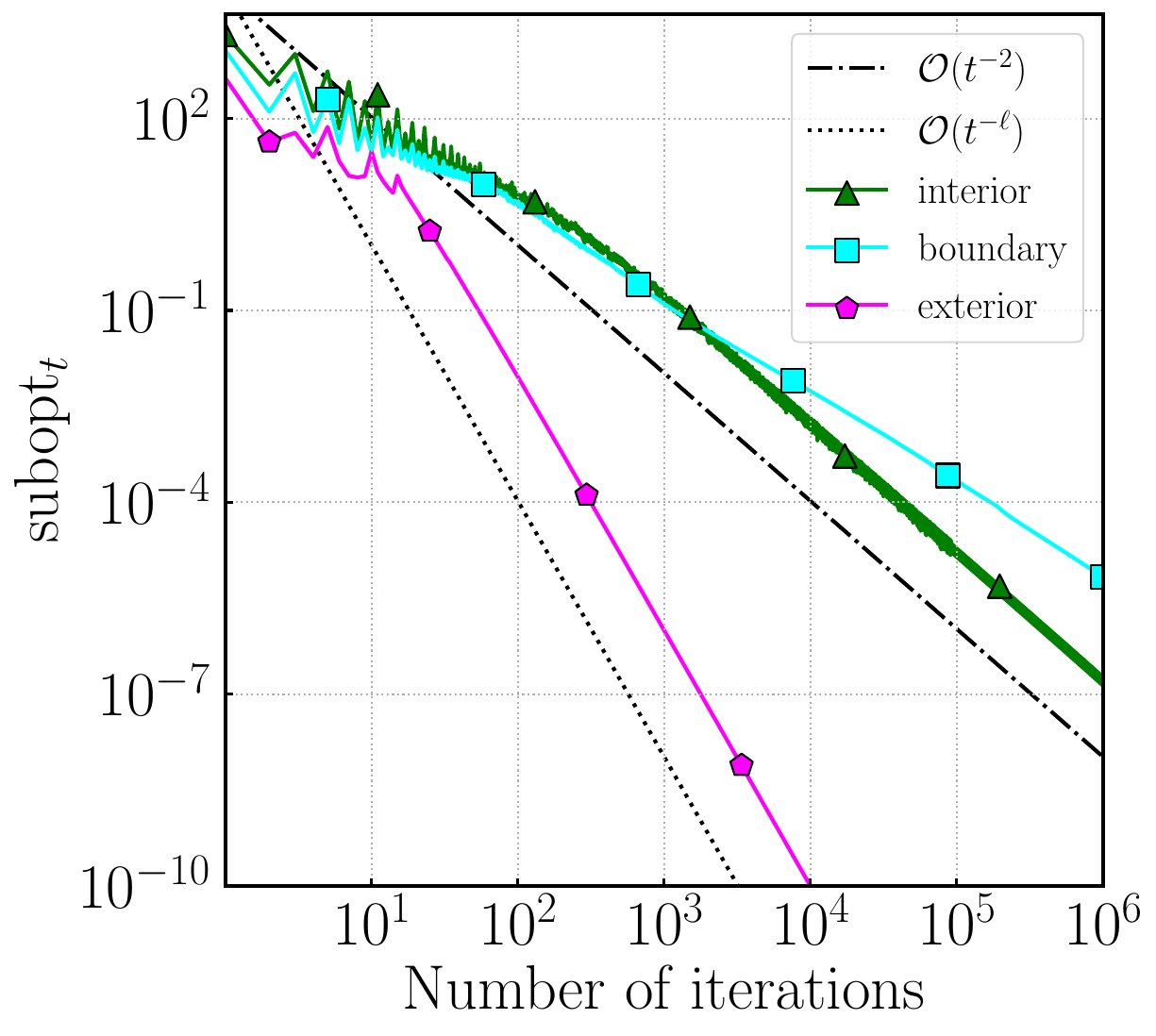}
        \caption{$\ell_5$-ball, $\subopt_t$.}\label{fig:reg_5_subopt}
    \end{subfigure}
\end{tabular}
\caption{\textbf{Constrained regression.} 
Convergence rate comparison of \fw{} with step-size $\eta_t = \frac{\ell}{t+\ell}$ for $\ell = 4$ and all $t\in\N$ applied to \eqref{eq:regression} for different locations of the unconstrained optimizer in the relative interior, on the relative boundary, and in the relative exterior of the feasible region. The comparison is performed for different $\ell_p$-balls with $p\in\{1, 2, 5\}$ and the three different optimality measures $\gap_t$, $\primaldual_t$, and $\subopt_t$. Axes are in log scale. 
}\label{fig:reg}
\end{figure}
In this section, we use \fw{} with open-loop step-size $\eta_t = \frac{\ell}{t+\ell}$ for $\ell = 4$ and all $t\in\N$ for constrained regression problems, that is, to solve
\begin{align}\label{eq:regression}
    \min_{\xx\in\R^n, \|\xx\|_p \leq \beta}&\; \frac{1}{2}\|A \xx - \yy\|_2^2,
\end{align}
where $A\in\R^{m \times n}$, $\yy\in\R^m$, $p\in\R_{\geq 1}$, and $\beta > 0$. Denote the unconstrained optimizer by $\xx_{\text{unc}}:= \argmin_{\xx\in\R^n} \frac{1}{2}\|A \xx - \yy\|_2^2$. In Figure~\ref{fig:reg}, for \fw{} with open-loop step-size $\eta_t = \frac{\ell}{t+\ell}$ for $\ell = 4$ and all $t\in\N$, we compare the convergence rate for different values of $\beta \in\|\xx_{\text{unc}}\|_p\cdot \{3/2, 1, 1/2\}$ corresponding to the unconstrained optimizer lying in the relative interior, on the relative boundary, and in the relative exterior of the feasible region, respectively. The comparison is performed for different $\ell_p$-balls with $p\in\{1, 2, 5\}$ and the three different optimality measures $\gap_t$, $\primaldual_t$, and $\subopt_t$. Matrix $A$ and vector $\yy$ are obtained by Z-score normalizing the data of the Boston-housing dataset\footnote{Obtained via the \textsc{sklearn} package, see, e.g., \href{https://scikit-learn.org/0.15/index.html}{https://scikit-learn.org/0.15/index.html}.} with $m = 506$ and $n = 13$. We also plot $\cO(t^{-\ell})$ and $\cO(t^{-2})$ for better visualization.

In Figures~\ref{fig:reg_1_subopt}, \ref{fig:reg_2_subopt}, and~\ref{fig:reg_5_subopt}, we observe that $\subopt_t$ always converges at the accelerated rates predicted by the theoretical results in this paper. 
In some of the remaining plots, we also observe accelerated rates for $\gap_t$ and $\primaldual_t$.

\subsection{Sparsity-constrained logistic regression}\label{sparse.logistic.regression}
\begin{figure}[h]
\centering
\begin{minipage}{.31\textwidth}
  \centering
  \includegraphics[width=\linewidth]{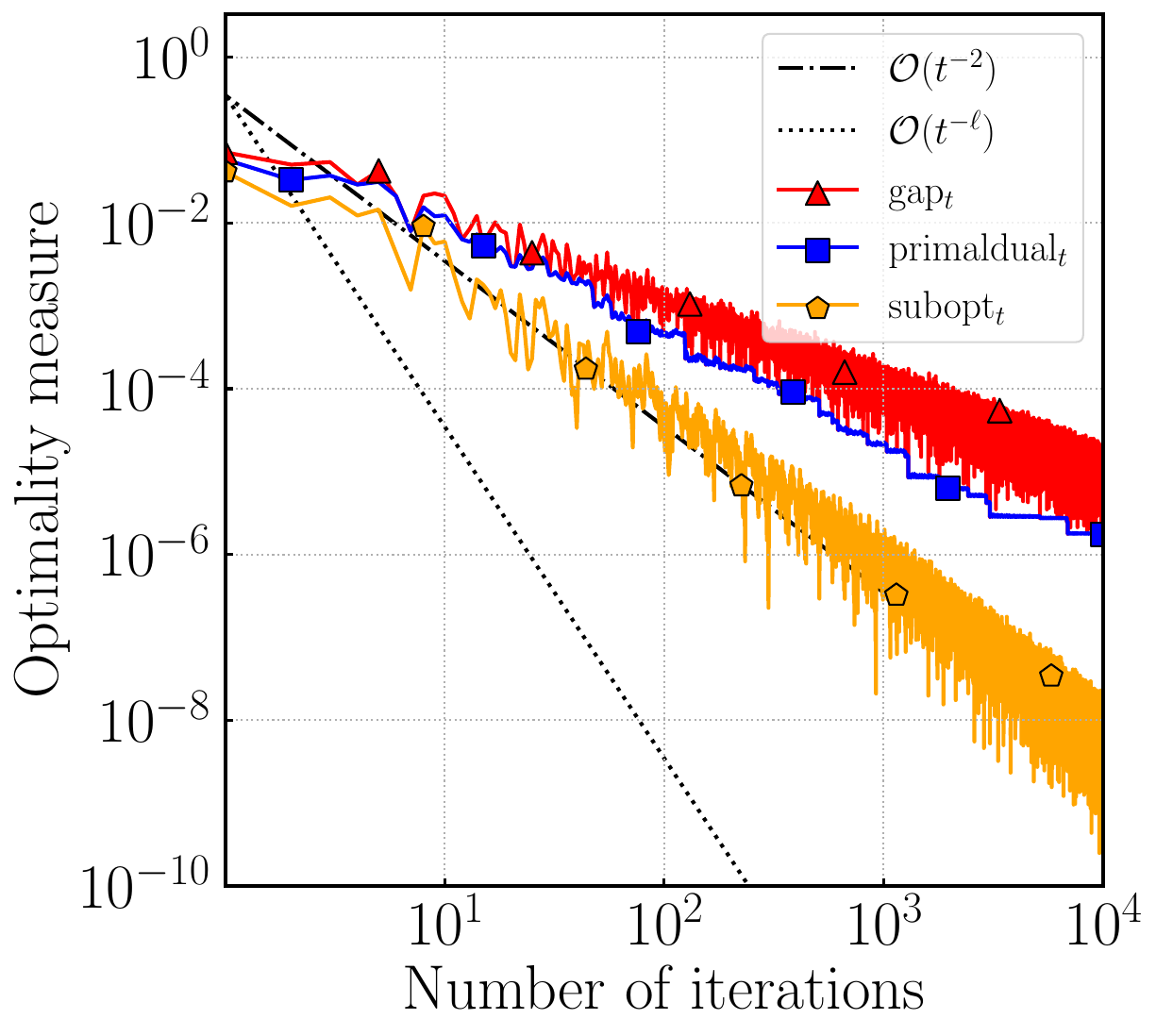}
\end{minipage}%
\caption{\textbf{Sparsity-constrained logistic regression.} Optimality measure comparison of \fw{} with step-size $\eta_t = \frac{\ell}{t+\ell}$ for $\ell = 4$ and all $t\in\N$ applied to \eqref{eq:log_reg}. Axes are in log scale.}
\label{fig:sparse_log_reg}
\end{figure}
In this section, we consider $\ell_1$-constrained logistic regression, which has the problem formulation 
\begin{align}\label{eq:log_reg}
\min_{\xx\in\R^n, \|\xx\|_1\leq 1}&\; \frac{1}{m}\sum_{i=1}^m \log(1+\exp(-b_i \aa_i^\intercal \xx)),
\end{align}
where $\aa_1,\ldots, \aa_m \in\R^d$ are feature vectors and $b_1, \ldots, b_m \in\{-1,+1\}$ are the corresponding labels.
The $\ell_1$-constraints induce sparsity in the iterates of \fw{}. 
In Figure~\ref{fig:sparse_log_reg}, for \fw{} with open-loop step-size $\eta_t = \frac{\ell}{t+\ell}$ for $\ell = 4$ and all $t\in\N$, we compare $\gap_t$, $\primaldual_t$, and $\subopt_t$ on the Z-score normalized Gisette dataset\footnote{Available online at \href{https://archive.ics.uci.edu/ml/datasets/Gisette}{https://archive.ics.uci.edu/ml/datasets/Gisette}.} \cite{guyon2003introduction} with $m=2000$ and $n=5000$. We also plot $\cO(t^{-\ell})$ and $\cO(t^{-2})$ for better visualization. We cannot plot $\fwt$ as in Theorem~\ref{thm:rate-gaps-relaxed} predicting when acceleration kicks in as we do not know the value of $\Afwt$ for which Assumption~\ref{assum.active} is satisfied.

In Figure~\ref{fig:sparse_log_reg}, we observe that $\gap_t = \cO(t^{-1})$, $\primaldual_t = \cO(t^{-1})$, and $\subopt_t = \cO(t^{-2})$, where the latter rate is captured by Theorem~\ref{thm:rate-gaps-relaxed}.

\subsection{Collaborative filtering}\label{collaborative.filtering}
\begin{figure}[ht]
\captionsetup[subfigure]{justification=centering}
\begin{tabular}{c c c}
    \begin{subfigure}{.31\textwidth}
    \centering
        \includegraphics[width=1\textwidth]{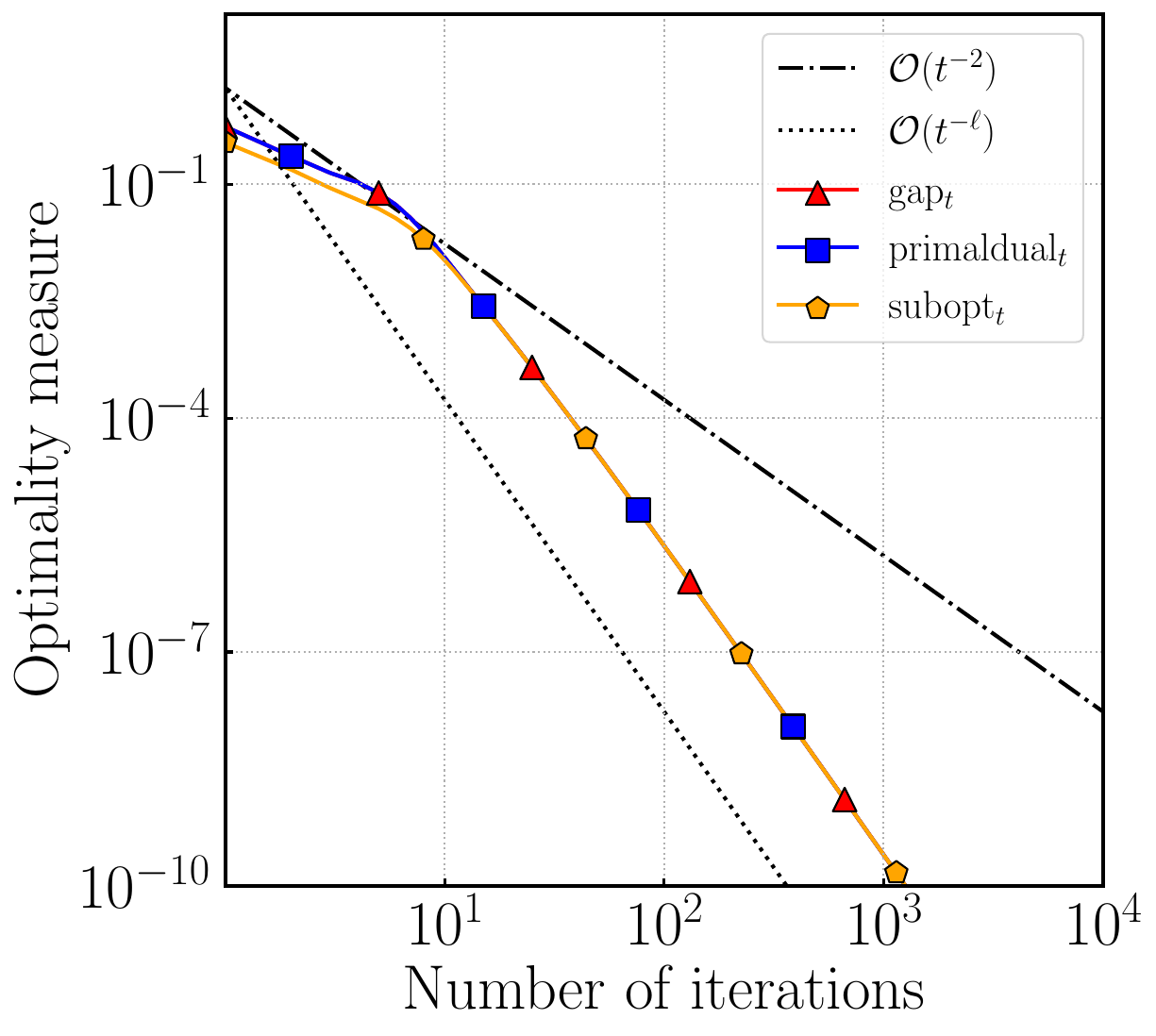}
        \caption{$\beta = 1000$.}\label{fig:col_fil_1000}
    \end{subfigure}& 
    \begin{subfigure}{.31\textwidth}
    \centering
        \includegraphics[width=1\textwidth]{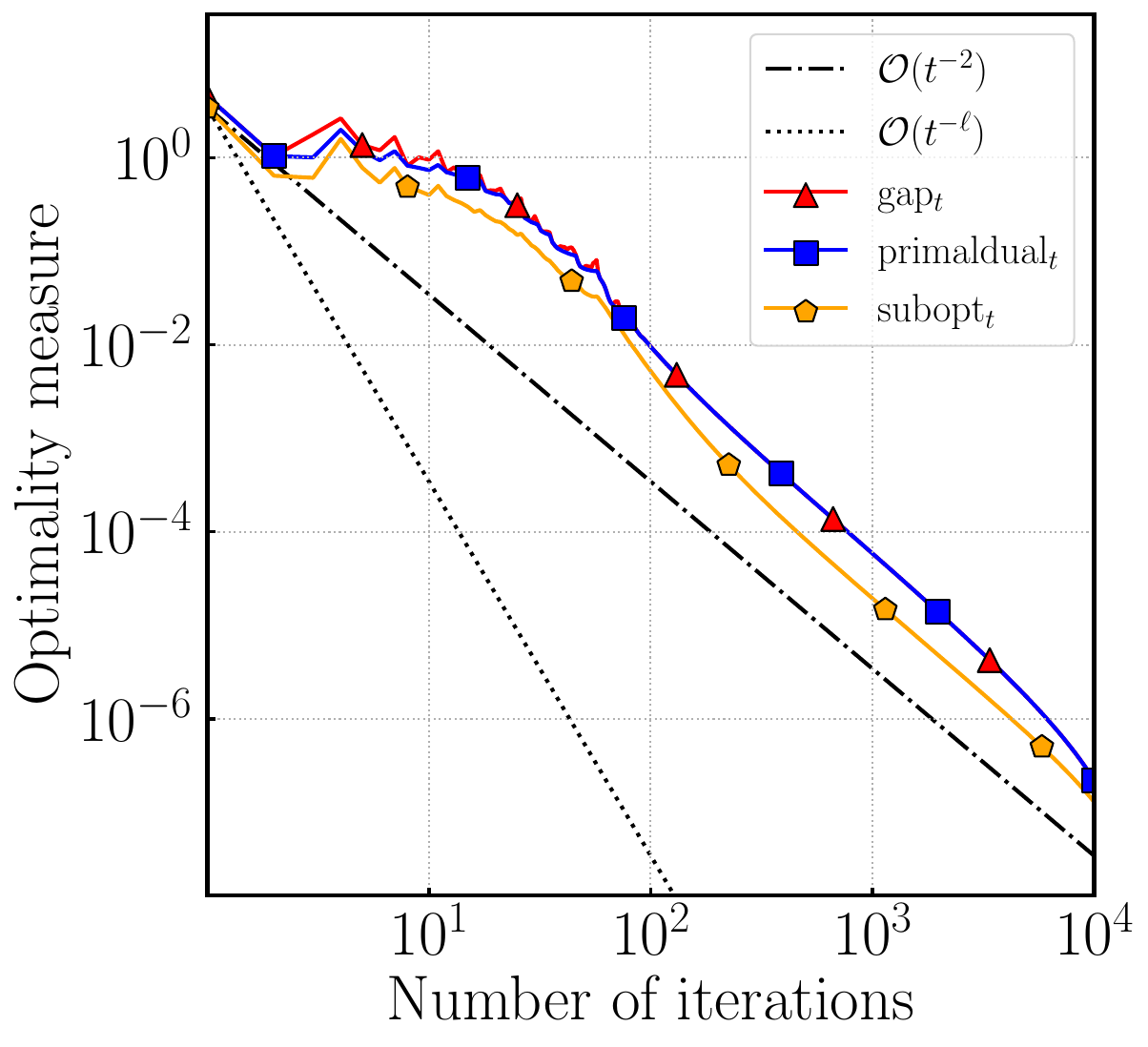}
        \caption{$\beta = 2000$.}\label{fig:col_fil_2000}
    \end{subfigure}& 
    
    \begin{subfigure}{.31\textwidth}
    \centering
        \includegraphics[width=1\textwidth]{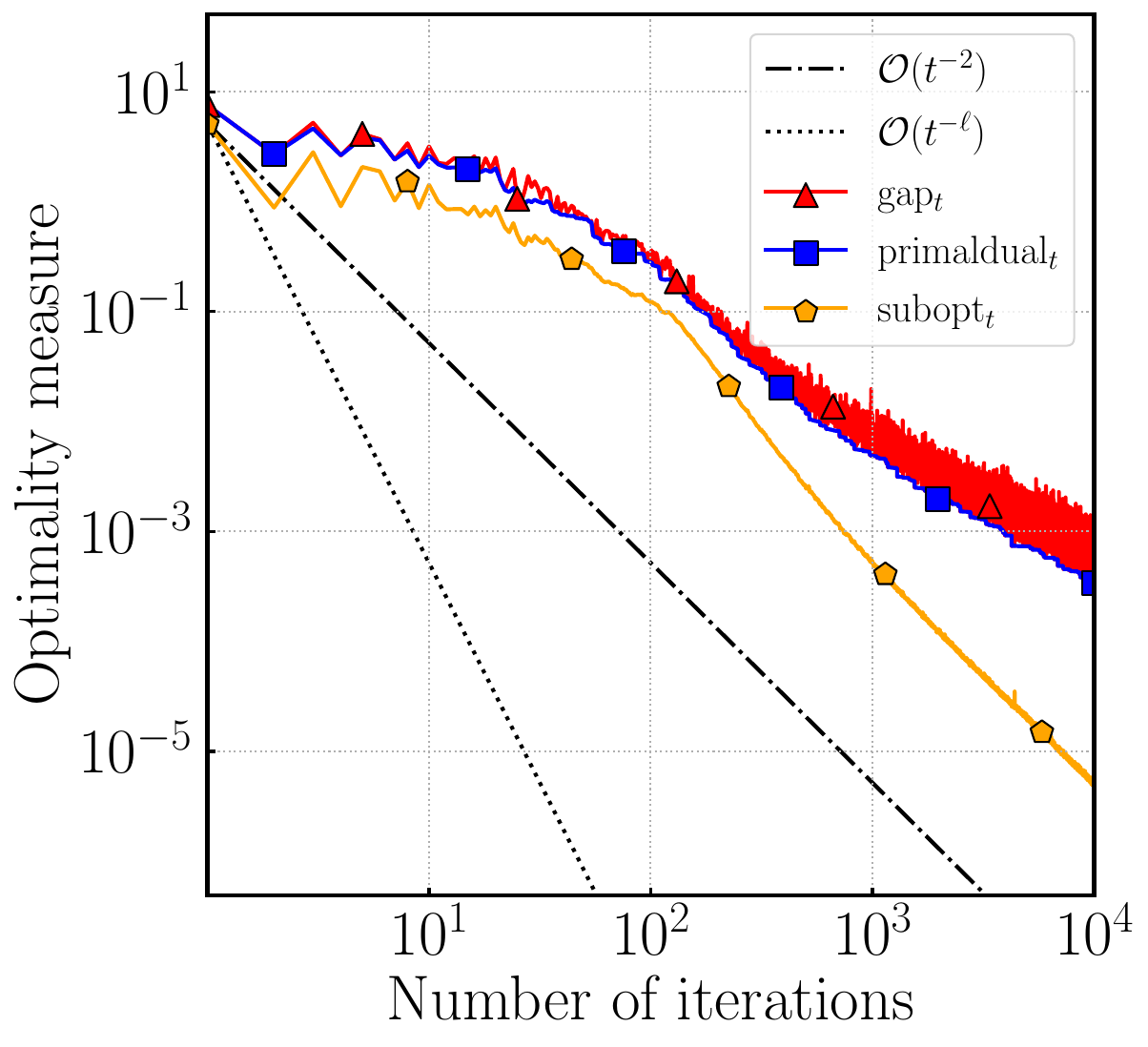}
        \caption{$\beta=3000$.}\label{fig:col_fil_3000}
    \end{subfigure}
\end{tabular}
\caption{\textbf{Collaborative filtering.} Optimality measure comparison of \fw{} with step-size $\eta_t = \frac{\ell}{t+\ell}$ for $\ell = 4$ and all $t\in\N$ applied to \eqref{eq:collaborative_filterting} for different radii $\beta\in\{2000, 3000, 4000\}$ of the nuclear norm ball. Axes are in log scale.
}\label{fig:col_fil}
\end{figure}
In this section, we consider the problem of collaborative filtering, which has the problem formulation by \cite{mehta2007robust}
\begin{align}\label{eq:collaborative_filterting}
    \min_{X\in\R^{m\times n}, \|X\|_{\nuc}\leq \beta} &\; \frac{1}{|\cI|} \sum_{(i,j)\in\cI} H_\rho(A_{i,j} - X_{i,j})
\end{align}
where $H_\rho$ is the Huber loss with parameter $\rho > 0$ \cite{huber1992robust}:
\begin{align*}
    H_\rho\colon x\in \R \mapsto  \begin{cases}
    \frac{x^2}{2}, & \text{if} \ |x| \leq \rho\\
    \rho(|x| - \frac{\rho}{2}), & \text{if} \ |x| > \rho,
    \end{cases}
\end{align*}
$A\in\R^{m\times n}$ is a matrix with partially observed entries, that is, there exists an index set $\cI \subseteq \{1,\ldots,m\}\times\{1,\ldots,n\}$ such that only $A_{i,j}$ with $(i,j)\in\cI$ are observed, 
$\|\cdot\|_{\nuc}\colon X\in\R^{m\times n} \mapsto \trace(\sqrt{X^\intercal X})$ is the nuclear norm, and $\beta>0$.

In Figure~\ref{fig:col_fil}, for \fw{} with open-loop step-size $\eta_t = \frac{\ell}{t+\ell}$ for $\ell = 4$ and all $t\in\N$, we compare $\gap_t$, $\primaldual_t$, and $\subopt_t$ on the movielens dataset\footnote{Available online at \href{https://grouplens.org/datasets/movielens/100k/}{https://grouplens.org/datasets/movielens/100k/}.} \cite{harper2015movielens} with $m=943$, $n=1682$, $\rho=1$, and $\beta \in\{1000, 2000, 3000\}$. We also plot $\cO(t^{-\ell})$ and $\cO(t^{-2})$ for better visualization. Since the nuclear norm ball is neither uniformly convex nor a polytope, our analysis does not cover the observed acceleration. 

In Figure~\ref{fig:col_fil_1000}, all optimality measures decay at a rate of $\cO(t^{-\ell})$, in Figure~\ref{fig:col_fil_2000}, all optimality measures decay at a rate of order $\cO(t^{-2})$, and in Figure~\ref{fig:col_fil_3000}, $\gap_t$ and $\primaldual_t$ appear to decay at a slower rate than $\subopt_t$, which again appears to converge at a rate of order $\cO(t^{-\ell})$.
To the best of our knowledge, none of these accelerated rates are explained by the current literature. This raises the open question to characterize the acceleration \fw{} enjoys when optimizing over the nuclear norm ball by, for example, exploiting the findings from \cite{garber2021convergence}.
\section{Discussion}\label{sec.discussion}

Throughout, we derived affine-invariant accelerated convergence rates for \fw{} with open-loop step-sizes of the form $\eta_t = \frac{\ell}{t+\ell}$, $\ell\in\N_{\geq 1}$. Notably, our results subsume the ones in \cite{wirth2023acceleration} up to constant factors. In particular, when the feasible region is strongly convex and the gradient of the objective is bounded away from zero, we confirm the conjecture of \cite{wirth2023acceleration} and extend prior rates of order $\cO(t^{-\ell/2})$ to $\cO(t^{-\ell})$. Furthermore, we establish the first affine-invariant rates of order $\cO(t^{-2})$ in the setting of Wolfe's lower bound, which demonstrates that \fw{} with line-search and short-step can be outperformed by \fw{} with open-loop step-sizes in one of the most important settings of optimization.

As discussed in \cite{pokutta2023frankwolfe}, in the future we intend to study open-loop step-sizes where we let $\ell$ increase gradually. That is, we are going to study step-sizes of the form $\eta_t = \frac{g(t)}{t+g(t)}$ for $g\colon\N \to [0, 1]$, where $g$ is monotone. Preliminary calculations and experiments suggest that \fw{} with $\eta_t = \frac{g(t)}{t+g(t)}$ and $g(t) = 2 + \log(t+1)$ admits arbitrarily fast yet sublinear rates in the strong $(M, 1)$-growth setting and rates of order $\tilde \cO(t^{-\frac{1}{1-r}})$ in the strong $(M, r)$-growth setting where the $\tilde \cO(\cdot)$ expression hides some polylogarithmic factors.

\section*{Declarations}

\subsection*{Funding}
Research reported in this paper was partially supported by the Deutsche Forschungsgemeinschaft (DFG, German Research Foundation) under Germany's Excellence Strategy – The Berlin Mathematics Research Center MATH$^+$ (EXC-2046/1, project ID 390685689, BMS Stipend), and by the Bajaj Family Chair at the Tepper School of Business, Carnegie Mellon University.

\subsection*{Availability of code and data}

The Python code for the experiments reported in this article are publicly available from \href{https://github.com/ZIB-IOL/affine_invariant_open_loop_fw}{GitHub}.

The Boston-housing dataset used in Section~\ref{sec:constrained regression} is publicly available via the open-source \textsc{sklearn} Python package \href{https://scikit-learn.org/0.15/modules/generated/sklearn.datasets.load_boston.html}{https://scikit-learn.org/0.15/index.html}.

The Gisette dataset used in the numerical experiments in Section~\ref{sparse.logistic.regression} is publicly available at the University of California at Irvine Machine Learning Repository
\href{https://archive.ics.uci.edu/ml/datasets/Gisette}{https://archive.ics.uci.edu/ml/datasets/Gisette}.

The movielens dataset used in the numerical experiments in Section~\ref{collaborative.filtering} is publicly available at the GroupLens research lab in the Department of Computer Science and Engineering at the University of Minnesota Twin Cities \href{https://grouplens.org/datasets/movielens/100k/}{https://grouplens.org/datasets/movielens/100k/}.


\bibliographystyle{apalike}
\bibliography{bibliography}

\appendix

\section{Proof of Lemma~\ref{lemma:telescope}}\label{app:proof_lemma_telescope}
\begin{proof}[Proof of Lemma~\ref{lemma:telescope}]
    It holds that
    \begin{align*}
        &\prod_{i=\fwt}^{t} \left(1-\left(1-\frac{\epsilon}{\ell}\right) \eta_i\right)  = \prod_{i=\fwt}^{t} \frac{i+\epsilon}{i+\ell} & \text{$\triangleright$ since $\eta_i=\frac{\ell}{i+\ell}$}\nonumber\\
        & = \left(\prod_{i=\fwt}^{t} \frac{i}{i+\ell} \right)\left(\prod_{i=\fwt}^{t}\frac{i+\epsilon}{i}\right) & \text{$\triangleright$ since $\fwt \in\N_{\geq 1}$}\\
        & \le \left(\frac{\fwt -1+ \ell}{t+\ell}\right)^\ell \prod_{j=\fwt}^{t}\left(1+\frac{\epsilon}{i}\right)  & \text{$\triangleright$ by Lemma~\ref{lemma:telescope.simple}}\\
        & \leq \left(\frac{\fwt -1+ \ell}{t+\ell}\right)^\ell\exp\left(\epsilon \sum_{i=\fwt}^{t}\frac{1}{i}\right)  & \text{$\triangleright$ for $x\geq 0$, we have  $1+x\leq\exp(x)$}\nonumber\\
        & \leq \left(\frac{\fwt -1+ \ell}{t+\ell}\right)^\ell\exp\left(\epsilon\left(\frac{\ell}{\fwt} + \sum_{i=\fwt+\ell}^{t+\ell}\frac{1}{i} \right)\right)   \nonumber\\
        & \leq \left(\frac{\fwt -1+ \ell}{t+\ell}\right)^\ell\exp\left(\frac{\epsilon\ell}{\fwt}\right)\cdot \exp\left(\epsilon\int_{\fwt+\ell-1}^{t+\ell}\frac{1}{x}dx\right)& \text{$\triangleright$ since $\frac{1}{i}$ is monotonously decreasing}\nonumber\\     
        & = \left(\frac{\fwt -1+ \ell}{t+\ell}\right)^\ell\exp\left(\frac{\epsilon\ell}{\fwt}\right)\cdot \left(\frac{t+\ell}{\fwt -1+ \ell}\right)^{\epsilon}\nonumber\\
        & \leq \left(\frac{\fwt -1+ \ell}{t+\ell}\right)^{\ell-\epsilon} \exp\left(\frac{\epsilon\ell}{\fwt}\right)\nonumber\\
        & = \left(\frac{\eta_t}{\eta_{\fwt-1}}\right)^{\ell-\epsilon} \exp\left(\frac{\epsilon\ell}{\fwt}\right).
    \end{align*}
\end{proof}

\end{document}